\documentclass[11pt]{amsart}
\usepackage{amsmath, amsthm, amssymb}
\usepackage{amsmath,amscd}
\usepackage{mathabx}
\usepackage{hyperref}
\usepackage{mathrsfs}
\usepackage{xcolor, changepage} 
\usepackage{graphicx}
\usepackage{titletoc}
\usepackage{enumerate}
\usepackage{accents}
\usepackage{esint}

\usepackage{tikz}
\usetikzlibrary{matrix,arrows}
\usetikzlibrary{shapes}
\usetikzlibrary{calc}
\usetikzlibrary{arrows}
\usetikzlibrary{decorations.pathreplacing,decorations.markings}
\usepackage[all]{xy}
\usepackage{tikz-cd}

\usepackage{caption}
\usepackage{subcaption}
\usepackage{geometry}
\theoremstyle{plain}
\newtheorem{theorem}{Theorem}
\newtheorem{proposition}[theorem]{Proposition} 
\newtheorem{lemma}[theorem]{Lemma}
\newtheorem{remark}[theorem]{Remark}
\newtheorem{corollary}[theorem]{Corollary}
\newtheorem{definition}[theorem]{Definition}

\numberwithin{theorem}{section}

\DeclareMathOperator{\Lip}{Lip}

\DeclareMathOperator{\diam}{diam}
\DeclareMathOperator{\supp}{supp}
\DeclareMathOperator{\degg}{deg}
\DeclareMathOperator{\width}{width}

\DeclareMathOperator{\divv}{div}
\DeclareMathOperator{\spt}{spt}

\DeclareMathOperator{\inj}{inj}
\DeclareMathOperator{\capp}{cap}

\numberwithin{equation}{section}

\title[Singulararity removal rigidity theorems]{Singularity removal rigidity theorems for minimal hypersurfaces in  manifolds with nonnegative scalar curvature}
\author{Shihang He}
\address{Key Laboratory of Pure and Applied Mathematics,
	School of Mathematical Sciences, Peking University, Beijing, 100871, P. R. China
}
\email{hsh0119@pku.edu.cn}
\author{Yuguang Shi}
\address{Key Laboratory of Pure and Applied Mathematics,
	School of Mathematical Sciences, Peking University, Beijing, 100871, P. R. China
}
\email{ygshi@math.pku.edu.cn}

\author{Haobin Yu}
\address{
	School of Mathematics, Hangzhou Normal University, Hangzhou, 311121, P. R. China
}
\email{yhbmath@hznu.edu.cn}
\thanks{S. He, Y. Shi  are funded by the National Key R\&D Program of China Grant 2020YFA0712800 and NSFC12431003. H. Yu is funded by Zhejiang Provincial NSFC No. LY24A010008}


\begin{document}
\begin{abstract}
        We prove two "Singularity removal rigidity theorems" for minimal hypersurfaces with isolated singularities in manifolds of nonnegative scalar curvature (Theorems \ref{thm: rigidity for minimal surface} and \ref{thm: georch free of singularity}). In particular, we observe a new phenomenon that the extremal scalar curvature condition forces smoothness, which reveals a kind of positive effect of minimal hypersurface singularities in scalar curvature geometry. 
        
        As an application, we obtain a direct proof of the positive mass theorem (PMT) for asymptotically flat $8$-manifolds with arbitrary ends (Theorem \ref{thm: pmt8dim}), without using N. Smale's generic regularity theorem. A key ingredient is a new spectral version of PMT for AF manifolds with arbitrary ends, whose proof relies on PMT for asymptotically locally flat (ALF) manifolds with $\mathbf{S}^1$-symmetry.
	\end{abstract}
	\maketitle
	\tableofcontents	
	
	\section{Introduction}

	\subsection{ Singularity removal rigidity theorems and Geroch's Conjecture in Dimensions Not Exceeding  $8$}
	
    $\quad$

Rigidity phenomena under scalar curvature lower bounds have been studied extensively in Riemannian geometry. A basic example is the rigidity statement in the proof of the positive mass theorem \cite{SY79} by Schoen and Yau, where flatness follows in the extremal case when the mass attains zero. Since then, many rigidity results have been obtained under nonnegative scalar curvature assumptions.

A central tool in these works is the theory of stable minimal hypersurfaces, where the stability inequality links the ambient scalar curvature to the intrinsic geometry of the hypersurface. The study of the rigidity for minimal hypersurfaces under nonnegative scalar curvature assumption dates back to \cite{FcS1980}. Later, \cite{CG00} and \cite{CEM19} studied rigidity of area-minimizing  tori and cylinders in 3-manifolds with nonnegative scalar curvature. \cite{BBN10} also studied the rigidity problem concerning area-minimizing spheres in 3-manifolds with positive scalar curvature, subjected to a systolic-extremal condition. More recently, \cite{Carlotto16,CCE16} and \cite{EK23} provided a thorough study of stable minimal hypersurfaces in asymptotically flat manifolds with nonnegative scalar curvature.

The previous results concern minimal hypersurfaces in manifolds of dimensions at most seven, where minimal hypersurfaces are smooth by regularity theory. In dimension eight and higher, minimal hypersurfaces may develop singularities. Therefore, when studying rigidity in these dimensions, one must allow minimal hypersurfaces with singular sets, which substantially increases the complexity of the problem, and requires a refined understanding on the interaction between minimal hypersurface singularities and the scalar curvature. Theorem \ref{thm: rigidity for minimal surface} continues the investigation of the above problem in dimension eight. Although singularities are allowed a priori, we show that in the scalar-extremal setting of Theorem \ref{thm: rigidity for minimal surface}, singularities cannot occur. We interprete it as a "Singularity removal rigidity theorem". In this sense, \textit{the scalar curvature condition enforces regularity in addition to geometric rigidity.}
	\begin{theorem}\label{thm: rigidity for minimal surface}
		Let $(M^{n+1},g)$ be an AF manifold with arbitrary ends, an AF end  $E$ of asymptotic order  $\tau>n-2$. Let $\Sigma^{n}\subset M^{n+1}$ be an area-minimizing boundary in $M^{n+1}$ with isolated singular set $\mathcal{S}$. Assume $\Sigma$ is  strongly stable (see Definition \ref{defn: strongly stable hypersurface} below), $R_g\ge 0$ along $\Sigma$, and one of the following holds:
		
		(a) $n+1\le 8$
		
		(b) $\Sigma\backslash\mathcal{S}$ is spin.
        
		$\quad$
		
		(1) If $\Sigma\backslash E$ is compact, then we have $\mathcal{S} = \emptyset$, and $\Sigma$ is isometric to $\mathbf{R}^{n}$, with $R_g = |A|^2 = Ric(\nu,\nu) = 0$ along $\Sigma$.
		
		(2) If there exists $U_1,U_2\subset M$ with $E\subset U_1\subset U_2$, such that $U_i\backslash E$ is compact ($i = 1,2$) and  $R_g> 0$ on $U_2\backslash U_1$, then $\Sigma\subset U_1$, and the same conclusion in (1) holds. 
	\end{theorem}

	 In his Four Lectures \cite[Section 3.7.1]{Gro23}, Gromov remarked that \textit{singularities must enhance the power of minimal hypersurfaces and stable $\mu$-bubbles, since the large curvatures add to the positivity of the second variation.} Gromov's observation suggests a possibility to use the minimal hypersurface method to study scalar curvature problems in higher dimensions without any smoothing procedure. Theorem \ref{thm: rigidity for minimal surface} and  Theorem \ref{thm: georch free of singularity} below  consolidate this Gromov's observation.
	
	 As a direct corollary of Theorem \ref{thm: rigidity for minimal surface}, we obtain the following positive mass theorem for AF manifolds $(M^{n+1},g)$ with arbitrary ends, where $n\leq 7$.

	\begin{theorem}\label{thm: pmt8dim}
		Let $(M^{n+1}, g)$ be a smooth AF manifold  with  arbitrary ends and non-negative scalar curvature. If $2\leq n\leq 7$, then the mass of $(M^n, g)$ is non-negative, the equality holds if and only if 	$(M^{n+1}, g)$ is isometric to $\mathbf{R}^{n+1}$.
	\end{theorem}

 In \cite{LUY21}, Lesourd-Unger-Yau established a shielded version of the positive mass theorem. The second statement of our Theorem \ref{thm: rigidity for minimal surface} is largely influenced by their work, and we refer to it as the \textit{shielding principle} for strongly stable area-minimizing hypersurfaces in AF manifolds with arbitrary ends. One of the novelty in the proof of Theorem \ref{thm: pmt8dim} lies in that it can be obtained independent of the $\mu$-bubble approach. On the other hand, manifolds with arbitrary ends arise naturally as the blow up models for singular spaces. A spectral version of the positive mass theorem for AF manifolds with arbitrary ends, Theorem \ref{thm: PMT arbitrary end spectral}, will play a central role in our establishment of the main results of this paper. This is our another motivation for considering manifolds with arbitrary ends.
	
	 Compared with
    Theorem \ref{thm: rigidity for minimal surface}, in the closed manifold setting, we obtain another "Singularity removal rigidity theorem".
		\begin{theorem}\label{thm: georch free of singularity}
			Let $(M^{n+1},g)$ be a closed Riemannian manifold and let $\Sigma^n\subset M^{n+1}$ be a  compact area-minimizing hypersurface with isolated singular set $\mathcal{S}$. Assume that there exists a non-zero degree map from $\Sigma$ to an enlargeable manifold $X$ in the sense of Definition \ref{defn: degree 2}.
			
			If $R_g\ge 0$ along $\Sigma$ and one of the following holds:
			
			\begin{enumerate}
				\item $n+1 = 8$;
				\item $\Sigma\backslash\mathcal{S}$ is spin.
			\end{enumerate}
			Then $\mathcal{S} = \emptyset$, $\Sigma$ is intrinsically flat, and $R_g = Ric(\nu,\nu) = |A|^2_\Sigma = 0$ along $\Sigma$.
	\end{theorem}

	As a direct corollary of Theorem \ref{thm: georch free of singularity}, we have the following topological obstruction result for PSC metric without using N.Smale's regularity theorem for minimal hypersurfaces in  compact $8$-dimensional  manifolds with generic metric  \cite{Smale93}.

	\begin{corollary}\label{cor:8dim georch}
		Let $M^n$ $(n\le 8)$ be a compact manifold and assume there exists a non-zero degree map $f: M^n\longrightarrow \mathbf{T}^n$, then $M^n$ admits no PSC metric.
	\end{corollary}

\subsection{Positive mass theorem on AF manifolds with arbitrary ends and nonnegative scalar curvature in the spectral sense.}	

$\quad$

	Recently, the notion of scalar curvature lower bound in the {\it spectral sense}  has been gaining increasing importance (see \cite{Gro20},\cite{CL23},\cite{CL2024},\cite{CL24b},\cite{LM2023},\cite{Gr2024} and references therein). Indeed, it is a natural generalization of the conception of non-negative scalar curvature, and is automatically satisfied by a stable minimal hypersurface in a Riemannian manifold with non-negative scalar curvature. For the case of asymptotically flat manifolds, it was remarked in \cite{ZZ00} that non-negative Dirichlet eigenvalue of the conformal Laplacian is not enough to deduce the non-negative mass. On the other hand, \cite{ZZ00} proved a positive mass theorem by assuming the Neumann eigenvalue for the conformal Laplacian to be non-negative on exhaustions of the AF end. 
    
    In this paper, we are also interested in the positive mass theorem for AF manifolds with non-negative scalar curvature in the {\it strong spectral sense} (See Theorem \ref{thm: PMT arbitrary end spectral} below), which naturally arises in area-minimizing hypersurfaces stable under asymptotically constant variations in AF manifolds. The positive mass theorem of this type dates back to Schoen's dimension reduction proof of the positive mass theorem for AF manifolds of dimension no greater than $7$ in \cite{Schoen1989}, see also \cite{Carlotto16},\cite{EK23},\cite{HSY24} for some further applications. To begin with, we introduce the following definitions:

	\begin{definition}\label{defn: strong test function}
		Let $(M^n,g)$ be  a manifold (not necessarily complete) containing an AF end $E$. We say that a locally Lipschitz function $\phi$ is an \textit{asymptotically constant test function} on $M$, if it is supported on a neighborhood $\mathcal{U}$ of $E$ such that $\mathcal{U}\Delta E$ is compact, satisfying $\lim\limits_{|x|\to\infty, x\in E}\phi = 1$ and $\phi-1\in W^{1,2}_{-q}(\mathcal{U})$, for some $q> \frac{n-2}{2}$.  i.e.
		$$
		\int_{\mathcal{U}}|\phi-1|^2 r^{2q-n}d\mu_g+\int_{\mathcal{U}}|\nabla\phi|^2 r^{2(q+1)-n}d\mu_g<\infty,
		$$
		where $r$ denotes a positive function on $\mathcal{U}$ and $r=|x|$ on $E$.
	\end{definition}
	
	\begin{definition}\label{defn: strong spectral}
		Let $(M^n,g)$ be  a manifold (not necessarily complete) containing an AF end $E$. Let $h\in L^1(E)$. We say that $(M^n,g)$ has $\beta$-scalar curvature no less than $h$ in the strong spectral sense, if for any locally Lipschitz function $\phi$ which is either compactly supported or asymptotically constant in the sense of Definition \ref{defn: strong test function}, it holds
		\begin{align}\label{eq: 20}
			\int_{M} |\nabla \phi|^2+\beta R_g\phi^2 d\mu_g \ge \int_M \beta h\phi^2d\mu_g\ \ \ \text{for some}\ \ \ \beta>0.
		\end{align}
		In particular, we say $(M^n,g)$ has $\beta$-scalar curvature non-negative in the strong spectral sense if $h=0$.
	\end{definition}
	\begin{remark}\label{re: NNSC}

$\quad$
    
    \begin{enumerate}
        \item By the definition, $(M^n,g)$ has $\beta'$-scalar curvature no less than $h$ in the strong spectral sense for any $\beta'\in(0,\beta)$ provided $(M^n,g)$ has $\beta$-scalar curvature no less than $h$ in the same  sense. Furthermore, according to our definition, the classical curvature condition $R_g\ge h$ corresponds to the case of having $\infty$-scalar curvature no less than $h$ in the strong spectral sense.
        \item In our application we always set $\beta = \frac{1}{2}$ and $h = R_g+|A|^2$. This aligns with the second variation formula for stable minimal hypersurfaces.
    \end{enumerate}
	\end{remark}
	
	The positive mass theorems on smooth AF manifolds  with non-negative scalar curvature  and with arbitrary ends  were obtained by \cite{CL2024} and \cite{LUY21}\cite{Zhu23} etc. The following positive mass theorem for smooth AF manifolds with arbitrary ends under the strong spectral non-negative scalar curvature condition,  which has its own interests, serves as a crucial ingredient in the proof of Theorem \ref{thm: rigidity for minimal surface}. This is because we will always regard the minimal hypersurface $\Sigma^n$ in Theorem \ref{thm: rigidity for minimal surface} as an AF space of $1$ dimension lower.
	
	\begin{theorem}\label{thm: PMT arbitrary end spectral}
		Let $(M^n,g)$ be an AF manifold with arbitrary ends and $E$  be its AF end 
		of order 
		$\frac{n-2}{2}<\tau\le n-2$ and $3\leq n\leq 7$. Suppose $(M^n, g)$ has $\beta$-scalar curvature no less than $h$ in the strong spectral sense for some $\beta>0$.
		\begin{enumerate}
			\item If $\beta\ge\frac{1}{2}$, $h\ge 0$ everywhere and $h(p)>0$ at some point $p$, then the  mass of $(M,g,E)$ is strictly positive.
			\item If $\beta\ge\frac{1}{2}$, $h\ge 0$ everywhere, then the  mass of $(M,g,E)$ is non-negative. Moreover, if $\beta>\frac{1}{2}$ and the  mass of $(M,g,E)$ is zero, then $(M,g)$ is isometric to $(\mathbf{R}^n,g_{Euc})$.
			\item If $\beta\ge\frac{n-2}{4(n-1)}$, $h\ge 0$ everywhere, and $M$ has only AF ends, then the  mass of $(M,g,E)$ is non-negative. Moreover, if the  mass of $(M,g,E)$ is zero, then $(M,g)$ is isometric to $(\mathbf{R}^n,g_{Euc})$.
		\end{enumerate}
	\end{theorem}
	
	We remark that for case (3) where $M$ has  only AF ends, the corresponding result was established in \cite{Schoen1989} and \cite{Carlotto16}, which we include here for comparison.
    The key advancement of Theorem \ref{thm: PMT arbitrary end spectral} lies in its allowance for manifolds with arbitrary ends, thereby extending its applicability to a broader class of geometric settings. In particular, it applies to spaces obtained by blowing up the singular set of a singular space.

	\subsection{Outline of Proofs}

$\quad$

\textbf{Blowing up the singular set}
    
	 We  sketch the proof of Theorem \ref{thm: rigidity for minimal surface}. To  prove the singular set $\mathcal{S}$ of $\Sigma$ must be empty, we adopt the contradiction argument. Suppose not, We will blow up the singular set $\mathcal{S}$ by suitable functions. The technique can be traced back to Schoen's work on the Yamabe problem \cite{Schoen1984} as well as \cite{LM2019}. By the strong stability of $\Sigma$, $\Sigma$ has $\frac{1}{2}$-scalar curvature no less than $h = R_g+|A|^2$ in the strong spectral sense. Thus, a natural candidate of the blow up function is the Green's function for the conformal Laplacian, which was used in \cite{Schoen1984}\cite{LM2019}, as it yields pointwise nonnegative scalar curvature. However, the scalar curvature of $\Sigma$ is not necessarily bounded near the singular set. It may blow up rapidly, making it hard to control the Green's function for the conformal Laplacian in our case.

     Instead of using the Green's function for the conformal Laplacian, we consider the ordinary singular harmonic function $G$ (which tends to $0$ as $x$ tends to the infinity of the AF end of $\Sigma$), making it easier to control its blow up rate near $\mathcal{S}$. Then we perform conformal deformations using the singular harmonic function $G_\delta = 1+\delta G$, obtaining a  complete AF manifold $(\Sigma\backslash\mathcal{S},G_\delta^{\frac{4}{n-2}}g_\Sigma)$ with arbitrary ends. This allows us to bypass the challenges involving hard analysis of the scalar curvature term near the singularity. The cost of this approach is that the blown up scalar curvature is not necessarily pointwise non-negative as in \cite{Schoen1984} and \cite{LM2019}. However, we note that the assumption that $\beta$-scalar curvature no less than $h$ in the strong spectral sense is preserved well under conformal deformations via singular harmonic functions (see Proposition \ref{conformal deformation1} ). To be more precise, $(\Sigma\backslash\mathcal{S},G_\delta^{\frac{4}{n-2}}g_\Sigma)$ has $\frac{1}{2}$-scalar curvature no less than $(R_g+|A|^2)G_\delta^{-\frac{4}{n-2}}$ in the strong spectral sense. The problem is now closely related to Theorem \ref{thm: PMT arbitrary end spectral}.

\textbf{Proof of Theorem \ref{thm: PMT arbitrary end spectral}}

	We briefly digress to discuss the proof of Theorem  \ref{thm: PMT arbitrary end spectral}. The case that $(M^n,g)$ has exactly one AF end follows from the classical conformal deformation argument as demonstrated in \cite{Schoen1989} and \cite{Carlotto16}. The conformal deformation makes the scalar curvature everywhere non-negative, allowing the application of the  classical positive mass theorem. However, $M$ necessarily has arbitrary ends when it is obtained from a singular space with singularity blown up. In this case, the completeness issue of the conformal deformed metric arises as a central problem. To overcome this difficulty, instead of performing the conformal deformation, we construct a $\mathbf{S}^1$-symmetric asymptotically locally flat (ALF) manifold $(\hat{M}^{n+1},\hat{g})$ by taking warped product with $\mathbf{S}^1$. The condition of $\beta$-scalar curvature no less than $h$ in the strong spectral sense for some  $\beta\ge \frac{1}{2}$ allows us to construct such a suitable warped-product function, and the resulting manifold is ALF with non-negative scalar curvature, allowing the application of the arguments in \cite{CLSZ2021}. 
	
	Here, we remark that the $\mathbf{S}^1$-symmetry of $\hat{M}^{n+1}$ is particularly important since the $\mathbf{S}^1$-symmetric minimal hypersurfaces or $\mu$-bubble can be chosen to be smooth provided the codimension of the $\mathbf{S}^1$-orbit is no greater than $7$ (See \cite[Proposition 3.4]{WY2023}), which is guaranteed by our assumption. Furthermore, we make a crucial observation that under the assumption of non-negative  $\beta$-scalar curvature in the strong spectral sense ($\beta\ge \frac{1}{2}$), the  mass is non-increasing in the process of constructing the warped product, illuminating similar phenomenon in case of performing conformal deformation in \cite{Schoen1989}, see Proposition \ref{lem: mass decay ALF} for detail. These new elements enable us to complete the proof of Theorem \ref{thm: PMT arbitrary end spectral}. 
    
     We further observe that the model space in Theorem \ref{thm: PMT arbitrary end spectral} not only serves as a blow up model for singular minimal hypersurfaces, but also naturally emerges when one attempts to prove the positive mass theorem with arbitrary ends through a direct application of Schoen-Yau's original method \cite{SY79}. Specifically, this involves constructing an area-minimizing hypersurface as the limit of a sequence of free-boundary minimizing hypersurfaces within coordinate cylinders. Thus, the method used in Theorem \ref{thm: PMT arbitrary end spectral} yields an alternative proof of the positive mass theorem with arbitrary ends independent of the $\mu$-bubble approach, see Remark \ref{remark: arbitrary ends} for details.

\textbf{Ruling out the singular set}
    
	Returning to the proof of Theorem \ref{thm: rigidity for minimal surface}, a notable distinction arises as the  mass may be increased slightly during the singular harmonic function blow-up process (It is roughly $O(\delta)$). However, we will demonstrate that this increment will be outweighed by the reduction in mass attributed to the $\mathbf{S}^1$-warped product as previously mentioned. This substantial reduction is essentially contributed by the fact $\mathcal{S}\ne\emptyset$ which forces the $|A|^2$ to be strictly positive somewhere near $\mathcal{S}$, providing a uniform positive term to the spectral scalar curvature lower bound $h = R_g+|A|^2$. See Proposition \ref{thm: negative mass 2} for detailed calculation. As a result, we obtain a negative mass ALF manifold $(\Sigma\backslash\mathcal{S})\times\mathbf{S}^1$ carrying a warped metric with nonnegative scalar curvature, which contradicts Proposition \ref{prop: PMT for S^1 symmetric ALF}.

 For the closed case, we employ similar metric deformation and warped product construction, which reduces the proof of Theorem \ref{thm: georch free of singularity} to the PSC topological obstruction result Proposition \ref{prop: noncompact dominate enlargeable 2}. 
 
 As an immediate corollary of Theorem \ref{thm: rigidity for minimal surface} and Theorem \ref{thm: georch free of singularity}, we provide a proof of positive mass theorem for AF manifolds with arbitrary ends and dimension no greater than $8$ (Theorem \ref{thm: pmt8dim}) without using N. Smale's regularity theorem for minimal hypersurfaces in a compact $8$-dimensional  manifold with generic metrics, and a new proof of Geroch conjecture in dimension $8$ (Theorem \ref{cor:8dim georch}).

	\subsection{Organization of The Paper}
	In Section 2, we introduce some key definitions and preliminary results necessary for proving the main theorems. Section 3 is devoted to the construction of singular harmonic functions on area-minimizing hypersurfaces. In Section 4, we  establish the positive mass theorem for smooth asymptotically flat (AF) manifolds with arbitrary ends and non-negative scalar curvature in the strong spectral sense. In Section 5, we construct smooth manifolds with nice properties from minimal hypersurfaces with isolated singularities, by performing metric deformation and warped construction. In Section 6, we prove the main theorems.

\section{Preliminaries}
\subsection{Asymptotically locally flat manifolds}

$\quad$

Asymptotically locally flat manifolds (ALF) manifolds are complete manifolds which can be regarded as natural generalizations of $\mathbf{R}^n\times \mathbf{T}^k$. For related studies about the positive mass theorems on ALF manifolds, one could see \cite{Dai04},\cite{Min09},\cite{CLSZ2021} and references therein. In this subsection, we will review some basic concepts. Let us begin with the following:

\begin{definition}\label{def: AF manifold}
		An end $E$ of an $n$-dimensional Riemannian manifold $(M,g)$ is said to be asymptotically flat (AF) of order $\tau$ for some $\tau>\frac{n-2}{2}$, if $E$ is diffeomorphic to $\mathbf{R}^n\backslash B^n_{1}(O)$\footnote{ Throughout the paper we use the following notational conventions for balls:

    \begin{itemize}
        \item We use capital letter $M$ (or $N$) to denote Riemannian ambient manifolds. The notation $\mathcal{B}_s^M(p)$ (or simply $\mathcal{B}_s(p)$) denotes the geodesic ball centered at $p$ of radius $s$ in $M$;
        \item We use $\Sigma$ to denote the minimal hypersurface in the ambient manifold $M$. The notation $B(p,s) = \Sigma\cap \mathcal{B}_s^M(p)$ denotes the extrinsic ball of radius $s$ lying in $\Sigma$. See Section 3 for detail;
        
    \end{itemize}
    
    If $(M^n,g)$ is a manifold with an AF end E, then:
    \begin{itemize}
        \item $B_s^n(O)$ (or simply $B_s(O)$) denotes the coordinate ball $\{x: x_1^2+x_2^2+\dots+x_n^2<s^2\}$ in $E\cong \mathbf{R}^n\backslash B^n_1(O)$.
    \end{itemize}}, and the metric $g$ in $E$ satisfies
		\begin{align*}
			|g_{ij}-\delta_{ij}|+|x||\partial g_{ij}|+|x|^2|\partial^2 g_{ij}| = O(|x|^{-\tau}),
		\end{align*}
		and $R_g\in L^1(E)$. The  mass of $(M,g,E)$ is defined by
            \begin{align*}
                m_{ADM}(M,g,E) = \frac{1}{2(n-1)\omega_{n-1}}\lim_{\rho\to\infty}\int_{S^{n-1}(\rho)}(\partial_ig_{ij}-\partial_j g_{ii})\nu^jd\sigma_x
            \end{align*}
  \end{definition}

    In the case above we always call $(M,g)$ an asymptotically flat manifold: $(M,g)$ has an distinguished AF end $E$ and some arbitrary ends. Throughout the paper, we always assume
    the AF end of a  Riemannian manifold $(M^n,g)$  is of order $\tau$ for some $\frac{n-2}{2}<\tau\leq n-2$.
    
  \begin{definition}\label{def: ALF manifold}
      An end $E$ of an $n+k$-dimensional Riemannian manifold $(M,g)$ is said to be asymptotically locally flat (ALF) of order $\tau$ for some $\tau>\frac{n-2}{2}$, if $E$ is diffeomorphic to $(\mathbf{R}^n\backslash B^n_{1}(O))\times \mathbf{T}^k$, and the metric in $E$ satisfies
		\begin{align*}
			|g_{ij}-\bar{g}_{ij}|+|x||\partial g_{ij}|+|x|^2|\partial^2 g_{ij}| = O(|x|^{-\tau}),
		\end{align*}
		with $R_g\in L^1(E)$. Here $\bar{g}_{ij} = g_{Euc}\oplus g_{\mathbf{T}^k}$, where $g_{\mathbf{T}^k}$ denotes a flat metric on $\mathbf{T}^k$. The  mass of $(M,g,E)$ is defined by
        \begin{align*}
            m_{ADM}(M,g,E) = \frac{1}{2\omega_{n-1}Vol(\mathbf{T}^k,g_{\mathbf{T}^k})}\lim_{\rho\to\infty}\int_{S^{n-1}(\rho)\times \mathbf{T}^k}(\partial_ig_{ij}-\partial_j g_{aa})\nu^jd\sigma_xds
        \end{align*}
        where $i,j$ run over the index of the Euclidean space $\mathbf{R}^n$ and $a$ runs over all index.
  \end{definition}

    For the convenience of subsequent discussions, we introduce the following weighted spaces:
  \begin{definition}\label{defn: weighted space}
        Let $(M^n,g,E)$ be a complete manifold with an AF end. Let $r$ be a positive smooth function on $M$ which equals $|x|$ in $E$ and equals $1$ outside a neighborhood of $E$. 
        
      (1) Given any $s\in \mathbf{R}$ and $p>1$, we define $L^p_s(M)$ to be the subspace of $L^p_{loc}(M)$ with finite weighted norm
      \begin{align*}
          \|u\|_{L^p_s(M)} = (\int_M|u|^pr^{-sp-n}d\mu)^{\frac{1}{p}}.
      \end{align*}

      For each $k\in \mathbf{Z}_+$, we define $W^{k,p}_s(M)$ to be the subspace of $W^{k,p}_{loc}(M)$ with finite weighted norm
      \begin{align*}
          \|u\|_{W^{k,p}_s(M)} = \sum_{i=0}^k \|\nabla^i u\|_{L^p_{s-i}(M)}.
      \end{align*}

      (2) Given any $s\in \mathbf{R}$ and $k\in \mathbf{Z}_+$, we define $C^k_{s}(M)$ to be the subspace of $C^k(M)$ with finite weighted norm
      \begin{align*}
          \|u\|_{C^k_{s}(M)} = \sum_{i=0}^k \|r^{-s}\nabla^i u\|_{C^0(M)}.
      \end{align*}
  \end{definition}

\subsection{PSC manifolds with free $\mathbf{S}^1$ actions}

$\quad$

In this subsection, we explore several topological obstructions to the existence of PSC on complete manifolds. These results will be used in the subsequent sections.

\begin{definition}\label{Defn: enlargeable} (\cite{GL83}\cite{Gro18})

		A compact Riemannian manifold $X^n$ is said to be enlargeable if for each $\epsilon>0$, there is an oriented covering $\Tilde{X}\longrightarrow X$ and a continuous non-zero degree map $f:\Tilde{X}\longrightarrow S^n(1)$ that is constant outside a compaact set, such that $\Lip f<\epsilon$. Here $S^n(1)$ is the unit sphere in $\mathbf{R}^{n+1}$.
	\end{definition}

The proof of the following Proposition is technical and not directly related to the main argument. We therefore defer it to Appendix C.
                \begin{proposition}\label{prop: noncompact dominate enlargeable}
            Let $X^n$ be a compact enlargeable manifold, and let $Y^n$ be a manifold which is not necessarily compact and admits a quasi-proper map $f: Y^n\longrightarrow X^n$ of non-zero degree $(n\le 7)$ (see Appendix B and Figure \ref{f2}). Then there exists no complete metric $g$ on $Y^n$ with $R_g> 0$.
        \end{proposition}

The proof of Proposition \ref{prop: noncompact dominate enlargeable} in Appendix C uses $\mu$-bubble. With a similar argument, using $\mathbf{S}^1$-invariant $\mu$-bubble in \cite{WY2023} instead of the ordinary $\mu$-bubble, we have the following proposition.
        
        \begin{proposition}\label{prop: noncompact dominate enlargeable 2}
            Let $X^n$ be a  enlargeable manifold, and let $Y^n$ be a manifold which is not necessarily compact and  admits a map $f: Y^n\longrightarrow X^n$ of non-zero degree $(n\le 7)$. Then there exists no $\mathbf{S}^1$-invariant complete metric $g$ on $Y^n\times \mathbf{S}^1$ with $R_g> 0$.
        \end{proposition}

        Since the proof of Proposition \ref{prop: noncompact dominate enlargeable 2} is a direct consequence of the argument of Proposition \ref{prop: noncompact dominate enlargeable} combined with the method developed in \cite{WY2023}, we omit the proof. As a corollary, we are able to prove the following obstruction result for PSC in the spectral sense.

        \begin{corollary}\label{cor: noncompact dominate enlargeable 3}
            Let $X^n$ be a  enlargeable manifold, and let $Y^n$ be a manifold which is not necessarily compact and admits map $f: Y^n\longrightarrow X^n$ of non-zero degree $(n\le 7)$. Then any complete metric $g$ on $Y^n$ with $\lambda_1(-\Delta_g+\beta R_g)\ge 0$ for some $\beta\ge\frac{1}{2}$ is flat.
        \end{corollary}

        To prove Corollary \ref{cor: noncompact dominate enlargeable 3}, we need an $\mathbf{S}^1$-invariant version of Kazdan's deformation theorem in \cite{Kazdan82}.

        \begin{lemma}\label{G-invariant Kazdan}
    Let $(M^n,g)$ be a complete Riemannian manifold with $\mathbf{S}^1$ acting freely and isometrically on it. Assume $R_g\ge 0$, and $Ric_g$ is not identically zero. Then there exists an $\mathbf{S}^1$-invariant complete metric $\tilde{g}$ on $M$ with $R_{\tilde{g}}>0$ everywhere.
\end{lemma}
        \begin{proof}
            We adopt the deformation arguments in \cite{Kazdan82}. We assert that the Riemannian  metric $G$ constructed  in Theorem B in \cite{Kazdan82} can be constructed to be $\mathbf{S}^1$-symmetric. Indeed, let $\rho(x)$ be the distance function of $M/\mathbf{S}^1$. By an $\mathbf{S}^1$-invariant mollification, we may assume $\rho$ is a $\mathbf{S}^1$-invariant smooth function on $M$. For any $i\geq 1$ we define
            \begin{align*}
                M_i:=\{x\in M: \rho(x)<i\}.
            \end{align*}
Then, $M_i$ is a $\mathbf{S}^1$-symmetric domain in $M$, and we can  make $p$ in Lemma 2.1 in  \cite{Kazdan82}  to be $\mathbf{S}^1$-symmetric, i.e. $p$ depends only on $x\in M/\mathbf{S}^1$. Thus, the domains and equations involved in  Theorem B in \cite{Kazdan82} are all $\mathbf{S}^1$-symmetric. By the uniqueness of solutions of involved equations, we know that the Riemannian manifold $(\hat M^{n}, G)$ constructed in  Theorem B in \cite{Kazdan82} is complete and  $\mathbf{S}^1$-symmetric, and  its scalar curvature is positive  everywhere. 
        \end{proof}

        \begin{proof}[Proof of Corollary \ref{cor: noncompact dominate enlargeable 3}]
            By \cite[Theorem 1]{FcS1980}, there exists $v\in C^{\infty}(Y)$ that solves $-\Delta_g v+\frac{1}{2}R_g v = 0$. Consider $(\hat{Y},\hat{g}) = (Y\times \mathbf{S}^1, g+v^2ds^2)$, then $R_{\hat{g}} = 0$, and $\hat{g}$ is $\mathbf{S}^1$-symmetric. If the Ricci curvature of $(\hat{Y},\hat{g})$ is not identically zero, it then follows from Lemma \ref{G-invariant Kazdan} that $\hat{Y}$ admits a $\mathbf{S}^1$-invariant complete metric with positive scalar curvature everywhere, which contradicts with Proposition \ref{prop: noncompact dominate enlargeable 2}. Hence, we have $Ric_{\hat{g}}\equiv 0$.  By direct computation, $Ric_{\hat{g}}(\frac{\partial}{\partial s},\frac{\partial}{\partial s}) = -v\Delta_g v = 0$, so $v$ is a constant and $Ric_{g}\equiv 0$.

            Next, we show $Y$ is compact. Suppose this is not the case, we have $S_{\infty}\ne\emptyset$. Let $\tilde{X}$ be the universal covering of $X$. We adopt the notations in the proof of Proposition \ref{prop: noncompact dominate enlargeable} in Appendix C. Since enlargeable manifolds have infinite fundamental group, $q_X^{-1}(U_{\epsilon})$ has infinitely many components. Consequently, $\tilde{f}^{-1}(q_X^{-1}(U_{\epsilon})) = q_Y^{-1}(f^{-1}(U_{\epsilon}))$ has infinitely many non-compact components. From the properness of $\tilde{f}$ we deduce $\tilde{Y}$ has infinitely many ends. Since manifolds with at least two ends contains a geodesic line, it follows readily from \cite{CG71} that Ricci flat manifolds have at most two ends, a contradiction.

            Since the enlargeability is preserved under non-zero degree maps, $Y$ is also enlargeable. The flatness then follows from \cite[Proposition 4.5.8]{LM89}.
        \end{proof}

       As an application, we obtain the following $\mathbf{S}^1$-invariant version positive mass theorem for ALF manifolds with $\mathbf{S}^1$ fibers and dimension less or equal than $8$:
        \begin{proposition}\label{prop: PMT for S^1 symmetric ALF}
	Let $(\hat{M}^{n+1}, \hat{g}) = (M^{n}\times \mathbf{S}^1,\hat{g})$ be an ${n+1}$-dimensional complete Riemannian manifold with an $ALF$ end $\hat{E} = E\times \mathbf{S}^1$ and $2\leq n\leq 7$. Assume $\hat{g}$ is $\mathbf{S}^1$ invariant along the $\mathbf{S}^1$ fiber, and $R_{\hat{g}}\ge 0$. Then $m_{ADM}(\hat{M},\hat{g},\hat{E})\ge 0$. Furthermore, if $m_{ADM}(\hat{M},\hat{g},\hat{E})= 0$, then $\hat{M}$ is isometric to the Riemannian product $(\mathbf{R}^{n}\times \mathbf{S}^1, g_{euc}+ds^2)$.
\end{proposition}

\begin{proof}
    The proof is a modification of that of \cite[ Theorem 1.10]{CLSZ2021}. Assume that the ADM mass is negative. Due to the $\mathbf{S}^1$-symmetry of $(\hat M, \hat g )$ and \cite[Proposition 4.10]{CLSZ2021}
 we know that the conformal factor in the gluing is independent of $s$, where $s$ parametrizes $\mathbf{S}^1$. To be more precise, from Fredholm alternative the equation (4.16) in \cite{CLSZ2021} has a unique solution, so it must be $\mathbf{S}^1$-symmetric as long as $(\hat M, \hat g )$ is $\mathbf{S}^1$-symmetric. Thus, just as what was done in the proof of \cite[Theorem 1.10 ]{CLSZ2021}, we finally get an $\mathbf{S}^1$-symmetric complete Riemannian manifold $((\mathbf{T}^{n}\# N^{n})\times \mathbf{S}^1, \tilde{g})$ with $n\leq 7$, which has non-negative scalar curvature everywhere and strictly positive scalar curvature somewhere.  Here $N^{n}$ is an $n$-dimensional manifold. Now, this contradicts to Proposition \ref{prop: noncompact dominate enlargeable 2} and Lemma \ref{G-invariant Kazdan}. The rigidity part also follows from the same reason and the argument in \cite{CLSZ2021}.
\end{proof}

\section{Singular harmonic functions on area-minimizing hypersurfaces with isolated singularities}

In this section, we aim to obtain nice estimates for singular harmonic functions on an area-minimizing hypersurface $\Sigma^n (n\ge 7)$ with isolated singular set $\mathcal{S}$, where $\Sigma^n$ lies in an ambient manifold $N^{n+1}$ . For the sake of applying general theory of the metric measure space, we regard $(\Sigma,d,\mu)$ as a metric measure space, where $d$ is the induced metric from the ambient space $N^{n+1}$ (In other word, $d = (d_N)|_\Sigma$, where $d_N$ is the ambient distance), and $\mu$ is the restriction of the Hausdorff $n$-measure of $N$ on $\Sigma$.

Throughout the paper, We use the notation $B(x,r)$ to denote the metric ball centered at $x$ with radius $r$ in $\Sigma$. Since $d$ is induced from $N$, it is worth noting that $B(x,r)$ is exactly the intersection of $\Sigma$ with the extrinsic ball $\mathcal{B}^N_r(x)$ centered at $x$ with radius $r$ in $N$, provided $r<\inj_N(x)$.

Since $\Sigma^n$ is area-minimizing, from the monotonicity formula of minimal hypersurfaces and \cite{BG1972}, we know that $(\Sigma,d,\mu)$ satisfies standard structural conditions of metric measure spaces, $i.e.$
\begin{enumerate}
    \item $(\Sigma,d,\mu)$ is \textit{locally volume doubling};
    \item $(\Sigma,d,\mu)$ supports a \textit{locally Poincare inequality} (See Proposition \ref{prop: Pioncare-Sobolev inequality} and Corollary \ref{cor: 12poincare} for detailed form).
\end{enumerate}
Here, the term \textit{locally} means that on a bounded domain $X\subset\subset\Sigma$, the doubling inequality and the Poincare inequality (PI) hold on all balls $B = B(x,r)\subset \Sigma$ whenever $x\in X$ and
\begin{align}\label{eq: r0}
    r<r_0 = r_0 (X,d_X,\mu_X).
\end{align}
Moreover, the doubling constant and PI constants are also bounded in terms of $(X,d_X,\mu_X)$. Here, $d_X = d|_X$ and $\mu_X = \mu|_X$ denote the restriction of $d$ and $\mu$ on $X$.
\subsection{Sobolev space on metric measure spaces}

$\quad$

In this subsection, we will collect some known facts about the Sobolev space defined on $(X,d_X,\mu_X)$. We will always assume $X\subset\subset \Sigma$, so it follows that $\mu(X)<+\infty$.

\begin{definition}
    For each $f\in L^2(X,\mu_X)$, the Cheeger energy of $f$ is defined by
    \begin{align}\label{eq: 67}
        Ch(f) = \inf\{\liminf_{i\to\infty}\frac{1}{2}\int_X(\Lip f_i)^2d\mu_X:\quad f_i\in\Lip(X),\  \|f_i-f\|_{L^2}\to 0.\}
    \end{align}

The Sobolev space $W^{1,2}(X,d_X,\mu_X)$ is then defined to be $\{f: Ch(f)<+\infty\}$. 
\end{definition}
Since $(X,d_X,\mu_X)$ is locally volume doubling, $W^{1,2}(X,d_X,\mu_X)$ is Banach and reflexiv when endowed with the norm
\begin{align}\label{eq: 73}
    \|f\|_{W^{1,2}} = (\|f\|_{L^2(X,\mu_X)}^2+2Ch(f))^{\frac{1}{2}}.
\end{align}

\begin{lemma}\label{lem: estimate for eta} (Zero capacity property)
    For any neighborhood $U$ for $\mathcal{S}$ and $\epsilon>0$, there exists an open set $V$ with $\mathcal{S}\subset V\subset U$ and a cut off function $\eta\in C^{\infty}(M\backslash \mathcal{S})$ that satisfies $\eta = 0$ in $V$, $\eta = 1$ in $M\backslash U$ and
    \begin{align*}
        \int_{M}|\nabla\eta|^2d\mu<\epsilon^2.
    \end{align*}
\end{lemma}
\begin{proof}
    The statement actually holds under the more general condition $\mathcal{H}^{n-2}(\mathcal{S}) = 0$, which is automatically satisfied when $n\ge 7$ and $\mathcal{S}$ is isolated. We use an argument of \cite[Appendix]{AX24}. For each $\epsilon>0$, we can find a finite collection of balls $B_{r_i}(x_i)$ that covers $\mathcal{S}$ and satisfies
    \begin{align*}
        \sum_ir_i^{n-2}<\epsilon^2.
    \end{align*}
    For each $i$, we find a cutoff function $\eta_i$ which is smooth in $M\backslash\mathcal{S}$ that satisfies
    \begin{align*}
        \eta_i = 0\text{ in } B_{2r_i}(x_i),\quad \eta_i = 1\text{ in } M\backslash B_{3r_i}(x_i),\quad |\nabla_g\eta_i|<\frac{2}{r_i}\text{ in }M\backslash\mathcal{S}.
    \end{align*}
    Let $\hat{\eta} = \min_i\{\eta_i\}\in\Lip(M\backslash\mathcal{S})$ and let $\eta$ be obtained by regularizing $\hat{\eta}$. We have $|\nabla_g\hat{\eta}|\le 2|\nabla_g\eta|$ and
    \begin{align*}
        \eta = 0\text{ in }\bigcup_iB_{r_i}(x_i),\quad\eta = 1\text{ in }M\backslash\bigcup_iB_{4r_i}(x_i).
    \end{align*}
    Therefore,
    \begin{align*}
        \int_M|\nabla_g\eta|^2d\mu\le 2\sum_i\int_M|\nabla_g\eta_i|^2d\mu\le 8C\sum_i r_i^n\cdot r_i^{-2}<8C\epsilon^2,
    \end{align*}
    and the conclusion follows.
\end{proof}

The next Proposition is a consequence of \cite[Proposition 4.10]{AGS14}, \cite[Proposition 3.3]{Hon18} and the zero capacity property Proposition \ref{lem: estimate for eta}. A slight difference is that in \cite[Proposition 3.3]{Hon18}, $X$ is assumed to be compact. As we have assumed $\mu(X)<+\infty$, the proof of \cite[Proposition 3.3]{Hon18} can be smoothly carried to our setting.
\begin{proposition} (Density of Lipschitz functions in $W^{1,2}$)
    $W^{1,2}(X)$ is a Hilbert space and $\Lip(X)$ is dense in $W^{1,2}(X)$.
\end{proposition}
\begin{proof}
    We sketch the proof following the strategy in \cite[Proposition 3.3]{Hon18}. First, on the smooth part, we have
    \begin{align}\label{eq: 66}
        Ch(\varphi+\psi)+Ch(\varphi-\psi) = 2Ch(\varphi)+2Ch(\psi)
    \end{align}
    for all $\varphi,\psi\in \Lip(X\backslash\mathcal{S})$. By combining Lemma \ref{lem: estimate for eta} with an approximation argument, we obtain that \eqref{eq: 66} holds for all $\varphi,\psi\in \Lip(X)$. By \eqref{eq: 67} we can further obtain the same thing holds for $\varphi,\psi\in W^{1,2}(X)$ (see \cite[Proposition 3.3]{Hon18} for detail). The conclusion then follows directly from \cite[Proposition 4.10]{AGS14}.
\end{proof}

\begin{lemma}\label{lem: compact embed} ($L^2$ strong compactness)
    The inclusion $W^{1,2}(X)\hookrightarrow L^2(X)$ is a compact operator.
\end{lemma}
\begin{proof}
    Since $X$ is doubling due to the local volume bound, and $X$ satisfies the Poincare inequality, the result of the lemma follows from \cite[Theorem 8.1]{HK00}.
\end{proof}

 For $f\in\Lip(X)$, we say $u\in W^{1,2}(X)$ is the solution of the Poisson equation $-\Delta u = f$ if for any $\varphi\in \Lip(X)$ with compact support, there holds
\begin{align}\label{eq: 71}
    \int_X \nabla u\cdot\nabla\varphi d\mu= \int_Xf\varphi d\mu.
\end{align}
Define $W^{1,2}_0(X)$ to be the closure of Lipschitz functions with compact support in $W^{1,2}(X)$. We say $u$ satisfies the Direchlet boundary condition $u = 0$ on $\partial X$ if it further satisfies $u\in W^{1,2}_0(X)$. We also say $-\Delta u\le f$ if we have $\le$ sign in \eqref{eq: 71} for all $\varphi\ge 0$ with compact support in $\Lip(X)$.

\begin{lemma}\label{lem: Poisson equation} (Solvability of the Poisson equation)
    For $f\in\Lip(X)$, there exists $u\in W^{1,2}_0(X)$ such that $-\Delta u = f$.
\end{lemma}

\begin{proof}
    Consider the functional
    \begin{align*}
        E(u) = \int_X (\frac{1}{2}|\nabla u|^2+fu)d\mu.
    \end{align*}
    By the Sobolev inequality (See Lemma \ref{prop: Sobolev inequality on S} below, which holds globally in the domain $X$), $E(u)$ is bounded from below, so we can consider a minimizing sequence $u_i$ for $E$. Using Lemma \ref{lem: compact embed}, we obtain a limit $u$. Together with the lower semicontinuity of $E$, $u$ is a minimizer for $E$. Then a variational argument shows that $u$ solves the equation.
\end{proof}

By employing the same reasoning as presented in \cite[Theorem 8.15, Theorem 8.16]{GT2001}, we deduce that the weak maximum principle holds for subharmonic functions $u$ on $X$.
\begin{lemma}\label{lem: weak maximum principle} (Weak maximum principle)
    Assume $u\in W^{1,2}(X)$ satisfies $\Delta u \ge 0$ in $X$. Define $\sup_{\partial X} u = \inf \{l\in\mathbf{R}: \max \{u-l,0\}\in W^{1,2}_0(X)\}$. Then we have
    \begin{align*}
        \sup_X u\le \sup_{\partial X} u.
    \end{align*}
\end{lemma}

\begin{lemma}\label{lem: removable singularity} (Removable singularity)
    Let $u\in W^{1,2}(X\backslash\mathcal{S})\cap L^{\infty}(X\backslash\mathcal{S})$ be a solution to $\Delta u = 0$ in $X\backslash\mathcal{S}$, then $u\in W^{1,2}(X)$ and is a solution to $\Delta u = 0$ in $X$. 
\end{lemma}
\begin{proof}
    Let $\eta$ be the cut-off function in Lemma \ref{lem: estimate for eta}. By integration by parts we have
    \begin{align*}
        \int_{X\backslash\mathcal{S}}\divv(\eta^2 u\nabla u)d\mu = \int_{X\backslash\mathcal{S}}|\nabla\eta u|^2-|\nabla\eta|^2u^2 d\mu.
    \end{align*}
    Since the left hand side is bounded, using Lemma \ref{lem: estimate for eta} and passing to a limit we obtain $u\in W^{1,2}(X)$. 
    
    Next, we check that $u$ is the weak solution of the Poisson equation $\Delta u = 0$ in $X$. Let $\varphi\in C^{\infty}_0(X)$ and $\eta$ be the cut-off function in Lemma \ref{lem: estimate for eta}, by the definition of the weak solution we have
    \begin{align*}
    \int_X \nabla u\cdot\nabla\eta\varphi d\mu= \int_Xf\eta\varphi d\mu.
\end{align*}
Passing to a limit we see that $u$ solves the equation $\Delta u = 0$ in $X$.
    \end{proof}

Since we have the Poincare inequality \eqref{eq: A Poincare}, by exactly the same argument in \cite{BG1972}, we have the following Harnack's inequality holds in $X$.
    \begin{lemma}\label{lem: Harnack}(Harnack's inequality)
        There exists $\alpha ,\beta, C$ depending only on $n$ and $X$, such that for any positive solution $u\in W^{1,2}(B(p,R))$ of equation $\Delta u = 0$ in $B(p,R)\subset X$, $R<r_0$ (Here $r_0$ is given by \eqref{eq: r0}), we have $u\in C^{\alpha}$ and
        \begin{align*}
           \sup_{B(p,\beta R)} u \le C \inf_{B(p,\beta R)} u .
        \end{align*}
        
        \end{lemma}

Finally, we have the strong maximum principle following directly from Lemma \ref{lem: Harnack}.

\begin{lemma}\label{lem: strong maximum principle}(Strong maximum principle)
    Assume $u\in W^{1,2}(X)$ satisfies $\Delta u = 0$ in $X$. If $u$ attains its maximum at an interior point $x_0$ in $X$, then $u$ is a constant in the connected component of $X$ containing $x_0$.
\end{lemma}
\begin{proof}
    Assume $u$ attains its maximum $A$ at $x_0$. Let $B(x_0,R)\subset X$ satisfying $R<r_0$, then by applying the Harnack inequality Lemma \ref{lem: Harnack} to $A+\epsilon - u$ yields that
    \begin{align*}
         \sup_{B(x_0,\beta R)} (A+\epsilon-u) \le C \inf_{B(p,\beta R)} (A+\epsilon-u) = C\epsilon 
    \end{align*}
    By letting $\epsilon\to 0$, we obtain $u\equiv A$ in $B(x_0,\beta R)$, which yields the desired conclusion.
\end{proof}

\subsection{Singular harmonic functions on a compact domain}

We will utilize a singular harmonic function to blow up the singular set $\mathcal{S}$. We begin by recalling the existence of Dirichlet Green's function on metric measure spaces, which was proved in \cite{BBL20}.

\begin{proposition}\label{prop:existence of singular positive hf}
Let $\Omega$ be a bounded domain in $\Sigma$ with smooth boundary and let $p\in \mathcal{S}\cap\Omega$ be a singular point. Then there is a harmonic function $G$ on $\Omega\setminus\{p\}$ such that
\begin{enumerate}
	\item $G|_{\partial\Omega}=0$; $G\in W^{1,2}_{loc}(\Omega\backslash\{p\})\cap C^\alpha(\Omega\backslash\{p\})\cap C^{\infty}(\Omega\backslash\mathcal{S})$;
	\item $G(x)>0$ for any $x\in \Omega\setminus\{p\}$;
	\item $\lim_{x\rightarrow p}G(x)=+\infty$.
\end{enumerate} 
\end{proposition}

\begin{proof}
	Let $\{r_i\}$ be a decreasing sequence which approaches zero, and $G_i$ be the solution of the following Dirichlet problem
\begin{equation}\label{eq:weak hf}
\left\{
\begin{aligned}
&\Delta G_i=0, \quad \text{in $\Omega\setminus B(p,r_i)$},\\
&G_i|_{\partial \Omega}=0,\\
&G_i|_{\partial B(p,r_i)}=1.
\end{aligned}
\right.
\end{equation}
By the weak maximum principle Lemma \ref{lem: weak maximum principle}, we see that 
$G_i$ in \eqref{eq:weak hf} is positive in $\Omega\setminus  B(p,r_i)$. Let $x_0\in \Omega\backslash \mathcal{S}$, by rescaling we may assume $G_i(x_0)=1$. From the Harnack inequality Lemma \ref{lem: Harnack}, $\{G_i\}$ converges subsequentially to a positive harmonic function $G$ in $\Omega\backslash\mathcal{S}$. Since $n\ge 7$, the monotonicity formula for minimal hypersurfaces yields that the $2$-capacity of $p$ in $\Omega$ is $0$. Hence, we can apply \cite[Lemma 4.3, Lemma 4.4]{BBL20} to conclude that $a = \lim_{x\to p}G(x)$ exists (We note that \cite{BBL20} only assumed $\Omega$ satisfies the doubling property and the Poincare inequality. This general assumption can be found at the beginning of \cite[Section 4]{BBL20}). A slight difference is that we only assume a local Poincaré inequality. However, this causes no difficulty since we can work on sufficiently small balls. If $a<\infty$, then Lemma \ref{lem: strong maximum principle} yields that $G\equiv 0$, a contradiction with $G(x_0) = 1$. Thus, we obtain $\lim_{x\to p}G(x) = +\infty$.

\end{proof}

\begin{proposition}\label{prop:growth of hf1}
Let \( G \) be a positive singular harmonic function in \( \Omega \setminus \{p\} \) as above. Then there exists a positive constant \( C\), such that for any sufficiently small $r$, the following holds:
$$
G|_{\partial B(p,r)}\geq Cr^{2-n},
$$
\end{proposition}
\begin{proof}
     By \cite[Theorem 1.5]{BBL20}, we have
\begin{align*}
    G(x) \ge C \, \capp_2(B_r,\Omega)^{-1}
\end{align*}
for $x \in \partial B(p,r)$ when $r$ is sufficiently small.
The monotonicity formula for $\Sigma$ yields the local estimate
\[
\mathcal{H}^n(\Sigma \cap B(p,r)) \le C r^n .
\]
for small $r$. Consequently, we obtain $\capp_2(B_r,\Omega) \le C r^{n-2}$ for $r$ sufficiently small, 
which completes the proof.

\end{proof}

\subsection{Extension of the singular harmonic function to the total space}

$\quad$

In this subsection, we extend the singular harmonic function to the whole singular space. We will use arguments presented in \cite{Li04} and \cite{guo2024}.
\begin{proposition}\label{prop: extend the singular harmonic function}
    Let $\Sigma^n$ be a complete area-minimizing hypersurface with isolated singular set $\mathcal{S}$ in an ambient manifold $N^{n+1}$. Suppose there is an AF end $E_\Sigma$ of $\Sigma$ satisfying $E_\Sigma\cap\mathcal{S} = \emptyset$. Then for each $p\in \mathcal{S}$, there exists a singular harmonic function $G\in W^{1,2}_{loc}(\Sigma\backslash\{p\})\cap C^{\alpha}(\Sigma\backslash\{p\})\cap C^{\infty}(\Sigma\backslash \mathcal{S})$ that satisfies
    \begin{itemize}
        \item $G$ is weakly harmonic in $\Sigma\backslash\{p\}$;
        \item $\lim_{x\in E_\Sigma, |x|\to\infty}G(x) = 0$;
        \item $G(x)\ge Cd(x,p)^{2-n}$ for  $d(x,p)$ small enough;
    \end{itemize}
\end{proposition}

\begin{proof}
For given $p$, let $p\in\Omega_0\subset \Omega_1\subset\dots$ be a compact exhaustion of $\Sigma$, such that $\Omega_0\cap E_\Sigma = \emptyset$, $\bar{\Omega}_0\cap\bar{E}_\Sigma = \partial E_\Sigma$, and $\partial\Omega_i\cap\mathcal{S} = \emptyset$.

\textbf{Step 1: } Construct singular harmonic functions $G_i$ on $\Omega_i$, such that $G_i(x)\ge G_j(x)$ for $i>j$ and $x\in\Omega_j$.

By  Proposition \ref{prop:growth of hf1}, there is a singular harmonic function $G_0\in W^{1,2}_{loc}(\Omega_0\backslash p)\cap C^{\alpha}(\Omega_0\backslash p)\cap C^{\infty}(\Omega_0\backslash \mathcal{S})$ satisfying
\begin{align}\label{eq: 52}
    G_0(x) \ge Cd^{2-n}(x,p),
\end{align}
when $d(x,p)$ is sufficiently small.

Denote the $\rho$-neighborhood of $\Omega_0$ by $B_{\rho}(\Omega_0)$. Since $\partial \Omega_0$ lies in the smooth part of $\Sigma$, we can select $\rho$ to be sufficiently small such that $(B_{\rho}(\Omega_0)\backslash \Omega_0)\cap \mathcal{S} = \emptyset$. Now we use a cut-off trick in \cite{guo2024}: Let $\chi$ be a cut off function that satisfies $\chi = 1$ in $\Omega_0$, $\chi = 0$ in  $\Sigma\backslash B_{\rho}(\Omega_0)$. Thanks to Lemma \ref{lem: Poisson equation}, we can find $h_i\in W^{1,2}(\Omega_i)$ which solves
\begin{equation}\label{eq: 53}	
\left\{
\begin{aligned}
 & -\Delta h_i = G_0\Delta\chi + 2\nabla G_0\cdot\nabla\chi  \quad \text{in $\Omega_i$}, \\
  &  h_i|_{\partial \Omega_i}=0.
\end{aligned}
\right.
\end{equation}
By the Harnack inequality, $h_i \in C^{\alpha}(\Omega_i)$ and hence is continuous on $\Omega_i$. 
Define $G_i = \chi G_0 + h_i$. Then $G_i \in W^{1,2}_{\mathrm{loc}}(\Omega_i \setminus \{p\}) 
\cap C^{\alpha}(\Omega_i \setminus \{p\}) 
\cap C^{\infty}(\Omega_i \setminus \mathcal{S}).$
Moreover, $G_i$ is a singular harmonic function on $\Omega_i$ with the same growth rate as $G_0$ near $p$.

\textbf{Claim: } If $i>j$, then $h_i\ge h_j$ in $\Omega_j$.

From Lemma \ref{lem: weak maximum principle} we obtain $G_i\ge 0$ in $\Omega_i$. Since $\chi$ is supported in $B_{2\rho}(p)$, we have $h_i = G_i\ge 0$ on $\partial\Omega_j$. On the other hand, $h_j = 0$ on  $\partial\Omega_j$. In conjunction with \eqref{eq: 53} and Lemma \ref{lem: weak maximum principle}, we see the Claim is true.

As a direct corollary, we see $G_i(x)\ge G_j(x)$ for $i>j$ and $x\in\Omega_j$.

\textbf{Step 2: } Obtain a uniform bound on $G_i$.

Let $s_i = \sup_{\partial\Omega_0} G_i(x)$. 

\textbf{Claim: }$s_i$ is uniformly bounded.

We follow the argument of \cite{Li04}. Assume, for the sake of contradiction that $s_i\to\infty$. Let $\varphi_i = s_i^{-1}G_i$, then by the maximum principle and the fact that the difference between  $G_i$ and $G_0$ is a bounded continuous function, we have
\begin{align*}
    s_i^{-1}G_0\le \varphi_i\le s_i^{-1}G_0 +1.
\end{align*}
Passing to a subsequence, we may assume that $\varphi_i$ converges smoothly to a function $\varphi$ with $0\le \varphi\le 1$ in $\Sigma\backslash\mathcal{S}$. By the removable singularity Lemma \ref{lem: removable singularity}, $\varphi$ extends to a $W^{1,2}$ weak harmonic function in $M$. 

 It is well known that there exists a harmonic function $f$ on the AF end $E_\Sigma$ that satisfies $f|_{\partial E_\Sigma} = 1$ and $\lim_{|x|\to\infty} f(x) = 0$. Note $\sup_{\partial E_\Sigma}\varphi_i\leq1$.
From the weak maximum principle,  $0\le \varphi_i\le f$ in $E_\Sigma\cap\Omega_i$. 
 Hence, passing to a subsequence, we have $0\le \varphi\le f$ in $E_\Sigma$. Once again, using the weak maximum principle, the function
\begin{align*}
    M_i(r) = \sup_{\partial B_r(p)} \varphi_i
\end{align*}
is non-increasing in $r$. By passing to a limit, we see that
\begin{align*}
    M(r) = \sup_{\partial B_r(p)} \varphi
\end{align*}
is also non-increasing, thereby attaining its  positive maximum in $p$. This contradicts the strong maximum principle Lemma \ref{lem: strong maximum principle} since $\lim_{x\in E_\Sigma,|x|\to\infty}\varphi(x) = 0$.

Once the Claim is verified, the desired extension $G$ comes by taking a limit $G = \lim_{i\to\infty}G_i$.
\end{proof}

\begin{proposition}\label{lem: extend harmonic function 2}
     Let $\Sigma^n$ be a complete area-minimizing hypersurface with isolated singular set $\mathcal{S}$ in an ambient manifold $N^{n+1}$. Suppose there is an AF end $E_\Sigma$ of $\Sigma$ satisfying $E_\Sigma\cap\mathcal{S} = \emptyset$. Then there exists $G\in C^{\infty}(\Sigma\backslash\mathcal{S})$ satisfying
    \begin{itemize}
        \item $\Delta_\Sigma G = 0$ in $\Sigma\backslash\mathcal{S}$;
        \item $\lim_{x\in E_\Sigma, |x|\to\infty}G(x) = 0$;
        \item For each $p\in \mathcal{S}$, there holds $G(x)\ge C_pd(x,p)^{2-n}$
        \ \ for $d(x,p)$   small enough.
    \end{itemize}
\end{proposition}
\begin{proof}
Since $\mathcal{S}$ is isolated, it is countable. Hence we may enumerate it as $\mathcal{S} = \{p_1,p_2,\dots,\}$. From Proposition \ref{prop: extend the singular harmonic function}, for each $p_l\in\mathcal{S}$, there exists a singular harmonic function $G^{(l)}$ blowing up $p_l$, which satisfies the conclusion of Proposition \ref{prop: extend the singular harmonic function} with $p$ replaced by $p_l$. By renormalization, we may assume $\sup_{\partial E_\Sigma}G^{(l)}\le 1$ for each $l$. It is direct to see the function
\begin{align}\label{eq: 65}
    G = \sum_{l=1}^{\infty}2^{-l}G^{(l)}
\end{align}
has the desired properties.
\end{proof}

\section{Positive mass theorem for AF manifolds with arbitrary ends and PSC in the strong spectral sense}

In this section, we prove Theorem \ref{thm: PMT arbitrary end spectral} (1) and (2). 
Since the rigidity statement (2) will not be used later, readers primarily interested in the minimal hypersurface singularity resolution in dimension $8$ may skip Proposition \ref{prop: eq1-arbitrary end} (2), Lemma \ref{lem: mass decay ALF 2} (Case 2), and Lemma \ref{lem: make a point strictly positive} on a first reading.

We begin with the maximum principle for elliptic equations on manifolds with certain spectral condition. The lemma is essentially proved in \cite{FcS1980} and will be frequently used in our argument, so we record it here for the convenience of the readers.

\begin{lemma}\label{lem: maximum principle for spectral condition}
Let $(M^n,g)$ be a complete Riemannian manifold, and $\Omega\subset M$ an open set with smooth boundary. Let $q$ be a smooth function on $M$, and assume $\lambda_1(-\Delta+q)\ge 0$ in $M$.
\begin{enumerate}
    \item $\Omega$ is bounded. If $u_1,u_2\in C^\infty(\Omega)$ solves $-\Delta u_i+qu_i = 0 \quad(i = 1,2)$ and $u_1>u_2$ on $\partial\Omega$, then $u_1>u_2$ in $\Omega$.
    \item $M$ contains an AF end $E$ and $\Omega$ is an open neighborhood of $E$ such that $\Omega\backslash E$ is compact. If $u_1,u_2\in C^\infty(\Omega)$ solve $-\Delta u_i+qu_i = 0 \quad(i = 1,2)$, $u_1>u_2\ge 0$ on $\partial\Omega$, and $\lim_{x\to\infty,x\in E}u_i(x) = 1\quad (i = 1,2)$, then $u_1>u_2$ in $\Omega$.
\end{enumerate}
\begin{proof}
    (1) We only need to consider the case that $U = \{x\in \Omega, u_2(x)>u_1(x)\}\ne\emptyset$. Otherwise, we have $u_1\geq u_2$ in $\Omega$ and the strict inequality follows from the strong maximum principle \cite[p. 33-34]{GT2001} for elliptic equations.
    Now, from \cite[Theorem 1]{FcS1980}, $\lambda_1(U)>\lambda_1(\Omega) = c>0$. Multiplying $u_2-u_1$ on both sides of the equation $-\Delta(u_2-u_1)+q(u_2-u_1)=0$ and integrating by parts, we deduce $u_2\equiv u_1$ in $U$. A contradiction will follow by the unique continuation if $U\ne \emptyset$. 

    (2) First,  the condition $u_1>u_2\ge 0$ on $\partial\Omega$, the asymptotics of $u_i$'s and (1) imply $u_i$ must be positive in $\Omega$. Once again, from the asymptotics of $u_i$'s at the infinity, we deduce that for any $\epsilon\in (0,1)$, $U = \{x\in \Omega, (1-\epsilon)u_2(x)>u_1(x)\}$ must be a compact set in $\Omega$. It follows from (1) that $U$ must be empty. Hence $(1-\epsilon)u_2\le u_1$ for any $\epsilon\in (0,1)$. Letting $\epsilon\to 0$ and using the strong maximum principle for elliptic equations, we get the desired result.
\end{proof}

\end{lemma}

        Next we prove a proposition, which is crucial in our construction of the ALF manifolds.
       \begin{proposition}\label{prop: eq1-arbitrary end}
		Let $(M^n,g)$ be a smooth AF manifolds with arbitrary ends which is not necessarily complete, $E$ being its AF end of order $\tau>\frac{n-2}{2}$. Suppose $(M^n,g)$ has $\beta$-scalar curvature no less than $h$ in the strong spectral sense for some $\beta>0$, where $h\in C^\infty(M)\cap C_{-\tau-2}(E)\cap L^1(E)$ is a non-negative function on $M$.

		(1) If $h(p)>0$ for some point  $p\in M$, then for any smooth function $Q$ with $0\leq Q\le h$ and $Q(p)<h(p)$, there is a positive and smooth function $u$ on $M$ that satisfies 
		$$
		-\Delta u +\beta(R_g-Q) u = 0
		$$
		and $u\to 1$ as $x\to \infty$ in $E$.
		
		(2) If $h\equiv 0$, then one of the following happens:

		\begin{itemize}
			\item $R_g\equiv 0$.
			\item for any $0<\beta'<\beta$, there is a smooth positive nonconstant function $u$ on $M$ that satisfies
			$$
			-\Delta u +\beta'R_g u = 0
			$$
			and $u\to 1$  as $x\to \infty$ in $E$.
		\end{itemize}
	\end{proposition}
	\begin{proof}[Proof of Proposition \ref{prop: eq1-arbitrary end} (1)]
		
	Let $\{\Omega_i\}$ be a sequence of exhausting compact domains of $M$ with smooth boundaries. We use $\partial_+ \Omega_i$ and $\partial_- \Omega_i$ to denote the boundary of $\Omega_i$ in $E$ and $M \setminus E$ respectively.   
    Our assumption $\lambda_1(-\Delta_g+\beta(R_g-Q))\ge 0$ and \cite{FcS1980} implies the first Dirichlet eigenvalue of $-\Delta_g+\beta(R_g-Q)$ in $\Omega_i$ is strictly positive. Thus, from the Fredholm alternative, we know that for each $\Omega_i$ there is a  positive smooth function $u_i$ on $\Omega_i$ that satisfies
		\begin{equation}\label{eq4}
			\left\{
			\begin{aligned}
				&-\Delta u_i +\beta(R_g-Q) u_i=0 \quad \text{in $\Omega_i$},\\
				&u_i|_{\partial_+\Omega_i}=1,\\
				&u_i|_{\partial_-\Omega_i}=0.
			\end{aligned}
			\right.	
		\end{equation}	
        
		Without loss of the generality, for each $i$, we may assume  $\partial_+\Omega_i=\partial B^n_{2^i}(O)$ is the coordinated sphere of radius $2^i$ in $E$,  and $p\in \Omega_i$ for each $i$. We claim the following holds:
		
		{\bf Claim:} {\it There is a constant $\Lambda>0$ independent on $i$ so that $0<u_i(p)\leq \Lambda$.}
        The proof of the Claim is devided into 4 steps.
        
        \textbf{Step 1: Reducing the Claim to the estimate of an energy $\mathcal{E}(\hat{w}_i)$} .

$\quad$

		Suppose the Claim  is false, then there exists a subsequence(still denoted by $i$) such that $\Lambda_i:=u_i(p)\rightarrow \infty$. Let $w_i:=\Lambda^{-1}_i u_i$, we have $w_i(p)=1$ and 
		\begin{equation}\label{eq5}
			\left\{
			\begin{aligned}
				&-\Delta w_i +\beta(R_g-Q) w_i=0 \quad \text{in $\Omega_i$},\\
				&w_i|_{\partial_+\Omega_i}=\Lambda^{-1}_i,\\
				&w_i|_{\partial_-\Omega_i}=0.
			\end{aligned}
			\right.	
		\end{equation}	
		
		By Harnack inequality, we may assume that after passing to a subsequence, $\{w_i\}$ locally smoothly converges to a positive and smooth function $w$ satisfying
		$$
		-\Delta w +\beta(R_g-Q) w=0 \quad \text{in $M$}, \lim_{x\in E,|x|\to\infty}w(x) = 0
		$$
and $w(p)=1$. Denote $\Omega_{ij}:=\Omega_i\setminus\{(\Omega_i\setminus \Omega_j)\cap E\}$ for $i>j\gg1$. Let $\hat{w}_i\in C^{0,1}(M)$ be given by
\begin{equation}
\hat{w}_i=\left\{
			\begin{aligned}
&\Lambda_i^{-1}\quad \text{ in $E\backslash \Omega_i$,}\\
& w_i\quad \text{ in $\Omega_i$,}\\
&0 \quad \text{ in $M\setminus(\Omega_i \cup E)$.}\end{aligned}
\right.	\nonumber	
\end{equation}
 By applying the strong spectral condition to $\hat{w}_i$ we obtain
		\begin{equation}\label{eq: 15}
\begin{split}
\mathcal{E}(\hat{w}_i)  :=\;&
\underbrace{\int_{E\backslash\Omega_i}\beta|R_g-Q|\Lambda_i^{-2}\,d\mu_g}_{\text{I}} +\underbrace{\int_{\Omega_{ij}}
\big(|\nabla_M w_i|^2+\beta(R_g-Q) w_i^2\big)\,d\mu_g}_{\text{II}}
 \\
&+\underbrace{\int_{(\Omega_{i}\setminus\Omega_{j})\cap E }
\big(|\nabla_M w_i|^2+\beta(R_g-Q) w_i^2\big)\,d\mu_g}_{\text{III}}
 \\
\ge\;&
\int_{\Omega_{ij} }\beta( h-Q) w_i^2\,d\mu_g \\
\ge\;& C_0 = C_0(M,g_{M},R_g,Q)>0.
\end{split}
\end{equation}

		In the last inequality we have used the fact $w_i(p) = 1$, $h(p)>Q(p)$, and the Harnack inequality for $w_i$.

To obtain a contradiction, we will verify that { \it for any $\epsilon>0$, there holds $\mathcal{E}(\hat{w}_i)<4\epsilon$ for $i\gg j\gg1$.}

\textbf{Step 2: Estimates of I and II.}

$\quad$

	 Note that $R_g$ and $Q$ is integrable on  the AF end $E$ of $(M^n,g)$ and $\Lambda_i\to\infty$.
     Then for any $\epsilon>0$, we have
			\begin{align}\label{eq: 14}
			\int_{E\backslash\Omega_i}\beta|R_g-Q|\Lambda_i^{-2}d\mu_g<\epsilon
		\end{align}
		for all sufficiently large $i$. Moreover, by multiplying $w_i$ on both sides of \eqref{eq5} in domain $ \Omega_{ij}$ and integrating by parts, we see that
		\begin{align*}
			\int_{\Omega_{ij}}|\nabla w_i|^2+\beta(R_g-Q) w_i^2=\int_{\partial_+\Omega_{j}} w_i \frac{\partial w_i}{\partial \Vec{n}_g}d\sigma.
		\end{align*}
		By Lemma \ref{lem: asymptotic behavior1} we have
        \begin{align}\label{eq: 60}
		    \sup_{\partial B_r}r^{-1}|w|+|\nabla w|\leq C r^{-\tau-1+\epsilon'}.
		\end{align}
         where $B_r$ denotes the coordinate ball of radius $r$ in $E$ and $\varepsilon'$ can be any positive number.
        Using the smooth convergence $w_i\to w$, we know that for any fixed coordinate sphere $\partial B_r$ and $i$ sufficiently large, it holds
		\begin{align}\label{eq: 51}
		    \sup_{\partial B_r}r^{-1}|w_i|+|\nabla w_i|\leq C r^{-\tau-1+\epsilon}.
		\end{align}
		where $C$ is a constant independent on $i$ and $r$. Therefore, for fixed sufficiently large $j$ and $i\gg j$, there holds
        \begin{align}\label{eq: 19}
        \int_{\Omega_{ij}}|\nabla w_i|^2+\beta(R_g-Q) w_i^2<\epsilon.
        \end{align}

        \begin{figure}
    \centering
    \includegraphics[width = 12cm]{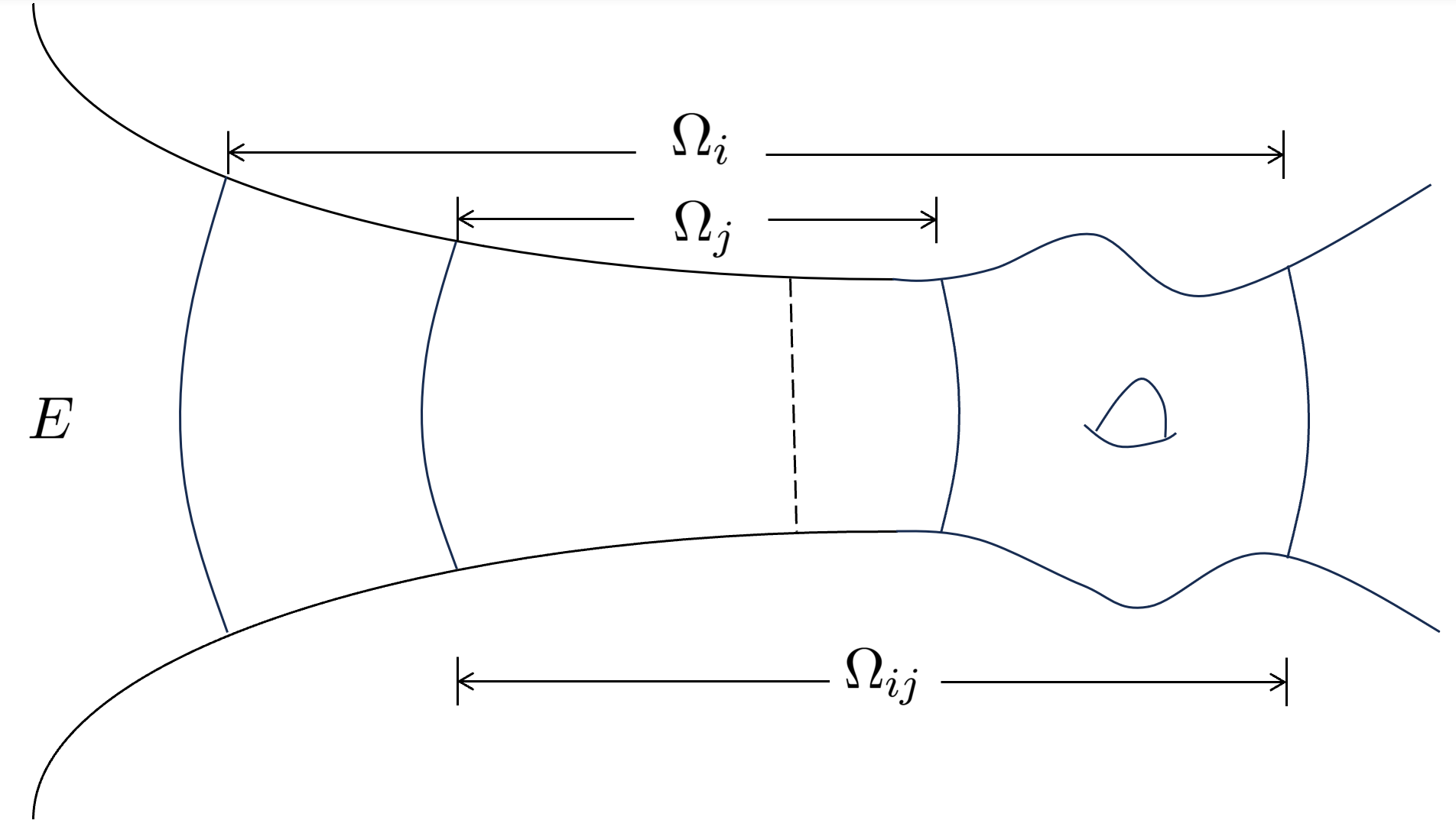}
    \caption{A schematic illustration of the proof of Proposition \ref{prop: eq1-arbitrary end}}
    \label{f3}
\end{figure}

\textbf{Step 3: The estimate of III (modulo Step 4).}

$\quad$

		Finally, it suffices to show that for sufficiently large $j$ and $i\gg j$, there holds
		\begin{align}\label{eq: 18}
			\int_{(\Omega_{i}\setminus\Omega_{j})\cap E }|\nabla w_i|^2+\beta(R_g-Q) w_i^2d\mu_g<2\epsilon.
		\end{align}
		Set $\tilde{w}_i = w_i-\Lambda_i^{-1}$. Then 
		\begin{equation}\label{eq: 16}
			\left\{
			\begin{aligned}
				&-\Delta\tilde{w}_i+\beta(R_g-Q)\tilde{w}_i = -\beta(R_g-Q)\Lambda_i^{-1}\mbox{ in }\Omega_i,\\
				&\tilde{w}_i = 0\mbox{ on }\partial\Omega_i^+.
			\end{aligned}
			\right.	
		\end{equation}
Since $\partial_+\Omega_j = \partial B_{2^j}(O)$, we can find a cut off function $\eta$ with   $|\nabla\eta|<2^{1-j}$ and 			
 \begin{equation}
\eta(x)=\left\{
			\begin{aligned}
&0\quad \text{ for $x$ with  $d(x,E\backslash\Omega_{j-1})>0$,}\\
&1\quad \text{ for $ x \in E\backslash\Omega_{j}$.}\\
\end{aligned}
\right.	\nonumber	
\end{equation}
Then
\begin{align}\label{eq: estimate for nabla eta}
    \int_M|\nabla\eta|^2 = \int_{\Omega_j\backslash\Omega_{j-1}}|\nabla\eta|^2\le C(2^{1-j})^2\cdot (2^j)^n = 4C\cdot (2^j)^{n-2}.
\end{align}		
Multiplying $\eta^2\tilde{w}_i$ on both sides of \eqref{eq: 16} and using
		\[
			\divv(\eta^2\tilde{w}_i\nabla\tilde{w}_i) 
            = \eta^2\tilde{w}_i\Delta\tilde{w}_i+\nabla(\eta^2\tilde{w}_i)\nabla\tilde{w}_i
			=\eta^2\tilde{w}_i\Delta\tilde{w}_i+|\nabla(\eta\tilde{w}_i)|^2-|\nabla\eta|^2\tilde{w}_i^2,
		\]
		we deduce that
		\begin{align*}
			&\int_{(\Omega_{i}\setminus\Omega_ {j-1})\cap E} |\nabla(\eta \tilde{w}_i)|^2+\beta\int_{(\Omega_{i}\setminus\Omega_{ j-1})\cap E }(R_g-Q)\eta^2\tilde{w}_i^2-\int_{(\Omega_{i}\setminus\Omega_{ j-1})\cap E }|\nabla\eta|^2\tilde{w}_i^2\\ 
            =& -\beta\Lambda_i^{-1}\int_{(\Omega_{i}\setminus\Omega_{ j-1})\cap E }(R_g-Q)\eta^2\tilde{w}_i.
		\end{align*}
		Therefore, 
\begin{equation}\label{eq: 17}
		    \begin{split}
		        &\int_{(\Omega_{i}\setminus\Omega_{j})\cap E }|\nabla w_i|^2+\beta(R_g-Q) w_i^2d\mu_g\\\leq&\int_{(\Omega_{i}\setminus\Omega_{ j-1})\cap E }|\nabla(\eta \tilde{w}_i)|^2+\beta|R_g-Q|\eta^2 (\tilde{w_i}+\Lambda_i^{-1})^2d\mu_g\\
                \le& \int_{(\Omega_{i}\setminus\Omega_{ j-1})\cap E }|\nabla(\eta \tilde{w}_i)|^2+2\beta\int_{(\Omega_{i}\setminus\Omega_{ j-1})\cap E }|R_g-Q|(\tilde{w}_i^2 + 1)\\
            \le& \int_{(\Omega_{i}\setminus\Omega_{ j-1})\cap E }|\nabla\eta|^2\tilde{w}_i^2
            +4\beta\int_{(\Omega_{i}\setminus\Omega_{ j-1})\cap E }|R_g-Q|(\tilde{w}_i^2 + 1)\\
            \le& \int_{(\Omega_{j}\setminus\Omega_{ j-1})\cap E }|\nabla\eta|^2\tilde{w}_i^2
            +4\beta\int_{(\Omega_{i}\setminus\Omega_{ j-1})\cap E }|R_g-Q|(\tilde{w}_i^2 + 1),
		    \end{split}
		\end{equation}
	where we have used $\nabla\eta$ is supported on 
    $(\bar{\Omega}_j\setminus\bar{\Omega}_{j-1})\cap E$
    and 
    $\Lambda_i\geq 1$ for $i\gg1$.  
    By \eqref{eq: 51}, we see  that $\tilde{w}_i(x) = O(2^{-j(\tau-\epsilon')})$ for all $x\in\Omega_j\backslash\Omega_{j-1} = B^n_{2^j}(O)\backslash B^n_{2^{j-1}}(O)$ as $i\rightarrow \infty$. Here $\epsilon'$ can be any positive number. Combined with \eqref{eq: estimate for nabla eta}, it follows that
    \begin{align*}
         \int_{(\Omega_{j}\setminus\Omega_{ j-1})\cap E }|\nabla\eta|^2\tilde{w}_i^2\le C\cdot (2^j)^{n-2}\cdot (2^{-j(\tau-\epsilon')})^2 = C\cdot 2^{-j(n-2-2\tau+2\epsilon')}.
    \end{align*}
    By selecting $\epsilon'$ small enough, the above term tends to $0$ as $j\to\infty$. Hence, the first term on the last line of \eqref{eq: 17} can be bounded by $\epsilon$ for  $i\gg j\gg 1$. 
    
    Because $R_g-Q\in L^1(E)$, the estimate for the second term on the last line of \eqref{eq: 17} can be derived from a uniform $C^0$ estimate of $\tilde{w}_i$ in $(\Omega_{i}\setminus\Omega_{ j-1})\cap E$, which we will treat in Step 4.

    \textbf{Step 4: The $C^0$-estimate for $\tilde{w}_i$}

$\quad$
    
 We first observe that $C_1 = \sup_{\partial E} \tilde{w}_i<+\infty$. Let $C_S$ be the Sobolev constant of $(E, g)$. By our assumption, 
we can choose some fixed $R\gg1$ such that
\begin{equation}\label{eq: intergral for R-Q}
  \beta C_S(\int_{ E\setminus B^n_R(O) }|R_g-Q|^{\frac{n}{2}})^{\frac{2}{n}}\leq\frac{1}{4}.  
\end{equation}
Let $0\leq\varphi\leq 1$ be another cut off function supported in $E$ with  $|\nabla\varphi|<2$
and 		
\begin{equation}
\varphi(x)=\left\{
			\begin{aligned}
&1\quad \text{ for $x\in E$ with  $|x|>R+1$,}\\
&0\quad \text{ for $x\in E$ with  $|x|\leq R$.}\\
\end{aligned}
\right.	\nonumber	
\end{equation}	
Multiplying $\varphi^2\tilde{w}_i$ on both sides of \eqref{eq: 16} and integrating by parts gives
    \begin{align*}
			\int_{ E} |\nabla(\varphi\tilde{w}_i)|^2
            +\beta\int_{ E }(R_g-Q)\varphi^2\tilde{w}_i^2-\int_{ E }|\nabla\varphi|^2\tilde{w}_i^2
            = -\beta\Lambda_i^{-1}\int_{ E }(R_g-Q)\varphi^2\tilde{w}_i.
		\end{align*}
    By Sobolev inequality and Holder inequality,
    \begin{equation}
    \begin{split}\label{eq: Sobolov ineq}
        (\int_{ E} |\varphi\tilde{w}_i|^{\frac{2n}{n-2}})^{\frac{n-2}{n}}\leq&C_S\int_{ E} |\nabla(\varphi\tilde{w}_i)|^2\\
    \leq&\beta C_S \Lambda_i^{-1}(\int_{  E\setminus B_R(O) }|R_g-Q|^{\frac{2n}{n+2}})^{\frac{n+2}{2n}}
    (\int_{ E} |\varphi\tilde{w}_i|^{\frac{2n}{n-2}})^{\frac{n-2}{2n}}\\
    &+C_S\int_{ B_{R+1}(O)\backslash B_R(O) }|\nabla\varphi|^2\tilde{w}_i^2
    +\beta  C_S(\int_{ E\setminus B_R(O) }|R_g-Q|^{\frac{n}{2}})^{\frac{2}{n}}
    (\int_{ E} |\varphi\tilde{w}_i|^{\frac{2n}{n-2}})^{\frac{n-2}{n}}.
    \end{split}
    \end{equation}
   By applying Harnack's inequality to \eqref{eq5} satisfied by $w_i$, we obtain $\tilde{w}_i$ is uniformly bounded on the support of $\nabla\varphi$. Combined with \eqref{eq: intergral for R-Q}, \eqref{eq: Sobolov ineq} and using Cauchy-Schwarz inequality yields
   \begin{equation}
       (\int_{ E} |\varphi\tilde{w}_i|^{\frac{2n}{n-2}})^{\frac{n-2}{2n}}\leq C_0
   \end{equation}
    for some uniform $C_0$ independent of $i$.
    Now we can use $L^{\frac{2n}{n-2}}$ estimate and Moser iteration to obtain the uniform $C^0$ bound for $\tilde{w}_i$ in $E\setminus B_R(O)$, concluding the proof of Step 4.
    
    As a consequence of the step by step proof above, the inequality $\mathcal{E}(\hat{w}_i)<4\epsilon$  can be deduced from \eqref{eq: 14}, \eqref{eq: 19} and \eqref{eq: 17}.  Combining  this with \eqref{eq: 15}, we obtain a contradiction, which implies the Claim.  Once Claim  is verified, in conjunction with Harnack inequality (\cite[Theorem 8.20]{GT2001}), we complete the proof of Proposition \ref{prop: eq1-arbitrary end} (1).
\end{proof}
$\quad$
     \begin{proof}[Proof of Proposition \ref{prop: eq1-arbitrary end} (2)]
     
     We assume $(M,g)$ is not scalar flat. Note that the condition of $\beta$-scalar non-negative in the strong spectral sense is equivalent to
        \begin{align}\label{eq: 32}
            \int_M |\nabla\phi|^2 +\beta'\phi^2d\mu_g\ge \int_M(1-\frac{\beta'}{\beta})|\nabla\phi|^2d\mu_g
        \end{align}
        for all locally Lipschitz function $\phi$ which is either compact supported or asymptotically constant in the sense of Definition \ref{defn: strong test function}. Let $\Omega_i$ and $\Omega_{ij}$ be as in (1), then we can  find a smooth function $u_i$ solving following equation:
        \begin{equation}\label{eq: 31}
			\left\{
			\begin{aligned}
				&-\Delta u_i +\beta' R_{ g} u_i=0 \quad \text{in $\Omega_i$},\\
				&u_i|_{\partial_+\Omega_i}=1,\\
				&u_i|_{\partial_-\Omega_i}=0.
			\end{aligned}
			\right.	
		\end{equation}
Next, following the same arguments as those in the proof of the first statement,  we can show that $u_i$ is locally uniformly bounded. 

In fact,  if  $u_i$ is not locally uniformly bounded we can define $w_i$ as before, then $w_i$ converges to a function $w$ solving
 \begin{align}\label{eq: 34}
            -\Delta w+\beta' R_{ g} w = 0
\end{align}
in $M$ and 
$$
\lim\limits _{ x\to\infty, x\in E}w(x) = 0
$$
with   $w(p)=1$ for some fixed $p\in M$. It follows from \eqref{eq: 32} and the same argument as before, we get  \begin{equation}\label{eq: 33}
    \begin{split}
&\int_{M\backslash\Omega_i}\beta'R\Lambda_i^{-2}d\mu_g + \int_{(\Omega_i\backslash\Omega_j)\cap E} |\nabla w_i|^2+\beta'Rw_i^2 d\mu_g
    +\int_{\Omega_{ij}} |\nabla w_i|^2 +\beta'Rw_i^2 d\mu_g\\
    \ge& \int_{\Omega_{ij}} (1-\frac{\beta'}{\beta})|\nabla w_i|^2d\mu_g.
    \end{split}
\end{equation}
By the same argument used in the proof of the first statement, we see that the left-hand side of \eqref{eq: 33} tends to zero as $i\gg j$ tends to infinity. Therefore, the right-hand side of \eqref{eq: 33} tends to zero. Passing to a limit we obtain $w\equiv \mbox{const}$. Plugging into \eqref{eq: 34} we obtain $R_g\equiv 0$, which contradicts our assumption.
 Thus, $\{u_i\}$ has a uniform bound.
 
Let $i\to\infty$ we see $u_i$ converges to the desired function $u$ satisfying the equation in (2) of Proposition \ref{prop: eq1-arbitrary end}.	
	\end{proof}

To apply Proposition \ref{prop: PMT for S^1 symmetric ALF}, we need to construct ALF manifolds with nonnegative scalar curvature. To this end, we consider two separate cases.

    \textbf{Case1: $(M^n, g)$ has $\beta$-scalar curvature no less than $h$ in the strong spectral sense for some $\beta\geq\frac{1}{2}$, where $h$ is a nonnegative smooth function with $h(p)>0$ for some $p\in M$.}
    
    Without loss of generality, we may assume $h\in C_{-2-\tau}(E)$.  By Remark \ref{re: NNSC}, $(M^n, g)$ has $\frac{1}{2}$-scalar curvature no less than $h$. By Proposition 
    \ref{prop: eq1-arbitrary end}, there exists some positive function $u$ satisfying
    \[
    Lu = -\Delta u +\frac{1}{2}R_g u -\frac{1}{4}hu=0\quad \text{in}\ \  M
    \]
     with $u\to 1$ as $x\to\infty$ in $E$.
     Let $ E\subset V_1\subset V_2\subset\dots$ be a sequence of exhausting domains of $M$.
     Since $\lambda_1(-\Delta_g+\frac{1}{2}R_g-\frac{1}{4}h)\ge 0$, from \cite{FcS1980} the solution $v_i$ of the following equation exists:
	\begin{equation}\label{eq: 35}
		\left\{
		\begin{aligned}
			&Lv_{i} = -\Delta v_{i} +\frac{1}{2}R_g v_{i} -\frac{1}{4}hv_{i}=0 \quad \text{in $V_i$},\\
			&\lim\limits _{ x\to\infty, x\in E} v_{i}(x) = 1,\\
			&v_{i}|_{\partial V_i}=0.
		\end{aligned}
		\right.	
	\end{equation}
By Lemma \ref{lem: maximum principle for spectral condition}, the operator $L$ satisfies the maximum principle. Thus we have $v_{i}<u$ for all $x\in M$ and  $v_{j}>v_{i}$ on $V_i$ for each $j>i$. Therefore, $v_{i}$ converges monotonically to a function $v$ in $C^2$ sense, where $v\in C^2(M)$ solves $Lv = 0$ and satisfies  $\lim\limits _{ x\to\infty, x\in E}v(x) = 1$. Moreover,  near the infinity of the AF end $E$ there holds
	\begin{equation}\label{eq: asymp behaviour for $v_i$}
		|v-1|+|x||\partial v| = O(|x|^{-\tau+\epsilon}),
	\end{equation}
	and hence, for any fixed coordinated ball $B_r$ in the AF end $E$, and $i$ large enough,  we have
	\begin{equation}\label{eq: asymp behaviour for $v$}
	    \sup_{\partial B_r}(|v_i-1|+r|\partial v_i|) = O(r^{-\tau+\epsilon}).
	\end{equation}

	Next, we construct a sequence of ALF manifolds as follows
	\begin{equation}\label{eq: 63}
		\begin{split}
			&(\hat{M}_{i},\hat{g}_{i}) = (V_i\times \mathbf{S}^1, g + v_{i}^2ds^2),\\
			&(\hat{M},\hat{g}) = (M\times \mathbf{S}^1, g + v^2ds^2).
		\end{split}
	\end{equation}
	Then we have 
    \[
    R_{\hat{g}}=R_{\hat{g}_i}=R_g-2v_i^{-1}\Delta v_i=\frac{h}{2}\geq0 \ \ \text{in}\ \  V_i.
    \]
    We remind the readers here $(\hat{M}_{i},\hat{g}_{i})$ is not necessarily complete. We use $\hat{E} = E\times \mathbf{S}^1$ to denote the distinguished end in these ALF manifolds. By \eqref{eq: asymp behaviour for $v$} we conclude that $\hat{E}$ is an ALF end of order $\tau-\epsilon$.
	
	The following lemma estimates the mass of the ALF end $(\hat{M}_{i},\hat{g}_{i},\hat{E})$.
	
	\begin{lemma}\label{lem: mass decay ALF}
		\begin{align}\label{eq: 42}
			m_{ADM}(\hat{M}_{i},\hat{g}_{i},\hat{E}) \le (n-1)m_{ADM}(M,g,E) - \frac{1}{4\omega_{n-1}}\int_{V_i}hv_i^2d\mu_{g}.
		\end{align} 
	\end{lemma}
	\begin{proof}
		
		By the definition of the ADM mass for ALF manifolds, we have
		\begin{equation}\label{eq: 36}
			\begin{split}
				m_{ADM}(\hat{M}_{i}, \hat{g}_i, \hat{E}) = &\frac{1}{4\pi\omega_{n-2}}\lim_{\rho\to\infty}\int_{S^{n-1}(\rho)\times \mathbf{S}^1}(\partial_k(\hat{g}_i)_{kj}-\partial_j (\hat{g}_i)_{aa})\nu^jd\sigma_xds\\
				=&\frac{1}{4\pi\omega_{n-1}}\lim_{\rho\to\infty}\int_{S^{n-1}(\rho)\times \mathbf{S}^1}(\partial_k(\hat{g}_i)_{kj}-\partial_j (\hat{g}_i)_{kk}-\partial_jv_i^2)\nu^jd\sigma_xds\\
				=&\frac{1}{2\omega_{n-1}}\lim_{\rho\to\infty}\int_{S^{n-1}(\rho)}(\partial_k(\hat{g}_i)_{kj}-\partial_j (\hat{g}_i)_{kk}-\partial_jv_i^2)\nu^jd\sigma_x\\
				=&(n-1)m_{ADM}(M,g,E)-\frac{1}{\omega_{n-1}}\lim_{\rho\to\infty}\int_{S^{n-1}(\rho)}v_i\frac{\partial v_i}{\partial \Vec{n}_{euc}}d\sigma_x\\
				=&(n-1)m_{ADM}(M,g,E)-\frac{1}{\omega_{n-1}}\lim_{\rho\to\infty}\int_{S^{n-1}(\rho)}v_i\frac{\partial v_i}{\partial \Vec{n}_g}d\sigma_g.\\
			\end{split}
		\end{equation}
		
		We claim the last term is well defined. In fact, for $\rho_1>\rho_2$, we have
		\begin{align*}
			&\int_{S^{n-1}(\rho_1)}v_i\frac{\partial v_i}{\partial \Vec{n}_g}d\sigma_g - \int_{S^{n-1}(\rho_2)}v_i\frac{\partial v_i}{\partial \Vec{n}_g}d\sigma_g\\
			= &\int_{B_{\rho_1}\backslash B_{\rho_2}} (|\nabla v_i|^2+v_i\Delta v_i)d\mu_g\\
			= &\int_{B_{\rho_1}\backslash B_{\rho_2}} (|\nabla v_i|^2+ \frac{1}{2}R_g v_i^2 -\frac{1}{4}hv_i^2)d\mu_g.
		\end{align*}
		Since $|v_{i}(x)-1|+|x||\partial v_{i}(x)| = O(|x|^{-\tau+\epsilon})$  for any $x\in B_{\rho_1}\backslash B_{\rho_2}$ and $R_g,h\in L^1(E)$, we see the last term in \eqref{eq: 36} is well defined.
		
		By the strong spectral PSC condition and \eqref{eq: 35}, we calculate
		\begin{equation}\label{eq: 37}
			\lim_{\rho\to\infty}\int_{S^{n-1}(\rho)}v_i\frac{\partial v_i}
			{\partial \Vec{n}_g}d\mu_{g}
			=\int_{V_i}|\nabla v_i|^2+(\frac{1}{2}R_g-\frac{1}{4}h)v_i^2d\mu_{g}
			\ge \int_{V_i}\frac{1}{4}hv_i^2d\mu_{g}.
		\end{equation}
		The conclusion then follows immediately.
	\end{proof}
	Next, we consider 

    \textbf{Case 2: $(M^n, g)$ has $\beta$-scalar curvature nonnegative  in the strong spectral sense for some $\beta>\frac{1}{2}$.}

    By Proposition 
    \ref{prop: eq1-arbitrary end}, there exists some positive function $u$ satisfying
    \[
    \mathcal{L}u = -\Delta u +\frac{1}{2}R_g u =0\quad \text{in}\ \  M
    \]
     with $u\to 1$ as $x\to\infty$ in $E$.
     Using the similar argument as before, we can solve
		\begin{equation}\label{eq: 38}
			\left\{
			\begin{aligned}
				&\mathcal{L}v_{i} = -\Delta v_{i} +\frac{1}{2}R_g v_{i} = 0 \quad \text{in $V_i$},\\
				&\lim\limits _{ x\to\infty, x\in E} v_{i}(x) = 1,\\
				&v_{i}|_{\partial V_i}=0,
			\end{aligned}
			\right.
		\end{equation}
     where $ E\subset V_1\subset V_2\subset\dots$ is a  sequence of exhausting domains of $M$.
 Also,  $v_{i}$ converges monotonically to a function $v$ in $C^2$ sense, where $v\in C^2(M)$ solves $\mathcal{L}v = 0$ and satisfies  $\lim\limits _{ x\to\infty, x\in E}v(x) = 1$. Moreover,  near the infinity of AF end $E$, $v_i$ and $v$ satisfy similar asymptotic behaviors as 
\eqref{eq: asymp behaviour for $v_i$} and \eqref{eq: asymp behaviour for $v$}
Let $(\hat{M}_{i},\hat{g}_{i},\hat{E})$ and $(\hat{M},\hat{g},\hat{E})$  be constructed as in \eqref{eq: 63}. Then
    \[
    R_{\tilde{g}}=R_{\tilde{g}_i}=R_g-2v_i^{-1}\Delta u_i=0 \ \ \text{in}\ \  V_i.
    \]
  Compared with Lemma \ref{lem: mass decay ALF}, we have
  \begin{lemma}\label{lem: mass decay ALF 2}
        
		\begin{align}\label{eq: 39}
			m_{ADM}(\hat{M}_{i},\hat{g}_{i},\hat{E}) \le (n-1)m_{ADM}(M,g,E) - \frac{2\beta-1}{2\beta\omega_{n-1}}\int_{V_i}|\nabla v_i|^2d\mu_{g}.
		\end{align} 
	\end{lemma}
    \begin{proof}
        Multiplying $v_i$ on both sides of \eqref{eq: 38} and integrating by parts gives
        \begin{equation}\label{eq: integral}
			\lim_{\rho\to\infty}\int_{S^{n-1}(\rho)}v_i\frac{\partial v_i}
			{\partial \Vec{n}_g}d\mu_{g}
			=\int_{V_i}|\nabla v_i|^2+\frac{1}{2}R_gv_i^2d\mu_{g}
			\ge \int_{V_i}(1-\frac{1}{2\beta})|\nabla v_i|^2d\mu_{g},
		\end{equation}
        where we have used \eqref{eq: 32}. Combining \eqref{eq: 36} with \eqref{eq: integral}, we get the desired estimate.
    \end{proof}
    The next Lemma considers the convergence of the ADM mass of $(\hat{M}_i,\hat{g}_i,\hat{E})$ to that of $(\hat{M},\hat{g},\hat{E})$.

    \begin{lemma}\label{lem: convergence ADM mass on ALF}
    Let $(\hat{M}_{i}, \hat{g}_{i}, \hat{E})$  and  $(\hat{M},\hat{g},\hat{E})$ be as in  \eqref{eq: 63}  with $v_i$ being the solution of \eqref{eq: 35} or \eqref{eq: 38}. In both cases, we have
        $$\lim_{i\to\infty}m_{ADM}(\hat{M}_{i}, \hat{g}_{i}, \hat{E}) = m_{ADM}(\hat{M},\hat{g},\hat{E}).$$
    \end{lemma}
    \begin{proof}
        By equations satisfied by $v_{i}$ and $v$ and using a standard Green's function argument, we have $v_{i}\to v$ in $C^0_{-\tau + \epsilon}(E)$. By Schauder estimates we obtain $v_{i}\to v$ in $C^2_{-\tau + \epsilon}(E)$, which shows that $\hat{g}_{i}\to \hat{g}$ in $C^2_{-\tau + \epsilon}(E)$.

        By the calculation of  (4.2) in \cite{Bar86}, we have
        \begin{align*}
            R(\hat{g}_{i})|\hat{g}_{i}|^{\frac{1}{2}} = \partial_p\left((\hat{g}_{i})_{pq,q}-(\hat{g}_{i})_{qq,p}+Q_1(\hat{g}_{i})\right)+Q_2(\hat{g}_{i}),
        \end{align*}
        where
        \begin{align*}
            &g_{pq,q}-g_{qq,p}+Q_1(g) = |g|^{\frac{1}{2}}g^{pq}(\Gamma_q-\frac{1}{2}\partial_q(log|g|)),\\
            &Q_2(g)|g|^{-\frac{1}{2}} = -\frac{1}{2}g^{pq}\Gamma_p\partial_q(log|g|)+g^{pq}g^{kl}g^{rs}\Gamma_{pkr}\Gamma_{qls},\\
            &\Gamma_{ijk} = \frac{1}{2}(g_{jk,i}+g_{ik,j}-g_{ij,k})\mbox{ and }\Gamma_k = g^{ij}\Gamma_{ijk}.
        \end{align*}
        In our case, it is straightforward to see $Q_1(\hat{g}_{i}) = O(r^{-2\tau-1+2\epsilon})$, $Q_2(\hat{g}_{i}) = O(r^{-2\tau-2+2\epsilon})$, and
        \begin{align*}
            &m_{ADM}(\hat{M}_{i}, \hat{g}_{i}, \hat{E})\\
            =&\frac{1}{4\pi\omega_{n-1}}\lim_{\rho\to\infty} \int_{S^{n-2}(\rho)\times \mathbf{S}^1}(\hat{g}_i)_{pq,q}-(\hat{g}_{i})_{qq,p}d\sigma_xds\\
            =&\frac{1}{4\pi\omega_{n-1}}\int_{S^{n-2}(\rho_0)\times \mathbf{S}^1}(\hat{g}_{i})_{pq,q}-(\hat{g}_{i})_{qq,p}+Q_1(\hat{g}_{i})d\sigma_xds +\frac{1}{2\omega_{n-1}}\int_{\rho>\rho_0}R_{\hat{g}_{i}}|\hat{g}_{i}|^{\frac{1}{2}}-Q_2(\hat{g}_{i})dx.
        \end{align*}
        By the dominate convergence theorem we have
        \begin{align*}
            &(\hat{g}_{i})_{pq,q}-(\hat{g}_{i})_{qq,p}+Q_1(\hat{g}_{i})\to \hat{g}_{pq,q}-\hat{g}_{qq,p}+Q_1(\hat{g}_{}),\\
            &\int_{\rho>\rho_0}Q_2(\hat{g}_{i})dx\to\int_{\rho>\rho_0}Q_2(\hat{g}_{})dx \mbox{ as }i\to\infty.
        \end{align*}
        From \eqref{eq: 35} and \eqref{eq: 38}, in both cases we have $R_{\hat{g}_{i}} = R_{\hat{g}}$ in $E$, which gives the desired conclusion.
        \end{proof}
        As the final step of the preparation, we need the following conformal deformation theorem that applies to the strong spectral PSC condition.

       \begin{lemma}\label{lem: make a point strictly positive}
    Let $(M^n,g)$ be a smooth AF manifold with arbitrary ends and let $E$ be its AF end. Suppose $(M^n,g,E)$ has negative mass and $\frac{1}{2}$-scalar curvature non-negative  in the strong spectral sense. Then there exists $\varphi\in C^{\infty}(M)$, such that $(M^n,g_1:=\varphi^{\frac{4}{n-2}}g)$ is also a complete manifold with the AF end $E$ and negative mass, and for any neighborhood $\mathcal{U}$ of $E$ such that $\mathcal{U}\Delta E$ is compact, any $\phi\in C^1(\mathcal{U})$ with 
    \[
    \phi|_{\partial \mathcal{U}}=0, \ \ \lim_{x\to\infty,x\in E}\phi = 1\ \  \text{and} \ \ 
    \phi-1\in   W^{1,2}_{-q}(\mathcal{U})\ \ \text{for some}\ \  q> \frac{n-2}{2},
    \]
    it holds
     \begin{align}\label{eq: 30}
        \int_{M}|\nabla_{g_1}\phi|^2+\frac{1}{2}R_{g_1}\phi^2d\mu_{g_1}\ge \int_{M}\frac{1}{2}h\phi^2d\mu_{g_1}
    \end{align}
    for some $h\in C^{\infty}(M)$ satisfying $h\ge 0$ everywhere and $h(p)>0$ at some point $p$.
\end{lemma}

\begin{proof}
    We first construct a positive function $u\in C^{\infty}(M)$ such that
    \begin{equation}\label{eq: 27}
			\left\{
			\begin{aligned}
				&\Delta_gu\le 0 \text{ for all $x\in M$},\\
				&\Delta_gu(p)<0 \text{ at some point $p\in M$},\\
				&\lim\limits _{x\to\infty, x\in E} u(x) = 1.
			\end{aligned}
			\right.	
		\end{equation}
    Choose $c_S$ to be the constant that depends on $E$ as in \cite[Lemma 2.1]{Zhu23}. Then, we can construct a nonpositive function $f$ with compact support in $E$, such that $f(p)<0$ at some point $p\in E$, and
    \begin{align*}
        (\int_E|f|^{\frac{n}{2}})^{\frac{2}{n}}\le\frac{c_S}{2}.
    \end{align*}
    By \cite[Proposition 2.2]{Zhu23}, there is a positive function $u\in C^{\infty}(M)$ that satisfies $\Delta_gu = fu\le 0$ for all $x\in M$ and $\Delta_gu(p) = f(p)u(p)<0$. Moreover, From Lemma \ref{lem: asymptotic behavior2} we know that $u$ has the expansion 
    \begin{align}\label{eq: 28}
        u(x) = 1+A|x|^{2-n} + o(|x|^{2-n}),
    \end{align}
    where
    \begin{align}\label{eq: 29}
        A = -\frac{1}{(n-2)\omega_{n-1}}\int_E fud\mu_g >0.
    \end{align}
    It follows that $u$ satisfies \eqref{eq: 27}.

    Let $\varphi_{\epsilon} = \frac{1+\epsilon u}{1+\epsilon}$ with $\varepsilon>0$ to be determined. Then $(M^n,g_{\epsilon}) = (M^n,\varphi_{\epsilon}^{\frac{4}{n-2}}g)$ is complete. For any neighborhood $\mathcal{U}$ of $E$ such that $\mathcal{U}\Delta E$ is compact, any $\phi\in C^1(\mathcal{U})$ with 
    \[
    \phi|_{\partial \mathcal{U}}=0,\ \ \lim_{|x|\to\infty,x\in E}\phi = 1 
    \ \ \text{and}\ \ \phi-1\in   W^{1,2}_{-q}(\mathcal{U})\ \ \text{for some}\ \  q> \frac{n-2}{2},
    \]
    we calculate
    \begin{align*}
        &\int_M|\nabla_{g_\epsilon}\phi|^2+\frac{1}{2}R_{g_\epsilon}\phi^2d\mu_{g_\epsilon}\\
        =&\int_M \varphi_{\epsilon}^2|\nabla_{g}\phi|^2+\frac{1}{2}\varphi_{\epsilon}(R_{g}\varphi_{\epsilon}-\frac{4(n-1)}{n-2}\Delta_g\varphi_{\epsilon})\phi^2d\mu_{g}\\
        =&\int_M |\nabla_{g}(\varphi_{\epsilon}\phi)|^2+\frac{1}{2}R_g(\varphi_{\epsilon}\phi)^2 d\mu_g- \lim_{\rho\to\infty} \int_{\partial B_{\rho}}\phi^2\varphi_{\epsilon}\frac{\partial \varphi_{\epsilon}}{\partial \Vec{n}}d\sigma_g -\int_M \frac{2(n-1)}{n-2}\varphi_{\epsilon}\phi^2\Delta_g \varphi_{\epsilon}d\mu_g\\
        \ge& -\int_M \frac{2(n-1)}{n-2}\phi^2\varphi_{\epsilon}^{-\frac{n+2}{n-2}}\Delta_g \varphi_{\epsilon}d\mu_g,
    \end{align*}
    where in the last inequality we have used that $(M^n,g)$ has $\frac{1}{2}$-scalar curvature
    non-negative  in strong spectral condition and \eqref{eq: 29}. It follows from \eqref{eq: 27} that $(M^n,g_1:=g_{\varepsilon})$ satisfies \eqref{eq: 30}. Furthermore, when $\epsilon$ is sufficiently small, we have
    \begin{align*}
        m_{ADM}(M^n,g_1,E) 
        =&m_{ADM}(M^n,g,E)+\frac{2A\epsilon}{1+\epsilon} < 0.
    \end{align*}
    This completes the proof.
\end{proof}

\begin{proof}[Proof of Theorem \ref{thm: PMT arbitrary end spectral}]

    Let us  assume $h(p)>0$ first. By Proposition \ref{prop: eq1-arbitrary end}, we can solve \eqref{eq: 35} and construct the ALF manifolds as described in \eqref{eq: 63}. Without loss of generality, we assume $p\in V_1$. It follows from Lemma \ref{lem: mass decay ALF}  and the monotonicity of $v_i$ with respect to $i$ that
    \begin{align*}
        (n-1)m_{ADM}(M,g,E)\ge& m_{ADM}(\hat{M}_{i},\hat{g}_{i},\hat{E})+\frac{1}{4\omega_{n-1}}\int_{V_i}hv_i^2d\mu_{g}\\
        \ge &m_{ADM}(\hat{M}_{i},\hat{g}_{i},\hat{E})+\frac{1}{4\omega_{n-1}}\int_{V_1}hv_1^2d\mu_{g}\\
        \ge &m_{ADM}(\hat{M}_{i},\hat{g}_{i},\hat{E})+\sigma_0,
    \end{align*}
    where $\sigma_0>0$ is a positive number. The conclusion  then follows from Lemma \ref{lem: convergence ADM mass on ALF} and the positive mass theorem for ALF manifolds (Proposition \ref{prop: PMT for S^1 symmetric ALF} and \cite[Theorem 1.8]{CLSZ2021}).

     For the general case, we can use Lemma \ref{lem: make a point strictly positive} to reduce the case that $h(p)>0$. To be more precise, if $m_{ADM}(M,g,E)<0$, then by utilizing Lemma \ref{lem: make a point strictly positive} we can find another metric $g_1$ on $M$ such that $m_{ADM}(M,g_1,E)<0$ and $(M,g_1)$ satisfies the condition that $h(p)>0$. This is a contradiction.

    Next, let us assume $\beta>\frac{1}{2}$. It suffices to prove the rigidity of $(M,g)$ under the assumption that $m_{ADM}(M,g,E) = 0$. We first show that $(M,g)$ is scalar flat. If it is not the case, by Proposition \ref{prop: eq1-arbitrary end}, we can solve \eqref{eq: 38} and construct a sequence of ALF manifolds as described in \eqref{eq: 63}. Therefore, by applying Lemma \ref{lem: mass decay ALF 2} we see that
    \begin{align*}
        m_{ADM}(\hat{M}_{i},\hat{g}_{i},\hat{E}) \le (n-1)m_{ADM}(M,g,E) - \frac{2\beta-1}{2\beta\omega_{n-1}}\int_{V_i}|\nabla v_i|^2d\mu_{g}.
    \end{align*}
    Since $(M,g)$ is not scalar flat, the limit function $v$ of $v_i$ is not a constant. Consequently, we have     
    $$\int_{V_i}|\nabla v_i|^2d\mu_{g}\geq a>0,$$ 
    where $a$ is a  constant independent of $i$.  Letting $i\to\infty$ and using Lemma \ref{lem: convergence ADM mass on ALF} and the positive mass theorem for ALF manifold (Proposition \ref{prop: PMT for S^1 symmetric ALF} and \cite[Theorem 1.8]{CLSZ2021}), we get a contradiction. Therefore, $(M,g)$ is scalar flat. The conclusion then follows directly from \cite[Theorem 1.2]{Zhu23}.
\end{proof}

\section{Deformation and warped construction of minimal hypersurfaces with singular set}
In this section, we will perform suitable deformation on minimal hypersurfaces with singular set to get smooth manifolds with desired properties.
Firstly, we prove the following Lemma for area-minimizing hypersurfaces with non-empty singular set.	
	
        \begin{lemma}\label{lem: |A|>0}
            Let $(M^{n+1},g)$ be a Riemannian manifold and $\Sigma$ be an area-minimizing hypersurface with singular set $\mathcal{S}$ such that $\Sigma$ is smoothly embedded outside $\mathcal{S}$. Suppose $\mathcal{S}\ne\emptyset$, then for any $r>0$, there exists a point $p\in \mathcal{B}_r^{n+1}(\mathcal{S})\cap reg\Sigma$ that satisfies $|A|^2(p)>0$. Here $\mathcal{B}_r^{n+1}(\mathcal{S}) = \{x\in M, d_M(x,\mathcal{S})<r\}$.
        \end{lemma}
        \begin{proof}
            The proof follows from analyzing nontrivial tangent cone at singular set, which has also been carried out in \cite{CLZ24}. Here we just include a proof for completeness.
            
            First, due to \cite[Theorem 27.8]{Simon83}, we deduce that for each $x\in\Sigma$, there exists an open set $W$ containing $x$, such that $\Sigma\cap W$ is an area-minimizing boundary in $W$. Thus, we can reduce the problem to the case of area-minimizing boundaries. We isometrically embed a compact domain of $(M,g)$  which has non-empty intersection with $\mathcal{S}$ into an Euclidean space $\mathbf{R}^{n+k}$, and assume the conclusion is not true, then $|A|=0$ in a neighborhood of $\mathcal{S}$. It follows from standard dimension reduction \cite[Appendix A]{Simon83} that at some point $q\in\mathcal{S}$, the tangent cone $C_q\subset T_qM$ splits as $\mathcal{C}^d\times \mathbf{R}^{n-d+k}$, where $\mathcal{C}^d$ is a nontrivial minimizing cone in $\mathbf{R}^{d+1}$ with isolated singularity. 

            Denote $\Sigma_i = \eta_{q,\lambda_i}\Sigma$ $(\lambda_i\to 0)$ to be a blow up sequence at $q$ (Here $\eta_{q,\lambda_i}(x) = \lambda_i^{-1}(x-q)$). Then by \cite[Theorem 34.5]{Simon83} $\Sigma_i$ converges to $C_q$ in multiplicity one in varifold sense, as  $\Sigma$ is an area-minimizing boundary. At the same time $M_i = \eta_{q,\lambda_i,}M$ converges to $T_qM$ in $C^1$ sense. The monotonicity formula implies the convergence also holds in Hausdorff sense. Combined with Allard's regularity theorem and standard results in PDE theory, we know that the regular part of $\Sigma_i$ converges to the regular part of $C_q$ in $C^2$. This implies the second fundamental form of $C_q\subset T_qM$ is identically zero on its regular part, which is a contradiction.
        \end{proof}
      Since an area-minimizing hypersurface must be stable, by the second variation  of area functional and Gauss equation we have
\begin{lemma}\label{lem: 2nd variation}
    Let $(M^{n+1},g)$ be a Riemannian manifold and $\Sigma$ be an area-minimizing hypersurface with singular set $\mathcal{S}$ such that $\Sigma$ is smoothly embedded outside $\mathcal{S}$. Then
        for any function $\phi\in C^{\infty}_{c}(\Sigma\backslash\mathcal{S})$, there holds
	\begin{equation}\label{eq:  stability}
		\int_{\Sigma\backslash\mathcal{S}}|\nabla_{g_{\Sigma\setminus\mathcal{S}}}\phi|^2+\frac{1}{2}R_{g_{\Sigma\setminus\mathcal{S}}}\phi^2\geq\int_{\Sigma\backslash\mathcal{S}}\frac{1}{2}(R_g+|A|^2)\phi^2. 
	\end{equation}
\end{lemma}
        In the rest of this section, we denote  $R_g+|A|^2$ by  $h$ for simplicity. Then $h$ is nonnegative if $R_g\geq 0$ along $\Sigma$ and $h$ is positive somewhere if we further  assume the singular set $\mathcal{S}\neq\emptyset$.  We use $\nabla$, $\Delta$ and  $d\mu_{g_{\Sigma\backslash\mathcal{S}}}$  to denote the Levi-Civita connection, Laplace operator and the area element with respect to the metric $g_{\Sigma\setminus\mathcal{S}}$. 
\subsection{Blow up-warped construction of area-minimizing hypersurface with singularities in asymptotically flat manifold}

$\quad$

In this subsection, we assume $(M^{n+1},g) $ is a complete AF manifold with an AF end $E$ of order $\tau>n-2$ and some arbitrary ends. Each point $x\in E\cong\mathbf{R}^{n+1}\backslash B_1(O)$ has a Euclidean coordinate $x = (x_1,x_2,\dots,x_n,z) = (x',z)$. We call $z$ the height coordinate. We have the following definition:
\begin{definition}
    An area-minimizing hypersurface $\Sigma$ is said to be asymptotic to the hyperplane $\mathbf{R}^n\times \{z_0\}$, if $\Sigma$ can be written as a graph over $\mathbf{R}^n$ with graph function $u$ outside some compact set, such that
    \begin{align*}
        \lim_{|x'|\to\infty} |u(x')-z_0| = 0.
    \end{align*}
\end{definition}

\begin{remark}\label{rmk: asymptotic to a hyperplane}
    From the above definition, if $\Sigma$ is asymptotic to the hyperplane $\mathbf{R}^n\times \{z_0\}$, then $\Sigma$ lies between two parallel coordinate hyperplanes. Thus, the tangent cone of $\Sigma$ at the infinity is a unique plane, so $\Sigma$ satisfies the conclusion of \cite[Lemma 19]{EK23}. Therefore, the proof of \cite[Proposition 8]{EK23} implies that $u$ has the following decay estimate
    \begin{align}\label{eq: decay estimate of u}
        |u(x')-z_0|+|x'||Du(x')|+|x'|^2|D^2u(x')| = O(|x'|^{-\tau+\epsilon})
    \end{align}
    for any sufficiently small $\epsilon>0$.

    Consequently, with the induced metric, $\Sigma\cap E$ is itself an AF end of order $\tau-\epsilon$ for any $\epsilon>0$. If $\tau>n-2$, then \cite[Corollary 20]{EK23} implies the $|A|^2$, $Ric(\nu,\nu)$ and therefore the intrinsic scalar curvature $R_{\Sigma\backslash\mathcal{S}}$ is integrable on the AF end $\Sigma\cap E$, so the mass of $\Sigma\cap E$ is well defined. Furthermore, when $\tau>n-2$, similar to \cite[Lemma 21]{EK23} we know the mass of $\Sigma\cap E$ is zero.
\end{remark}

Now let us recall the notion of \textit{strongly stable}.
\begin{definition}\label{defn: strongly stable hypersurface}
    Let $\Sigma$ be  an area-minimizing hypersurface in $M^{n+1}$ which is  asymptotic to a  coordinate hyperplane in the AF end $E$ of $M^{n+1}$ with $\Sigma\cap E$ being an AF end for $\Sigma$. We say $\Sigma$ is  strongly stable, if for any function $\phi\in \Lip_{loc}(\Sigma\backslash\mathcal{S})$ which is either compactly supported or asymptotically constant in the sense of Definition \ref{defn: strong test function}, there holds
	\begin{equation}\label{eq: strongly stability}
		\int_{\Sigma\backslash\mathcal{S}}|\nabla\phi|^2-(Ric(\nu,\nu)+|A|^2)\phi^2\geq 0. 
	\end{equation}
    Here $\nu$ is the unit normal vector of $\Sigma^{n}$ in $(M^{n+1},g)$.
\end{definition}
In the rest of this subsection, we always assume $\Sigma^n$ is a strongly stable area-minimizing boundary in $M^{n+1}$ with isolated singular set  $\mathcal{S}$, and $R_g\ge 0$ along $\Sigma$.  We also make the following assumption throughout this subsection:
\begin{align}\label{eq: assumption in sec 5}
    h := R_g+|A|^2 \text{ is strictly positive at some point } p\in\Sigma\backslash\mathcal{S}.
\end{align}
From Lemma \ref{lem: |A|>0}, the assumption \eqref{eq: assumption in sec 5} holds when $\mathcal{S}\ne\emptyset$. By Gauss equation,  $\Sigma^n\setminus\mathcal{S}$ has $\frac{1}{2}$-scalar curvature not less than $h=R_g+|A|^2$ in the strong spectral sense. 

Let $G$ be a positive singular harmonic function defined on $\Sigma\backslash\mathcal{S}$ as in Proposition \ref{lem: extend harmonic function 2}. Recall that for each $p_i\in \mathcal{S}$ it holds
      \[
      G(x)\ge C_id_M(x,p_i)^{2-n}
        \ \ \text{for}\ \  d_M(x,p_i)\ \   \text{small enough}.
      \]
Since the intrinsic distance $d_\Sigma$ is larger than the extrinsic distance $d_M$ \footnote{Recall in Section 3, we worked with the extrinsic distance $d_M$ throughout the section. For the sake of obtaining the completeness of the conformal deformed metric on $\Sigma\backslash\mathcal{S}$ in Lemma \ref{lem: completeness}, we have to translate the estimate for $G$ to an intrinsic distance version.}, we have
\begin{align}\label{eq: Green estimate d Sigma}
    G(x)\ge C_id_\Sigma(x,p_i)^{2-n}
        \ \ \text{for}\ \  d_\Sigma(x,p_i)\ \   \text{small enough}.
\end{align}
By Lemma \ref{lem: asymptotic behavior2},  $G$ has an expansion at infinity of the AF end $\Sigma\cap E$(for any $\varepsilon'>0$) 
\begin{equation}\label{eq: 25}
G = a|x|^{2-n}+\omega \ \ \text{with}
\ \ |\omega| +|x||\partial\omega| = O(|x|^{1-n})+O(r^{-\tau+2-n+\varepsilon'})\ \ \text{as} \quad |x|\to\infty
	\end{equation}
For any $\delta>0$, let $G_\delta(x):=1+\delta G$, then $G_\delta$ is also a  positive singular harmonic function in $\Sigma\setminus \mathcal{S}$. Let $ g_\delta:=G_\delta^{\frac{4}{n-2}}g_{\Sigma\setminus\mathcal{S}}$. We collect basic properties of the conformally deformed manifolds $(\Sigma\backslash\mathcal{S},g_{\delta})$.
\begin{lemma}\label{lem: completeness}
		$(\Sigma\setminus\mathcal{S},g_{\delta})$ is a complete manifold with arbitrary ends and an AF end $\Sigma\cap E$ of the order $\tau-\varepsilon$ for any $\epsilon>0$.
        \end{lemma}
	\begin{proof}
 For simplicity, we may assume that the singular set $\mathcal{S} = \{p_i\}_{i = 1,2,\dots}$. Let 
\begin{align}\label{eq: 72}
    U_r: = \{x\in \Sigma\setminus\mathcal{S}: d_\Sigma (x, \mathcal{S})<r\},
\end{align}
     Here we use the intrinsic distance function $d_\Sigma$ on $\Sigma$. Now for any small $r_1$, $r_2$ with $0<r_1 <r_2$, choose an arbitrary unit speed curve (with respect to the metric $g_{\Sigma\setminus\mathcal{S}}$) $\gamma:[0,l]\longrightarrow U_{r_2}\backslash U_{r_1}$, $\gamma(0)\in \partial U_{r_1}, \gamma(l)\in \partial U_{r_2}$. From \eqref{eq: 72} we know that $d_\Sigma(\gamma(l),p_k) = r_2$ for some $k\in \{1,2,\dots\}$.
     
      Denote $D(t) = d_\Sigma(\gamma(t),p_k)$. The triangle inequality yields $|D'(t)|\le 1$ almost everywhere. 
      We calculate
        \begin{align*}
            \mbox{length}_{g_{\delta}}(\gamma) 
            = &\int_0^l G_{\delta}^{\frac{2}{n-2}}(\gamma(t))|\gamma'(t)|_{g_{\Sigma\setminus\mathcal{S}}}dt 
            \ge \int_0^l (C_k\delta)^{\frac{2}{n-2}}D(t)^{-2}|D'(t)|dt\\
            \ge& (C_k\delta)^{\frac{2}{n-2}}\int_0^l D(t)^{-1}|D'(t)|dt
            \ge(C_k\delta)^{\frac{2}{n-2}}\ln(\frac{d_{g_{\Sigma\setminus\mathcal{S}}}(\gamma(l),p_k)}{d_{g_{\Sigma\setminus\mathcal{S}}}(\gamma(0),p_k)}) \ge(C_k\delta)^{\frac{2}{n-2}}\ln(\frac{r_2}{r_1}),
        \end{align*}
         where in the first inequality we used \eqref{eq: Green estimate d Sigma}. This shows $d_{g_{\delta}}(\partial U_{r_1},\partial U_{r_2})\ge (C_k\delta)^{\frac{2}{n-2}}\ln(\frac{r_2}{r_1})$, which immediately implies the completeness. The asymptotical order of $\Sigma\cap E$ follows from   Remark \ref{rmk: asymptotic to a hyperplane} 
         and \eqref{eq: 25}.    
        \end{proof}

The next proposition shows nonnegative scalar curvature in the (strong) spectral sense is well preserved under conformal transformation by harmonic functions.         

\begin{proposition}\label{conformal deformation1}
		 $(\Sigma\setminus\mathcal{S},g_{\delta})$ has $\frac{1}{2}$-scalar curvature no less than $hG_\delta^{-\frac{4}{n-2}}$ in the strong spectral sense, \textit{i.e.} for any $\phi\in Lip(\Sigma\setminus\mathcal{S})$ with compact support set or approaching to a constant at the infinity of the AF end $\Sigma\cap E$, there holds
        \begin{equation}\label{eq:spectrum psc1}
\int_{\Sigma\setminus\mathcal{S}}(|\nabla_{g_\delta} \phi|_{g_\delta}^2+\frac{1}{2} R_{g_\delta}\phi^2)d\mu_{g_\delta}\geq \int_{\Sigma\setminus\mathcal{S}}\frac{1}{2} hG_\delta^{-\frac{4}{n-2}}\phi^2 d\mu_{g_\delta}.
\end{equation}
	\end{proposition}
	\begin{proof}
		 Owing to the well-known formula, we know that
		\begin{align}\label{eq: 48}
		    R_{g_\delta}=G_\delta^{-\frac{n+2}{n-2}}(R_{g_{\Sigma\setminus\mathcal{S}}} \cdot
		G_\delta-\frac{4(n-1)}{n-2}\Delta G_\delta)=G_\delta^{-\frac{4}{n-2}}\cdot R_{g_{\Sigma\setminus\mathcal{S}}}.
		\end{align}
		Therefore, for $\phi$ with compact support in $\Sigma\setminus\mathcal{S}$, we have
		\begin{equation}
			\begin{split}
				&\quad\int_{\Sigma\setminus\mathcal{S}}(|\nabla_{g_{\delta}} \phi|^2+\frac{1}{2}  R_{g_{\delta}} \phi^2)d\mu_{g_{\delta}}	\\
				&=\int_{\Sigma\setminus\mathcal{S}}(G_{\delta}^2|\nabla \phi|^2 +\frac{1}{2} R_{g_{\Sigma\setminus\mathcal{S}}} (G_{\delta}\phi)^2)d\mu_{g_{\Sigma\backslash\mathcal{S}}}\\
				&=\int_{\Sigma\setminus\mathcal{S}}(G_{\delta}^2|\nabla \phi|^2 +\frac{1}{2} R_{g_{\Sigma\setminus\mathcal{S}}} (G_{\delta}\phi)^2)d\mu_{g_{\Sigma\backslash\mathcal{S}}}-\int_{\Sigma\setminus\mathcal{S}}(\phi^2 G_\delta \Delta G_\delta)d\mu_{g_{\Sigma\backslash\mathcal{S}}}\\
				&=\int_{\Sigma\setminus\mathcal{S}}(G_{\delta}^2|\nabla \phi|^2 +\frac{1}{2} R_{g_{\Sigma\setminus\mathcal{S}}} (G_{\delta}\phi)^2)d\mu_{g_{\Sigma\backslash\mathcal{S}}}+\int_{\Sigma\setminus\mathcal{S}}(\phi^2|\nabla
				G_\delta|^2+2 \phi G_\delta \nabla G_\delta \cdot \nabla\phi)d\mu_{g_{\Sigma\backslash\mathcal{S}}} \\
            &=	\int_{\Sigma\setminus\mathcal{S}}(|\nabla (G_\delta\phi)|^2 +\frac{1}{2} R_{g_{\Sigma\setminus\mathcal{S}}} (G_{\delta}\phi)^2)d\mu_{g_{\Sigma\backslash\mathcal{S}}}.
			\end{split}
		\end{equation} 
		Now, we consider the case that $\phi$ approaches to a constant at the infinity of the AF end $\Sigma\cap E$. In this case, we have	
		\begin{align*}
			&\int_{\Sigma\setminus\mathcal{S}}(|\nabla_{g_{\delta}} \phi|^2+\frac{1}{2}  R_{g_{\delta}} \phi^2)d\mu_{g_{\delta}}\\
			=&\int_{\Sigma\setminus\mathcal{S}}G_{\delta}^2|\nabla\phi|^2+\frac{1}{2} R_{g_{\Sigma\setminus\mathcal{S}}}(G_{\delta}\phi)^2d\mu_{g_{\Sigma\backslash\mathcal{S}}}\\
			=&\int_{\Sigma\setminus\mathcal{S}}|\nabla (G_{\delta}\phi)|^2+\frac{1}{2} R_{g_{\Sigma\setminus\mathcal{S}}}(G_{\delta}\phi)^2d\mu_{g_{\Sigma\backslash\mathcal{S}}} - \int_{\Sigma\setminus\mathcal{S}}\divv_g(\phi^2G_{\delta}\nabla G_{\delta})d\mu_{g_{\Sigma\backslash\mathcal{S}}}\\
			\ge& \int_{\Sigma\setminus\mathcal{S}}\frac{1}{2} h(G_{\delta}\phi)^2d\mu_{g_{\Sigma\backslash\mathcal{S}}}-\lim_{\rho\to +\infty}\int_{\partial B_{\rho}}\phi^2 G_{\delta}\frac{\partial G_{\delta}}{\partial \Vec{n}} d\sigma.
		\end{align*}
		Since $G$ is a positive singular harmonic function that tends to $0$ at the infinity of the AF end $E$, by applying the maximum principle, we know  the expansion coefficient $a$ in \eqref{eq: 25} is positive. It follows that
		\begin{align*}
			\lim_{\rho\to +\infty}\int_{\partial B_{\rho}}\phi^2 G_{\delta}\frac{\partial G_{\delta}}{\partial \Vec{n}} d\sigma = (2-n)a\delta\le 0.
		\end{align*}
		Thus, in both cases, we get the desired inequality.
	\end{proof}

By Proposition \ref{prop: eq1-arbitrary end} and the paragraph behind it, for a sequence of  exhausting  domains $ (E\cap \Sigma)\subset V_1\subset V_2\subset\dots$ of $\Sigma\setminus\mathcal{S}$, we can solve the following equations:
	\begin{equation}\label{eq: 43}
		\left\{
		\begin{aligned}
&\bar{L}_{\delta}v_{i,\delta} = -\Delta_{g_{\delta}} v_{i,\delta} +\frac{1}{2}R_{g_{\delta}} v_{i,\delta} -\frac{1}{4}hG_{\delta}^{-\frac{4}{n-2}}v_{i,\delta}=0 \quad \text{in $V_i$},\\
&\lim\limits _{ x\to\infty, x\in E\cap \Sigma} v_{i,\delta}(x) = 1,\\
&v_{i,\delta}|_{\partial V_i}=0.
		\end{aligned}
		\right.	
	\end{equation}
  For any fixed coordinated ball $B_r$ in the AF end $\Sigma\cap E$, and $i$ large enough,  we have
	$$
	 \sup_{\partial B_r}(|v_{i,\delta}-1|+r|\partial v_{i,\delta}|) = O(r^{-\tau+\epsilon}).
	 $$
Moreover, we have $v_{j,\delta}>v_{i,\delta}$ on $V_i$ for each $j>i$, and $v_{i,\delta}$ converges monotonically to a function $v_{\delta}\in C^2(\Sigma\setminus \mathcal{S})$ in $C^2$ sense with  $\bar{L}_{\delta}v_{\delta} = 0$ and  $|v_{\delta}(x)-1|+|x||\partial v_{\delta}(x)| = O(|x|^{-\tau+\epsilon})$  for $x$ near the infinity of $\Sigma\cap E$.

Now, we can follow the argument in Section 4 to construct a sequence of ALF manifolds with 
nonnegative scalar curvature and  arbitrary ends by
\begin{align*}
		&(\hat{\Sigma}_{i,\delta},\hat{g}_{i,\delta}) = (V_i\times \mathbf{S}^1, g_{\delta} + v_{i,\delta}^2ds^2),\\
		&(\hat{\Sigma}_{\delta},\hat{g}_{\delta}) = ((\Sigma\backslash\mathcal{S})\times \mathbf{S}^1, g_{\delta} + v_{\delta}^2ds^2),
	\end{align*}
where $s\in \mathbf{S}^1$. 
Note that $(\hat{\Sigma}_{\delta},\hat{g}_{\delta})$  is complete but $(\hat{\Sigma}_{i,\delta},\hat{g}_{i,\delta})$ is not necessarily complete.
Without loss of generality, we assume $p\in V_1$ such that $h>0$ in a small neighborhood $\Omega$ of $p$, where $p$ comes from the assumption \ref{eq: assumption in sec 5}. Note that
\begin{align*}
    &R_{\hat{g}_{i,\delta}} = R_{g_\delta}-\frac{2\Delta_{g_\delta}v_{i,\delta}}{v_{i,\delta}} = \frac{1}{2}hG_\delta^{-\frac{4}{n-2}}\ge 0, \text{ for }x\in V_i,\\
    &R_{\hat{g}_{\delta}} = R_{g_\delta}-\frac{2\Delta_{g_\delta}v_{\delta}}{v_{\delta}} = \frac{1}{2}hG_\delta^{-\frac{4}{n-2}}\ge 0, \text{ for }x\in \Sigma\backslash\mathcal{S}.\\
\end{align*}
Thus, $R_{\hat{g}_{i,\delta}}, R_{\hat{g}_{\delta}}$ are strictly positive in $\Omega\times\mathbf{S}^1$. 

We end this subsection with the following Proposition, which gives the relation of the ADM mass of $\Sigma\cap E$ and that of the $(\hat{\Sigma}_{\delta},\hat{E})$. For reader's convenience, we recall the assumptions made at the beginning of this subsection in the statement below.

\begin{figure}
    \centering
    \includegraphics[width = 15cm]{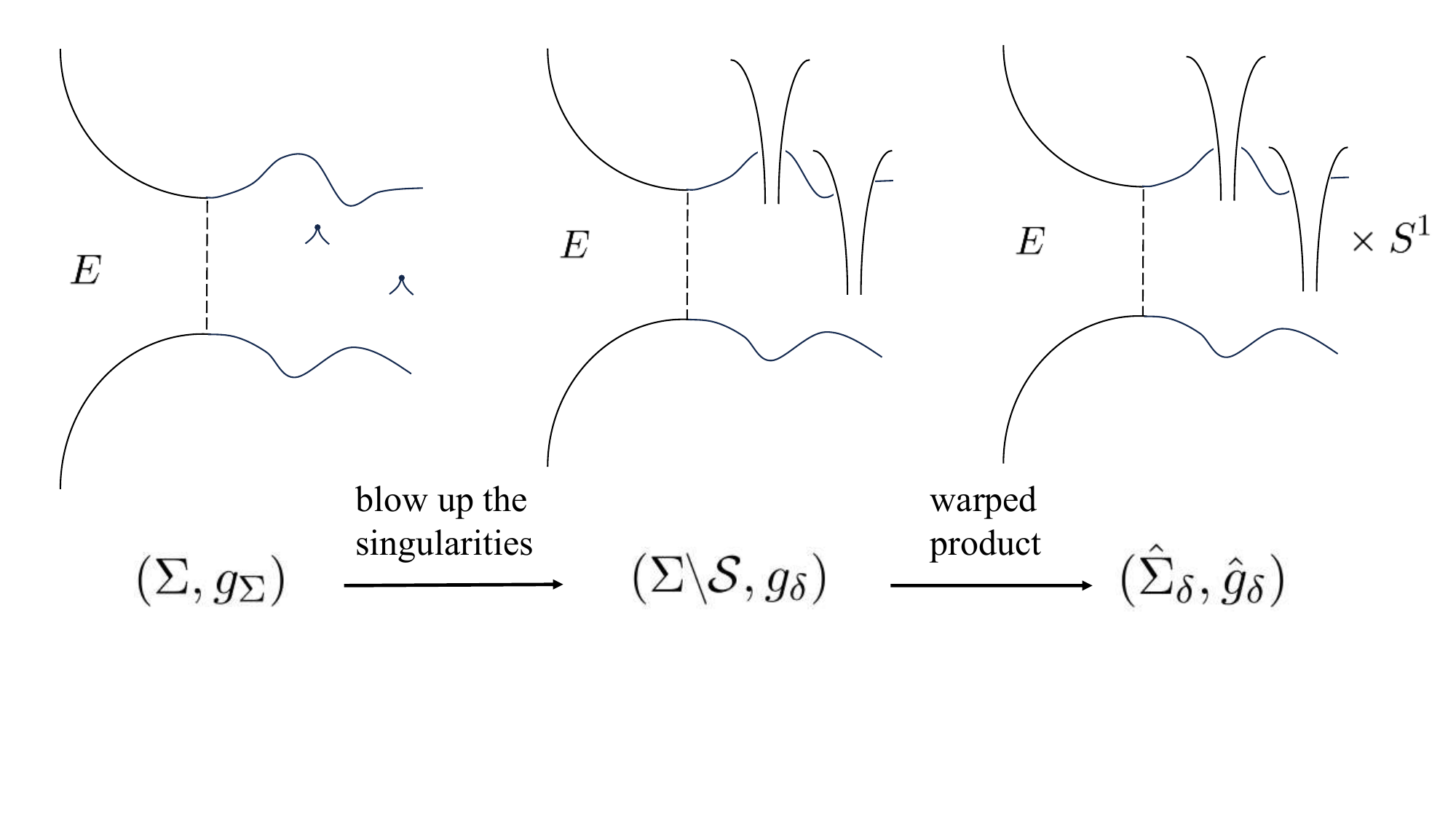}
    \caption{The "blow up - warped product" procedure}
    \label{f4}
\end{figure}

	\begin{proposition}\label{thm: negative mass 2}
		  Let $(M^{n+1},g)$ be a complete manifold with an AF end $E$ of order $\tau>n-2$, and $\Sigma^n\subset M^{n+1}$ a strongly stable area-minimizing hypersurface with isolated singular set $\mathcal{S}$, such that $R_g\ge 0$ along $\Sigma$. Assume $h = R_g+|A|^2$ is strictly positive at some point $p\in \Sigma\backslash\mathcal{S}$.
          
          Then there exists $\delta_0\in (0,1)$, $\sigma_0 = \sigma_0(V_1,g_{\Sigma\backslash\mathcal{S}}|_{V_1},h|_{V_1})>0$ (In particular, $\sigma_0$ is independent of $\delta$ below), such that for all $0<\delta<\delta_0$, we have
		\begin{align}\label{eq: 81}
			 m_{ADM}(\hat{\Sigma}_{\delta},\hat{g}_{\delta},\hat{E})+\sigma_0 < m_{ADM}(\Sigma,g_\Sigma,\Sigma\cap E).
		\end{align}
        Here $\cdot|_{V_1}$ denotes the restriction on $V_1$.
	\end{proposition}
	\begin{proof}
        By Lemma \ref{lem: mass decay ALF} and Proposition \ref{conformal deformation1}, we deduce that
        \begin{align*}
            m_{ADM}(\hat{\Sigma}_{i,\delta}, \hat{g}_{i,\delta}, \hat{E})
            \le &(n-1)m_{ADM}(\Sigma\backslash\mathcal{S}, g_{\delta}, \Sigma\cap E) - \frac{1}{4\omega_{n-2}}\int_{V_i}hG_{\delta}^{-\frac{4}{n-2}}v_{i,\delta}^2d\mu_{g_\delta}\\
            \le &(n-1)m_{ADM}(\Sigma,g,\Sigma\cap E)+2(n-1)a\delta - \frac{1}{4\omega_{n-2}}\int_{V_i}h(G_{\delta}v_{i,\delta})^2d\mu_{g_{\Sigma\backslash\mathcal{S}}}.
        \end{align*}
        Let $\tilde{v}_{i,\delta} = G_{\delta}v_{i,\delta}$. The conformal change formula (See \cite[equations (238)(239)]{Bray2001} for example) yields the following rather simple relation:
        \begin{align*}
            \Delta\tilde{v}_{i,\delta} =  \Delta (G_{\delta}v_{i,\delta}) = G_\delta ^{\frac{n+2}{n-2}}\Delta_{g_{\delta}}v_{i,\delta}+ v_{i,\delta}\Delta G_\delta = G_\delta ^{\frac{n+2}{n-2}}\Delta_{g_{\delta}}v_{i,\delta}
        \end{align*}
        Plugging this into \eqref{eq: 43} and using \eqref{eq: 48}, we deduce that $\tilde{v}_{i,\delta}$ satisfies
        \begin{align}\label{eq: 82}
             -\Delta\tilde{v}_{i,\delta} +\frac{1}{2}R_{g_{\Sigma\setminus\mathcal{S}}} \tilde{v}_{i,\delta} -\frac{1}{4}h\tilde{v}_{i,\delta}=0 \quad \text{in $V_i$}
        \end{align}
        Note that the coefficients in \eqref{eq: 82} are independent of $\delta$. Moreover, all $\tilde{v}_{i,\delta}$ satisfy the same boundary condition
        \begin{equation}\label{eq: 23}
		\left\{
		\begin{aligned}
			&\lim\limits _{ x\to\infty, x\in E} \tilde{v}_{i,\delta}(x) = 1,\\
			&\tilde{v}_{i,\delta}|_{\partial V_i}=0.
		\end{aligned}
		\right.	
	\end{equation}
    On the other hand, since $\lambda_1(-\Delta+\frac{1}{2}R_{g_{\Sigma\setminus\mathcal{S}}}-\frac{1}{4}h)\ge 0$, from the maximum principle Lemma \ref{lem: maximum principle for spectral condition} the solution $\tilde{v}$ to the equation
\begin{equation}\label{eq: 83}
		\left\{
		\begin{aligned}
            &-\Delta\tilde{v} +\frac{1}{2}R_{g_{\Sigma\setminus\mathcal{S}}} \tilde{v} -\frac{1}{4}h\tilde{v}=0 \quad \text{in $V_i$}\\
			&\lim\limits _{ x\to\infty, x\in E} \tilde{v}(x) = 1,\\
			&\tilde{v}|_{\partial V_i}=0.
		\end{aligned}
		\right.	
	\end{equation}
    is unique, which we denote by $\tilde{v}_i$. Thus $\tilde{v}_{i,\delta} = \tilde{v}_i$ for all $\delta>0$. 
    
    Since $(\Sigma\setminus\mathcal{S},g_{\Sigma\setminus\mathcal{S}})$ has $\frac{1}{2}$-scalar curvature at least $h$ in the strong spectral sense, combined with the maximum principle Lemma \ref{lem: maximum principle for spectral condition}, we have $\tilde{v}_{j}>\tilde{v}_{i}$ on $V_i$ for any $j>i\ge 1$. It follows that
        \begin{equation}
            \frac{1}{4\omega_{n-2}}\int_{V_i}h(G_{\delta}v_{i,\delta})^2d\mu_{g_{\Sigma\backslash\mathcal{S}}} =  \frac{1}{4\omega_{n-2}}\int_{V_i}h\tilde{v}_i^2d\mu_{g_{\Sigma\backslash\mathcal{S}}} \ge  \frac{1}{4\omega_{n-2}}\int_{V_1}h\tilde{v}_1^2d\mu_{g_{\Sigma\backslash\mathcal{S}}} := 2\sigma_0>0.
        \end{equation}
        Here we note that $\sigma_0$ is independent of $i,\delta$. 
        Therefore, there exists $\delta_0>0$, such that
        \begin{align}\label{eq: 24}
            m_{ADM}(\hat{\Sigma}_{i,\delta}, \hat{g}_{i,\delta}, \hat{E})\le m_{ADM}(\Sigma,g_\Sigma,\Sigma\cap E)-\sigma_0<m_{ADM}(\Sigma,g_\Sigma,\Sigma\cap E)
        \end{align}
        for all $0<\delta<\delta_0$.
        Letting $i\to\infty$ and using Lemma \ref{lem: convergence ADM mass on ALF}, we obtain the desired result.
	\end{proof}

\begin{remark}
    Since the scalar curvature $R_g$ of the ambient manifold is only assumed to be nonnegative along $\Sigma$, Proposition \ref{thm: negative mass 2} can be considered as a version of the positive mass type theorem on AF spaces with isolated singularities, as is extensively studied in \cite{LM2019}\cite{CFZ24}\cite{DaSW2024}\cite{DWWW2024} and references therein. In fact, $\Sigma$ in Proposition \ref{thm: negative mass 2} contains an AF end $\Sigma\cap E$, and its intrinsic scalar curvature $R_{\Sigma\backslash\mathcal{S}}$ is nonnegative in the strong spectral sense. The strict inequality in \eqref{eq: 81} (non-rigidity) reflects certain positive effect of the singular set $\mathcal{S}\ne\emptyset$.
\end{remark}

\subsection{Blow up-warped construction of area-minimizing hypersurface with singularities in closed manifold}

$\quad$

In this subsection, we always assume $(M^{n+1}, g)$ is a closed manifold, $\Sigma^n$ is a  compact area-minimizing hypersuface with isolated singular set $\mathcal{S}\neq\emptyset$ and $R_g\geq 0$ along $\Sigma^n$. Our goal in this subsection is to construct an $\mathbf{S}^1$-invariant Riemannian metric on $(\Sigma\setminus\mathcal{S})\times \mathbf{S}^1$ with positive scalar curvature.

 From Lemma \ref{lem: |A|>0} there exists $p\in\Sigma\backslash\mathcal{S}$ such that $h(p)=R_g(p)+|A|^2(p)>0$. Let $\mathcal{U}\subset \Sigma$ be a  small neighborhood of $p\in \Sigma$ such that $h(x)>0$ for any $x\in \mathcal{U}$. We always assume $\mathcal{U}\cap \mathcal{S}=\emptyset$.   We need the following lemma:

\begin{lemma}\label{lmm:singular function}
Let  
$f$ be a  non-negative smooth function on $\Sigma\setminus\mathcal{S}$ which is strictly positive somewhere, satisfying  $\supp(f)\subset\subset \mathcal{U}$. Then there is a positive smooth function $u$ on $\Sigma\setminus\mathcal{S}$ with
$$
-\Delta u+fu=0 \quad \text{in $\Sigma\setminus\mathcal{S}$},
$$	
and 
$$
u|_{\partial B(p_i,r)}\geq Cr^{2-n}, \quad \text{ for any  $p_i\in \mathcal{S}$ and any $r>0$}.
$$
\end{lemma}

\begin{proof}
	
For simplicity we assume $\mathcal{S} $	consists of a single point, say $p_1$.  Choose $\{r_i\}$ decreasing to $0$ and consider the following Dirichlet problem:

\begin{equation}\label{eq:Dirichlet prob3}
\left\{
\begin{aligned}
&-\Delta u_i +fu_i=0 \quad \text{in  $\Sigma\setminus B(p_1,r_i)$} ,\\
&u_{i}|_{\partial  B(p_1,r_i)}=1.\\
\end{aligned}
\right.
\end{equation}
Note that $f\geq 0$, we see that  \eqref{eq:Dirichlet prob3}	admits a positive  smooth solution $u_i$. Fix a point $q\in \Sigma\setminus\mathcal{S}$. By a scaling we may assume $u_i(q) = 1$. By the Harnack inequality, passing to a subsequence if necessary, we may assume $u_i$ converges locally smoothly to a function $u$ in $\Sigma\setminus\mathcal{S}$. The assumption that $\mathcal{U}\cap\mathcal{S} = \emptyset$ and $\supp(f)\subset\mathcal{U}$ implies $u$ is harmonic near $\mathcal{S}$.

If  $u$ is bounded on  $\Sigma\backslash\mathcal{S}$, then by the removable singularity Lemma \ref{lem: removable singularity} for harmonic functions, we have
$$
-\Delta u +fu =0.
$$ in $\Sigma$. By choosing the cut-off function $\eta$ in Lemma \ref{lem: estimate for eta}, multiplying two sides of the above equation by $\eta^2u$ and passing to a limit, we obtain
$$
\int_{\Sigma\setminus\mathcal{S}} |\nabla u|^2 d\mu<\infty,
$$		
and 	
$$
\int_{\Sigma\setminus\mathcal{S}} (|\nabla u|^2 +fu^2)d\mu	=0,
$$	
which implies $u=0$ and contradicts to $u(q)=1$. Therefore, $u$ must be unbounded near $\mathcal{S}$. Since $u$ is harmonic near $\mathcal{S}$, we can apply Proposition \ref{prop:existence of singular positive hf} and Proposition \ref{prop:growth of hf1} to get the estimate as stated in the Lemma.
\end{proof}

The next Lemma shows that with a suitable choice of $f$, the conformal metric $\bar g_{\Sigma\setminus\mathcal{S}}= u^{\frac{4}{n-2}}g_{\Sigma\setminus\mathcal{S}}$ has $\frac{1}{2}$-scalar curvature  non-negative in the spectral sense.

\begin{lemma}\label{lmm:spectrum psc1}
Let $u$ be as in Lemma \ref{lmm:singular function} and  $\bar g_{\Sigma\setminus\mathcal{S}}= u^{\frac{4}{n-2}}g_{\Sigma\setminus\mathcal{S}}$. Then for any $\phi \in C^\infty_0(\Sigma\setminus\mathcal{S})$ there holds
$$
\int_{\Sigma\setminus\mathcal{S}} (|\bar\nabla \phi|^2+\frac{1}{2} R_{\bar g_{\Sigma\setminus\mathcal{S}}}\phi^2)d\mu_{\bar g_{\Sigma\setminus\mathcal{S}}}
\geq \int_{\Sigma\setminus\mathcal{S}} \left(\frac{h}{2}-\frac{nf}{n-2} \right)u^{-\frac{2n}{n-2}}\phi ^2 d\mu_{\bar g_{\Sigma\setminus\mathcal{S}}}.
$$	
\end{lemma}

\begin{proof}
The scalar curvature of the conformal metric $\bar{g}_{\Sigma\setminus\mathcal{S}}$ is given by
\begin{equation}
R_{\bar g_{\Sigma\setminus\mathcal{S}}}=u^{-\frac{n+2}{n-2}}(R_{g_{\Sigma\setminus\mathcal{S}}} u-\frac{4(n-1)}{n-2}\Delta u)
=u^{-\frac{4}{n-2}}(R_{g_{\Sigma\setminus\mathcal{S}}} -\frac{4(n-1)}{n-2}f).
\end{equation}
Note that 
$$
\nabla u^2 \cdot\nabla\phi^2=div(\phi^2 \nabla u^2)-2\phi^2 |\nabla u|^2-2f(u\phi)^2,
$$
we obtain
$$
\int_{\Sigma\setminus\mathcal{S}} |\nabla (u\phi) |^2 d\mu_{g_{\Sigma\backslash\mathcal{S}}} =\int_{\Sigma\setminus\mathcal{S}} u^2|\nabla \phi |^2 d\mu_{g_{\Sigma\backslash\mathcal{S}}}-\int_{\Sigma\setminus\mathcal{S}} f(u\phi)^2 d\mu_{g_{\Sigma\backslash\mathcal{S}}}$$
Therefore, we have
\begin{equation}
	\begin{split}
&\int_{\Sigma\setminus\mathcal{S}} (|\bar\nabla \phi|^2+\frac{1}{2} R_{\bar g_{\Sigma\setminus\mathcal{S}}}\phi^2)d\mu_{\bar g_{\Sigma\setminus\mathcal{S}}}\\
&=\int_{\Sigma\setminus\mathcal{S}}(u^2|\nabla\phi|^2+\frac{1}{2}(R_{g_{\Sigma\setminus\mathcal{S}}} -\frac{4(n-1)}{n-2}f)(u\phi)^2	)d\mu_{g_{\Sigma\backslash\mathcal{S}}}\\
&=\int_{\Sigma\setminus\mathcal{S}}(|\nabla(u\phi)|^2+\frac{1}{2}(R_{g_{\Sigma\setminus\mathcal{S}}} -\frac{4(n-1)}{n-2}f)(u\phi)^2	)d\mu_{g_{\Sigma\backslash\mathcal{S}}}	+\int_{\Sigma\setminus\mathcal{S}} f(u \phi)^2d\mu_{g_{\Sigma\backslash\mathcal{S}}}\\
&\geq \int_{\Sigma\setminus\mathcal{S}} \left(\frac{h}{2} -\frac{nf}{n-2}  \right)(\phi u)^2 d\mu_{g_{\Sigma\backslash\mathcal{S}}}\\
&=\int_{\Sigma\setminus\mathcal{S}} \left(\frac{h}{2} -\frac{nf}{n-2}  \right)u^{-\frac{4}{n-2}}\phi ^2 d\mu_{\bar g_{\Sigma\setminus\mathcal{S}}}.
\end{split}\nonumber
\end{equation}
\end{proof}
We end this section with the following Proposition.
\begin{proposition}\label{thm: PSC metric}
     Let $(M^{n+1}, g)$ be a closed manifold, $\Sigma^n$ a  compact area-minimizing hypersuface with isolated singular set $\mathcal{S}\neq\emptyset$ and $R_g\geq 0$ along $\Sigma^n$. Then there exists a complete and $\mathbf{S}^1$-invariant Riemannian metric on $(\Sigma\setminus\mathcal{S})\times \mathbf{S}^1 $ with positive scalar curvature everywhere.
\end{proposition}
\begin{proof} Let $f\ge 0$ be a function on $\Sigma\setminus\mathcal{S}$ supported in $\mathcal{U}$.  Since $h>0$ in $\mathcal{U}$, we may choose $f$ in Lemma \ref{lmm:singular function} small enough so that 
\begin{align*}
    \eta(x):=\left(\frac{h}{2}-\frac{nf}{n-2} \right)u^{-\frac{4}{n-2}} (x)\geq 0, \quad \text{for all $x\in \Sigma\setminus\mathcal{S}$}, 
\end{align*}
with
\begin{align*}
    \eta(p)=\left(\frac{h}{2}-\frac{nf}{n-2}\right)u^{-\frac{4}{n-2}}(p)> 0.
\end{align*}
Let $u$ be as in Lemma \ref{lmm:singular function} and $\bar g_{\Sigma\setminus\mathcal{S}}= u^{\frac{4}{n-2}}g_{\Sigma\setminus\mathcal{S}}$.
From Lemma \ref{lem: completeness}, $(\Sigma\setminus\mathcal{S}, \bar g_{\Sigma\setminus\mathcal{S}})$ is a complete Riemannian manifold.
        Owing to  Lemma \ref{lmm:spectrum psc1}, we know that 
        $$\lambda_1(-\Delta_{\bar{g}_{\Sigma\setminus\mathcal{S}}}+\frac{1}{2}R_{\bar{g}_{\Sigma\setminus\mathcal{S}}}-\eta)\ge 0.$$
Hence, we may find a positive smooth function $\varphi$ on $\Sigma\setminus\mathcal{S}$ such that 
$$
\hat g:=\bar g_{\Sigma\setminus\mathcal{S}}+ \varphi^2 d\theta^2, \quad \theta \in \mathbf{S}^1
$$
is a complete $\mathbf{S}^1$-invariant Riemannian metric on $(\Sigma\setminus\mathcal{S})\times \mathbf{S}^1 $ with $R_{\hat g}\geq 0$ and $R_{\hat g}(p\times\{\theta\}) = 2\eta(p)> 0$ for all $\theta\in\mathbf{S}^1$. Then by Lemma \ref{G-invariant Kazdan}, we may deform $\hat g$ to get a new complete and $\mathbf{S}^1$-invariant Riemannian metric on $(\Sigma\setminus\mathcal{S})\times \mathbf{S}^1 $ with positive scalar curvature everywhere.
\end{proof}

\section{Desingularization of area-minimizing hypersurfaces in 8-manifolds}

In this section, we give the proof of the main theorems. 

   \begin{proof}[Proof of Theorem \ref{thm: rigidity for minimal surface}]
   The condition $\tau>n-2$, in conjunction with \cite[Lemma 22]{EK23} implies that $m_{ADM}(\Sigma,g_{\Sigma},\Sigma\cap E)=0$.

        (1) Assume $\mathcal{S}\ne\emptyset$, so Lemma \ref{lem: |A|>0} implies $h = R_g+|A|^2$ is strictly positive at some point $p\in\Sigma\backslash\mathcal{S}$. 
        Now $\Sigma^n$ has $\frac{1}{2}$-scalar curvature not less than $R_g+|A|^2$ in the strong spectral sense due to its strong stability, hence Proposition \ref{thm: negative mass 2} implies
        \begin{align*}
            m_{ADM}(\hat{\Sigma}_{\delta},\hat{g}_{\delta},\hat{E})\le m_{ADM}(\Sigma,g_{\Sigma},\Sigma\cap E)-\sigma_0 = -\sigma_0<0.
        \end{align*}

        (a) $n+1=8$. Since $(\hat{\Sigma}_{\delta},\hat{g}_{\delta},\hat{E})$ is an $\mathbf{S}^1$-invariant ALF manifold, the above inequality follows contradicts Proposition \ref{prop: PMT for S^1 symmetric ALF}.

        (b) $\Sigma\backslash\mathcal{S}$ is spin. We need a version of the positive mass theorem for ALF manifolds with arbitrary ends. This can be derived using the same argument of Proposition \ref{prop: PMT for S^1 symmetric ALF}, with the topological obstruction result Proposition \ref{prop: noncompact dominate enlargeable 2} replaced by \cite[Theorem 1.4]{Zei20}. Thus we also get a contradiction.
        
        Therefore, we have $\mathcal{S}=\emptyset$. This enables us to argue as in the smooth case, and the rigidity of $\Sigma$ as stated in the Theorem is an immediate consequence of \cite[p.14-15]{Carlotto16},\cite[Proposition 32]{EK23}.\\
        (2) Suppose $\Sigma\not\subset U_1$, then $h = R_g+|A|^2$ is strictly positive in $\Sigma\cap (U_2\backslash U_1)$ and Proposition \ref{thm: negative mass 2} applies. Thus we can argue as in (1) to get a contradiction. Consequently, we can always assume $\Sigma\subset U_1$, and the remaining part is a consequence of Theorem \ref{thm: rigidity for minimal surface} (1).
    \end{proof}
   To establish positive mass theorem on AF manifolds with arbitrary end, we introduce a sequence of \textit{free boundary problems with inner obstacle} to effectively compensate for the lack of compactness in minimal surfaces caused by arbitrary ends. Recall we use $B_r^{n+1}(x)$ or simply $B_r(x)$ to denote coordinate balls in the AF end $E$. Recall we use $(x',z),x'\in\mathbf{R}^n$, $z\in \mathbf{R}$ to denote the points in the AF end $E$ and $S_t = \{x = (x',z)| z = t\}$ to represent the coordinate $z$-hyperplane. Within the AF end $E$,  the cylinder-shaped hypersurface $(\partial B^{n+1}_r(O) \cap S_0) \times \mathbf{R}$ partitions $E$ into two distinct regions: an inner part and an outer part. We use $C_r$ to denote the inner part.
Let
\begin{align*}
    \partial E\subset V_1\subset V_2\subset\dots
\end{align*}
be  a sequence of compact exhausting domains of $M\setminus E$ and let $j:\mathbb{R}\longrightarrow \mathbb{Z}_+$ be a non-decreasing index function with 

\begin{align*}
    \lim_{R\to\infty} j(R) = \infty
\end{align*}

 We define (see Figure \ref{fig: inner obstacle}) 
 \begin{figure}
     \centering
     \includegraphics[width=0.5\linewidth]{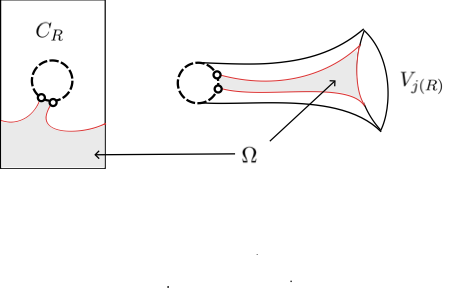}
     \caption{free boundary problems with inner obstacle}
     \label{fig: inner obstacle}
 \end{figure}

\begin{equation}\label{eq: 79}
\begin{split}
\mathcal{F}_{R}&:=\{\Sigma=\partial \Omega \backslash \partial C_{R}: \Omega \subset C_{R}\cup V_{j(R)} \text{  is a Caccioppoli set that satisfies}\\
	&\text{ $C_{R} \cap \{z\leq a\}\subset \Omega $ and $\Omega \cap \{z\geq b\}= \emptyset  $ for some $-\infty< a\leq b<\infty$} \}.
	\end{split}
\end{equation}
\begin{proof}[Proof of Theorem \ref{thm: pmt8dim}]
    For the inequality part, we argue by contradiction. Assume that $m_{ADM}(M,g,E)<0$. By \cite[Proposition 3.2]{Zhu23}, we may assume that $(M^{n+1},g)$ is asymptotically Schwarzschild and $E\backslash B^{n+1}_{r_1}$ is harmonically flat for some $r_1>1$, i.e. $g_{ij} = u^{\frac{4}{n-1}}\delta_{ij}$ in $\mathbf{R}^{n+1}\backslash B^{n+1}_{r_1}(O)$, where $u$ is harmonic in $\mathbf{R}^{n+1}\backslash B^{n+1}_{r_1}(O)$. We have the expansion
    \begin{align*}
        u = 1+\frac{m}{2|x|^{n-1}} + O(|x|^{-n})
    \end{align*}
    in the AF end $E$. Using \cite[Lemma 5.1]{LUY21}, there exists $\varphi\in C^{\infty}(M)$ and $r_2>r_1$, such that

    \begin{equation}\label{eq: 80}
\left\{
\begin{aligned}
&\varphi = 1-\frac{m}{4|x|^{n-1}}\quad \text{for} \quad|x|>r_2,\\
&\varphi = \mbox{const}>0\quad \text{for} \quad x\in M\backslash (E\backslash B_{r_1}(O)),\\
&\Delta_g\varphi\le 0 \quad \text{in} \quad M,\\
&\Delta_g\varphi< 0 \quad \text{in} \quad E\backslash B_{r_2}(O).
\end{aligned}
\right.
\end{equation}
Let $\tilde{g} = \varphi^{\frac{4}{n-1}}g$, then we have $R_{\tilde{g}}>0$ in $E\backslash B_{r_2}(O)$ and
    \begin{align}\label{eq: 50}
        \tilde{g}_{ij} = \varphi^{\frac{4}{n-1}}\delta_{ij} \quad\text{ in }\quad E\backslash B^{n+1}_{r_2}(O),
    \end{align}
    where 
    \begin{align*}
        \varphi(x) = (1+\frac{m}{2|x|^{n-1}})(1-\frac{m}{4|x|^{n-1}})+O(|x|^{-n})
    \end{align*}
    for some $m<0$. Note that the second item in \eqref{eq: 80} guarantees the completeness of $\tilde{g}$. 

      For $t_0>r_2$, it follows from \cite{SY79} that the mean curvatures of $S_{\pm t}$ with respect to $\pm\frac{\partial}{\partial t}$ directions are positive for all $t\ge t_0$. Consider the cylinders $C_{R_i}$ with $R_i\to\infty$. Then \eqref{eq: 50} ensures that the dihedral angle between $\partial C_{R_i}$ and $S_{\pm t}$ is $\frac{\pi}{2}$. 
    Denote $C_{R_i,t_0}$ to be the intersection of $C_{R_i}$ with the region bounded between two hyperplanes $S_{\pm t_0}$.
    Now  we can seek for an  area-minimizing boundary $\Sigma_{R_i}$ in $(C_{R_i,t_0}\cup V_i,\tilde{g})$ with inner obstacle by the compactness theorem of Caccioppoli sets \cite[Theorem 12.26]{Mag12} (See Figure \ref{fig: inner obstacle}). Since $\Sigma_{R_i}\cap K\neq \emptyset$ for some fixed compact $K$, $\Sigma_{R_i}$ converges to a limiting area-minimizing boundary $\Sigma$.
   
   \textbf{Claim: There exists some fixed $R_0$ such that outside the coordinate ball $B^{n+1}_{R_0}$ of the AF end $E$, $\Sigma\cap(E\setminus B^{n+1}_{R_0})$ is smooth.}\\
   Suppose not, then there exist a sequence of points $q_i\in E$ with $2r_i:=|q_i|\to\infty$ such that $\Sigma$ is singular at $q_i$. Recall $\Sigma\cap E$ is bounded by two hyperplane $S_{\pm t_0}$. By the minimizing property of $\Sigma$, 
   \begin{align*}
       \mathcal{H}^n(\Sigma\cap C_{r_i}(q_i))\le\mathcal{H}^n(\partial C_{r_i}(q_i)\cap (S_0\times [-t_0,t_0]))+\mathcal{H}^n(C_{r_i}(q_i)\cap S_{t_0}) =  \omega_{n}r_i^n(1+o(1))
   \end{align*}
   where $C_{r_i}(q_i) $ is the cylinder of radius $r_i$ whose axis passes $p_i$ and is parallel to the $z$-axis.
   We take the blow down argument.  Rescale the asymptotically flat metric $\tilde{g}_{ij}$ in $B^{n+1}_{r_k}(q_k)$ to the metric $\tilde{g}_{ij}^k$
		using the transformation $\Phi$ given by $x=q_k+r_k\tilde{x}$, i.e.,
		\[\tilde{g}_{ij}^k=\Phi^*g_{ij}.
		\]
		Then $\tilde{g}_{ij}^k$  converge to the Euclidean metric $\delta_{ij}$ on $\tilde{B}_1^{n+1}(\tilde{O})$ locally uniformly in $C^2$ sense by asymptotic flatness, since $r_k\to\infty$. By passing to a subsequence, we have $\Phi^*(\Sigma\cap B^{n+1}_{r_k}(y_k))$ converge to an
		area-minimizing hypersurface $\tilde{\Sigma}$ in both current and varifold sense with $\mathcal{H}^n(\tilde{\Sigma}\cap \tilde{B}_1^{n+1}(\tilde{O}))=\omega_{n-1}$ \cite[Theorem 34.5]{Simon83}. Note  $\spt(\tilde{\Sigma})\subset\tilde{B}_1^{n+1}(\tilde{O})\cap\{\tilde{x}_{n+1}=0\}$. Then
        by constancy theorem in geometric measure theory \cite[Theorem 26.27]{Simon83}, there exists some $m\in \mathbb{Z}$ such that $\|\tilde{\Sigma}\| = m\|\tilde{B}_1^{n+1}(\tilde{O})\cap\{\tilde{x}_{n+1}=0\}\|$. The area estimate above yields $m = 1$. Hence, $\tilde{\Sigma}$ is a hyperplane of multiplicity $1$, hence smooth at the point $\tilde{O}$. It follows from the Allard's regularity theorem that $\Sigma$ must be smooth at the point $q_k$ when $k$ is sufficiently large, we get the contradiction.
   
Now we can use the argument in  \cite[Lemma 3.6]{HSY24} to show $\Sigma\cap E$ is asymptotic to a hyperplane  which is parallel to $S_0$ at infinity in Euclidean sense and $\Sigma$ is strongly stable. Note $R_{\tilde{g}}> 0$ in $E\backslash B_{r_2}(O)$. Then we get the contradiction by Theorem \ref{thm: rigidity for minimal surface}.

    For the rigidity part, it follows directly from the inequality part and the proof of \cite[Theorem 1.2]{Zhu23}
\end{proof}
\begin{proof}[Proof of Theorem \ref{thm: georch free of singularity}]
    Assume $\mathcal{S}\ne \emptyset$.  If $n\le 7$, then by Proposition \ref{thm: PSC metric}, 
    there exists a complete and $\mathbf{S}^1$-invariant Riemannian metric on $(\Sigma\setminus\mathcal{S})\times \mathbf{S}^1 $ with positive scalar curvature everywhere, which contradicts with Proposition \ref{prop: noncompact dominate enlargeable 2}. For the case that $\Sigma\backslash \mathcal{S}$ is spin, we can apply \cite[Theorem 1.1]{LSWZ24} to obtain the result.
\end{proof}

    \begin{proof}[Proof of Corollary \ref{cor:8dim georch}]
    Since the case for $n\le 7$ has been established as classical work by Schoen and Yau \cite{SY79b}, our focus is directed towards the specific case where $n=8$. Decompose $\mathbf{T}^8 = \mathbf{T}^7\times \mathbf{S}^1$. Without loss of generality, we can assume that $f$ is transversal to $\mathbf{T}^7\times \{1\}$, so $\Sigma_0 = f^{-1}(\mathbf{T}^7)$ is an embedded regular submanifold in $M$. By classical results in geometric measure theory we can find a homological minimizing integer multiplicity current $\tau$ with $\partial R = \tau-\|\Sigma_0\|$, where $R$ is an integer multiplicity $8$-current. From the regularity theory we know that the support of $\tau$, denoted by $\Sigma$, has isolated singular set $\mathcal{S}$. By constancy theorem, the restriction of $\tau$ on $M\backslash\mathcal{S}$ is equal to $m\| \Sigma\backslash \mathcal{S}\|$ for some $m\in\mathbf{Z}$.

    Consider the composition of the following sequence of maps
    \begin{align*}
        F:\Sigma\hookrightarrow M^8\longrightarrow \mathbf{T}^8\longrightarrow \mathbf{T}^7,
    \end{align*}
    where the last map is the projection map from $\mathbf{T}^8$ to its $\mathbf{T}^7$-factor. Let $\omega\in\Omega_c^7(\mathbf{T}^7\backslash F(\mathcal{S}))$ be a differential form with integration $1$ on $\mathbf{T}^7\backslash F(\mathcal{S})$. We have
    \begin{align*}
        m\deg(F|_{\Sigma\backslash \mathcal{S}}) =& m\int_{\Sigma\backslash\mathcal{S}}F^*\omega = \tau(F^*\omega) = \|\Sigma_0\|(F^*\omega)\\
        = &\int_{\Sigma_0}F^*\omega = \deg (F|_{\Sigma_0}) = \deg f\ne 0.
    \end{align*}
    Using the definition in Appendix B, we see that $M$ admits a nonzero degree map to $\mathbf{T}^7$ in the sense of Definition \ref{defn: degree 2}. This contradicts Theorem \ref{thm: georch free of singularity}.
\end{proof}

We end this section with the following remark.
\begin{remark}\label{remark: arbitrary ends}
    When $n\le 7$, we may modify the proof of Theorem \ref{thm: pmt8dim} to obtain an alternative proof for the positive mass theorem for AF manifold $(M^n,g)$ with arbitrary ends. Assume $R_g>0$ everywhere without loss of generality. First, we can find an area-minimizing boundary $\Sigma_1\subset M$ which is also strongly stable following the proof of Theorem \ref{thm: pmt8dim}. By Proposition \ref{prop: eq1-arbitrary end}, we can solve $u_1\in C^{\infty}(\Sigma_1)$ with $-\Delta_{\Sigma_1}u_1+\frac{1}{4}R_{\Sigma_1} = 0$ and $\lim_{x\in {\Sigma_1\cap E},|x|\to\infty}u_1 = 1$. Let
    \begin{align*}
        (\hat{\Sigma}_1, \hat{g}_1) = (\Sigma_1\times \mathbf{S}^1, u_1^2ds_1^2+g_\Sigma).
    \end{align*}
    Then $R_{\hat{g}_1}>0$ everywhere. In the $k$-th step, we begin with the $n$-dimensional ALF manifold with arbitrary ends
    \begin{align*}
        (\hat{\Sigma}_{k-1}, \hat{g}_{k-1}) = (\Sigma_{k-1}\times \mathbf{T}^{k-1}, u_1^2ds_1^2+\dots +u_{k-1}^2ds^2_{k-1}+g_{\Sigma_{k-1}})
    \end{align*}
    where $u_i\in C^{\infty}(\Sigma_{k-1})$ and $\Sigma_{k-1}\subset\Sigma_{k-2}\subset\dots\Sigma_1\subset M$ is a sequence of submanifolds. Next, using the method in the proof of Theorem \ref{thm: pmt8dim}, we can find an area-minimizing boundary $\bar{\Sigma}_k\subset \hat{\Sigma}_{k-1}$ that is also strongly stable. Following the argument of \cite[p.188]{GL83}, $\bar{\Sigma}_k$ is $\mathbf{T}^{k-1}$-symmetric, so we just regard it as $\Sigma_k\times \mathbf{T}^{k-1}$ with $\Sigma_k\subset\Sigma_{k-1}$. Then by applying Proposition \ref{prop: eq1-arbitrary end} and conducting the warped product process once again we obtain the $k$-th ALF manifold with arbitrary ends
    \begin{align*}
        (\hat{\Sigma}_{k}, \hat{g}_{k}) = (\Sigma_{k}\times \mathbf{T}^{k}, u_1^2ds_1^2+\dots +u_{k}^2ds^2_{k}+g_{\Sigma_{k}}).
    \end{align*}
    This reduction process continues until $\dim\Sigma_k\le 3$, at which point the positive mass theorem follows from the spin-theoretic results in \cite{Zei20}. A contradiction follows from the argument of the mass decay Lemma \ref{lem: mass decay ALF}. An interesting aspect of this proof is that it is based on the torical-symmetrization method for minimal hypersurfaces (see \cite{FcS1980},\cite{GL83}, and \cite{Gro18}), providing an alternative to the $\mu$-bubble approach.
    
\end{remark}

\appendix
\section{Poincare-Sobolev inequality on area-minimizing boundaries}

In this appendix we want to verify a local Poincare-Sobolev inequality. Let $\mathcal{B}_r$ be a geodesic ball with radius $r$ in $(N^{n+1},g)$ contained in a normal coordinates chart $(r<r_0)$. Then we have
\begin{proposition}\label{prop: Pioncare-Sobolev inequality}
	Let $S$ be an area-minimizing boundary in  $N$ with $\partial S \cap \mathcal{B}_r=\phi$, and assume the sectional curvature of $\mathcal{B}_r$ is bounded by $K_0$. Then there exists $\beta,\gamma$ depending only on $n,r_0$ and $K_0$, such that for every $f\in C^1(\mathcal{B}_ r)$, it holds
	\begin{align}\label{eq: A Poincare}
	    \inf_{k\in \mathbf{R}}\{\int_{\mathcal{B}_{\beta r}}|f-k|^{\frac{n}{n-1}} d\|S\|\}^{\frac{n-1}{n}}
	\leq 2\gamma\int_{\mathcal{B}_r}|\nabla_S f|d\|S\|.
	\end{align}

\end{proposition}

To achieve this, we need  the following isoperimetric inequality.

\begin{lemma}\label{lmm: isoperimetric ineq}
	Let $N,\mathcal{B}_r$ be as in Proposition \ref{prop: Pioncare-Sobolev inequality}, and let $S\subset\mathcal{B}_r$ be an area-minimizing boundary. Then one has the isoperimetric inequality
	$$
	M(S)^{\frac{n-1}{n}}\leq \gamma M(\partial S).
	$$
	Here $M(S)$ denotes the mass of $S$, and $\gamma$ depends only on  $n,r_0$ and $K_0$.
\end{lemma} 
Once Lemma \ref{lmm: isoperimetric ineq} is proved, Theorem 2 in \cite{BG1972}, and hence Proposition \ref{prop: Pioncare-Sobolev inequality}, can be established on $(N^{n+1},g)$ by exactly the same arguments as in \cite{BG1972}.

\begin{proof}[Proof of Lemma \ref{lmm: isoperimetric ineq}]
	As $\mathcal{B}_r$ is contained in a normal coordinates of $(M^{n},g)$ and $|K|\le K_0$, $r<r_0$, there is Euclidean metric $g_0$ on $\mathcal{B}_r$ and a constant $\Lambda$ depending only on $n, r_0,K_0$ with 
	$$
	\Lambda^{-1} g_0\leq g\leq \Lambda g_0.
	$$
	For any $n$-integral current $G$ in $\mathcal{B}_r$, let $M_0(G)$ denote the mass of $G$	 w.r.t $g_0$. Then by the definition of the mass, we have
	\begin{equation}\label{eq: eq2}
		\Lambda^{-\frac{n}{2}} M_0(G)\leq  M(G)\leq \Lambda^{\frac{n}{2}} M_0(G).	
	\end{equation}
	Let $S_0$ be the least area $n$-integral current in $(\mathcal{B}_r,g_0)$ with $\partial S_0=\partial S$ (we may assume $\mathcal{B}_r$ is a convex ball with $g_0$ and $g$). The isoperimetric inequality for minimal $n$-integral currents with compact support in $(\mathcal{B}_r,g_0)\hookrightarrow\mathbf{R}^n$ yields
	$$
	M_0(S_0)^{\frac{n-1}{n}}\leq \gamma_0 M_0(\partial S), ~\text{for some constant }\gamma \text{ depending only on $n$}.
	$$
	In conjunction with \eqref{eq: eq2}, we have
	$$
	M(S)\leq M(S_0)\leq \Lambda^{\frac{n}{2}} M_0(S_0)\leq \Lambda^{\frac{n}{2}} \gamma_0M_0(\partial S)\leq \Lambda^{n}\gamma_0 M(\partial S).
	$$
	Setting $\gamma=\Lambda^{n}\gamma_0$ gives the conclusion.
\end{proof}

\begin{corollary}\label{cor: 12poincare}
    Let $N,\mathcal{B}_r,S,r_0,K_0$ be as in Proposition \ref{prop: Pioncare-Sobolev inequality}. Then there exists $\beta,C_P$ depending only on $n,r_0$ and $K_0$, such that for every $f\in C^1(\mathcal{B}_ r)$, it holds
    \begin{align*}
        \fint_{\mathcal{B}_{\beta r}}|f-f_{\mathcal{B}_{\beta r}}|d\|S\|\le C_P r(\fint_{\mathcal{B}_r} |\nabla_S f|^2)^\frac{1}{2}
    \end{align*}
    Here the notation $f_{\mathcal{B}_{\beta r}}$ stands for the average of $f$ on $\mathcal{B}_{\beta r}$.
\end{corollary}
\begin{proof}
    Let $k$ be the number such that the left hand side of \eqref{eq: A Poincare} attains its minimum.  Applying \eqref{eq: A Poincare}, we have
    \begin{align*}
        |f_{\mathcal{B}_{\beta r}}-k|\le \frac{1}{\|S\|(\mathcal{B}_{\beta r})}\int_{\mathcal{B}_{\beta r}}|f-k|\le \frac{1}{\|S\|(\mathcal{B}_{\beta r})^{\frac{n-1}{n}}}(\int_{\mathcal{B}_{\beta r}}|f-k|^{\frac{n}{n-1}})^{\frac{n-1}{n}}.
    \end{align*}
    Thus
    \begin{equation}\label{eq: A2 Poincare}
        \begin{split}
             (\int_{\mathcal{B}_{\beta r}}|f-f_{\mathcal{B}_{\beta r}}|^{\frac{n}{n-1}})^{\frac{n-1}{n}}&\le (\int_{\mathcal{B}_{\beta r}}|f-k|^{\frac{n}{n-1}})^{\frac{n-1}{n}}+|f_{\mathcal{B}_{\beta r}}-k|\cdot\|S\|({\mathcal{B}_{\beta r}})^{\frac{n-1}{n}}\\
       & \le 2  (\int_{\mathcal{B}_{\beta r}}|f-k|^{\frac{n}{n-1}})^{\frac{n-1}{n}}\le 4\gamma\int_{\mathcal{B}_r}|\nabla_S f|d\|S\|.
        \end{split}
    \end{equation}
    Then we can apply Cauchy-Schwarz inequality on both sides of \eqref{eq: A2 Poincare} and conclude the proof using the monotonicity formula for minimal hypersurfaces.
    
\end{proof}

Next, we establish a Sobolev inequality on $S$ by using Proposition \ref{prop: Pioncare-Sobolev inequality}.

\begin{proposition}\label{prop: Sobolev inequality on S}
    Let $S$ be an area-minimizing boundary in $(N^{n+1},g)$. Let $\mathcal{U}$ be bounded domain in $(N^{n+1},g)$ such that
    \begin{align}\label{eq: nonempty bdry of S}
        (\spt S)\backslash\mathcal{U}\ne\emptyset.
    \end{align}
     Then there exists $C = C(\mathcal{U},S,g)$, such that for any $f\in C^1_0(\mathcal{U})$, there holds
    \begin{align}\label{eq: 40}
        \int_{\mathcal{U}}|f|^{\frac{n}{n-1}} d\|S\|\le C\int_{\mathcal{U}}|\nabla_S f| d\|S\|.
    \end{align}
\end{proposition}

\begin{remark}
    The condition \eqref{eq: nonempty bdry of S} here guarantees $\spt S\cap \mathcal{U}$ has non-empty boundary.
\end{remark}

\begin{proof}[Proof of Proposition \ref{prop: Sobolev inequality on S}]
    We adopt the argument in \cite{SY79} by Schoen-Yau. Let $\mathcal{U}_1$ be a neighborhood of $\mathcal{U}$ with compact closure such that $(\spt S)\backslash\mathcal{U}_1\ne\emptyset$. For each $p\in\mathcal{U}_1$, there exist constants $r_p,\beta_p,\gamma_p>0$ given by Proposition \ref{prop: Pioncare-Sobolev inequality}. By the finite covering theorem we can select $p_1,p_2,\dots,p_k\in \spt S$ such that
    $$\spt S\cap\mathcal{U}_1\subset \bigcup_{i=1}^k \mathcal{B}_{\beta_ir_i}(p_i).
    $$
    Assume \eqref{eq: 40} is not true, then we can select a sequence of $f_i\in C^1_0(\mathcal{U})$ with
    \begin{align}\label{eq: 41}
        \int_{\mathcal{U}}|\nabla_S f_j| d\|S\|<j^{-1} \mbox{ and }\int_{\mathcal{U}}|f_j|^{\frac{n}{n-1}} d\|S\| = 1.
    \end{align}
    Let $\tilde{f}_j\in Lip(\mathcal{U}_1)$ be the function that extends $f_j$ to $\mathcal{U}_1\setminus \mathcal{U}$ by zero. It follows from Proposition \ref{prop: Pioncare-Sobolev inequality} that
    \begin{align*}
        (\int_{\mathcal{B}_{\beta_i r_i}(p_i)}|\tilde{f}_j-k_{j,i}|^{\frac{n}{n-1}} d\|S\|)^{\frac{n-1}{n}}
	\leq 2\gamma_i\int_{\mathcal{B}_{r_i}(p_i)}|\nabla_S \tilde{f}_j|d\|S\|<Cj^{-1}.
    \end{align*}
    for some $k_{j,i}$. By \eqref{eq: 41} we see $k_{j,i}$ is uniformly bounded, so $k_{j,i}$ converges subsequentially to a constant $k_i$ as $j\to\infty$, and
    \begin{align*}
        \int_{\mathcal{B}_{\beta_i r_i}(p_i)}|\tilde{f}_j-k_i|^{\frac{n}{n-1}} d\|S\|\to 0\mbox{ as }j\to\infty.
    \end{align*}
    Since $\tilde{f}_j\equiv 0$ on $\mathcal{U}_1\backslash\mathcal{U}$, we conclude that $k_i = 0$ for all $i$. Consequently, we deduce that
    \begin{align*}
        \int_{\mathcal{U}}|f_j|^{\frac{n}{n-1}} d\|S\| \to 0
\mbox{ as }j\to\infty.
\end{align*}
This is not compatible with \eqref{eq: 41}.
\end{proof}

By a similar argument, we can also establish the following.

\begin{proposition}\label{prop: Sobolev inequality on S 2}
    Let $(M,g)$ be a complete manifold with an AF end $E$, and let $S$ be an area-minimizing boundary in $(M,g)$. Suppose $\mathcal{U}$ is a neighborhood of $E$ such that $\mathcal{U}\backslash E$ has compact closure in $M$. Then there exists $C = C(\mathcal{U},S,g)$, such that for any $f\in C^1_0(\mathcal{U})$, there holds
    \begin{align*}
        \int_{\mathcal{U}}|f|^{\frac{n}{n-1}} d\|S\|\le C\int_{\mathcal{U}}|\nabla_S f| d\|S\|.
    \end{align*}
\end{proposition}

\section{Degree of proper maps on noncompact manifolds and singular spaces}
In this appendix, we record some definitions of the degree of proper maps between possibly noncompact manifolds and singular spaces. We begin with the following proposition.

\begin{proposition}\label{prop: equivalence degree}
    Let $X,Y$ be oriented $n$-manifolds and $f:Y\longrightarrow X$ be a smooth proper map. Then the degree of $f$ can be defined in one of the following four ways, and all these definitions are equivalent:
    
    (1) For a regular value $x\in X$ of $f$, count the number of elements of $f^{-1}(x)$ with respect to the orientation.

    (2) Consider the induced map of the locally finite homology
    \begin{align*}
        f_*: \mathbf{Z} = H_{n}^{lf}(Y)\longrightarrow H_{n}^{lf}(X) = \mathbf{Z}
    \end{align*}
    and set $f_*([Y]) = (\deg f)[X]$. Here $[X]$ denotes the fundamental class of $X$ in the locally finite homology.

    (3) Consider the induced map of the compactly supported cohomology
    \begin{align*}
        f^*: \mathbf{Z} = H_{c}^{n}(X)\longrightarrow H_{c}^{n}(Y) = \mathbf{Z}
    \end{align*}
    and set $f^*([Y]^*) = (\deg f)[X]^*$. Here $[X]^*$ denotes the fundamental class of $X$ in the compactly supported cohomology.

    (4) Consider the induced map of the compactly supported de Rham cohomology
    \begin{align*}
        f^*: \mathbf{R} = H_{dR,c}^{n}(X)\longrightarrow H_{dR,c}^{n}(Y) = \mathbf{R}
    \end{align*}
    In this case we have
    \begin{align*}
        \int_{Y}f^*\omega = (\deg f)\int_X\omega
    \end{align*}
    for any compactly supported differential $n$-form $\omega$ on $X$.
\end{proposition}

\begin{proof}
    (1)$\Longleftrightarrow$(2) follows from the excision lemma in combination with a local homology group argument. (2)$\Longleftrightarrow$(3) follows from the non-degeneracy of the bilinear form
    \begin{align*}
        H^n_c(X)\otimes H_{n}^{lf}(X)\longrightarrow\mathbf{Z}
    \end{align*}
    (3)$\Longleftrightarrow$(4) is a consequence of the universal coefficient theorem and the de Rham Theorem.
\end{proof}

As an application of Proposition \ref{prop: equivalence degree}, we can give the following definition for degrees of quasi-proper maps, which slightly extends \cite[Definition 1.7]{CCZ23}.

\begin{definition}\label{defn: non-zero degree}
            Let $X^n$ and $Y^n$ be oriented smooth $n$-manifolds, where $Y^n$ is not necessarily compact. Let $f: Y^n\longrightarrow X^n$ be quasi-proper (see \cite[Theorem 1.4]{CCZ23}) and satisfies
            \begin{itemize}
                \item $S_{\infty}$ consists of discrete points, where 
                \begin{align*}
                    S_{\infty} = \bigcap_{K\subset Y \text{ compact }} \overline{f(Y-K)}
                \end{align*}
            \end{itemize}
            Then it is direct to check that the restriction map $f|_{Y\backslash f^{-1}(S_{\infty})}: Y\backslash f^{-1}(S_{\infty})\longrightarrow X\backslash S_{\infty}$ is proper, and we say $f$ has degree $k$ if $f|_{Y\backslash f^{-1}(S_{\infty})}: Y\backslash f^{-1}(S_{\infty})\longrightarrow X\backslash S_{\infty}$ has degree $k$ as a proper map.
        \end{definition}

Next, we focus on the definition of non-zero degree maps defined on singular spaces. We will introduce two different but compatible definitions. The first definition is directly linked to Definition \ref{defn: non-zero degree}, and the second one is more intrinsic.

\begin{definition}\label{defn: degree 2}
    Let $M$ be a compact topological space, such that $M\backslash \mathcal{S}$ is a smooth oriented $n$-manifold, where $\mathcal{S} = \{p_1,p_2,\dots,p_k\}$. Let $N^n$ be a compact oriented manifold. Let $\psi:M\longrightarrow N$ be a continuous map, then the restriction map $\psi|_{M\backslash\mathcal{S}}$ is quasi-proper. The degree of $\psi$ is defined as the degree of the quasi-proper map $\psi|_{M\backslash\mathcal{S}}:M\backslash\mathcal{S}\longrightarrow N$ as in Definition \ref{defn: non-zero degree}.
\end{definition}

Since $M$ may not have a manifold structure, we aim to establish the concept of non-zero degree through purely topological considerations. This leads us to the following lemma:

\begin{lemma}
    Let $M,N,\psi$ be as in Definition \ref{defn: degree 2}. Assume the induced map
    \begin{align*}
        \psi_*:H_n(M)\longrightarrow H_n(N)\cong\mathbf{Z}
    \end{align*}
    is not identically zero, then $\psi$ has non-zero degree in the sense of Definition \ref{defn: degree 2}.
\end{lemma}
\begin{proof}
    For a locally compact space $X$ and an open set $U\subset X$, there is a restriction map $\iota: H_{*}^{lf}(X)\longrightarrow H^{lf}_n(U)$ (see for instance \cite[Section 6.1.3]{Morgan} for the definition). In the case that $X$ is an oriented $n$-manifold, $\iota$ is an isomorphism on $H_n^{lf}$. By functoriality of locally finite homology and naturality of the restriction map $\iota$, the following diagram commutes.
    \begin{equation*}
\xymatrix{&H_n(M)\ar[r]\ar[d]^{\psi_*}&H_{n}^{lf}(M\backslash \psi^{-1}(\psi(\mathcal{S})))\ar[d]^{\psi_*}\\
&H_n(N)\ar[r]^{\cong}&H_n^{lf}(N\backslash\mathcal{S}) \quad \quad }
\end{equation*}
where the horizontal arrows are given by $\iota$. From our assumption and the fact that the lower horizontal map is an isomorphism, we see that the vertical map on the right side is not identically zero. This shows that $\psi$ has non-zero degree in the sense of Definition \ref{defn: degree 2}.
\end{proof}

\begin{remark}
    If $n\ge 2$ and $M$ has structure as considered in \cite{DaSW24b}, $i.e.$ there exists a compact oriented manifold $N$ with boundary, such that $M$ is the quotient space of $N$ obtained by pinching each component of $\partial N$ into a point, then $H_n(M) = \mathbf{Z}$.
\end{remark}

As a consequence, the following definition is consistent with Definition \ref{defn: degree 2}.
\begin{definition}\label{defn: degree 4}
    Let $M,N,\psi$ be as in Definition \ref{defn: degree 2}. Then $\psi$ is said to have non-zero degree, if the induced map
    \begin{align*}
        \psi_*:H_n(M)\longrightarrow H_n(N)\cong\mathbf{Z}
    \end{align*}
    is not identically zero.
\end{definition}

\section{A topological obstruction result for PSC on noncompact manifolds}
In this appendix, we prove Proposition \ref{prop: noncompact dominate enlargeable}. Related results have been obtained recently in \cite{CL2024}\cite{CCZ23} via $\mu$-bubble methods and in \cite{WZ2022}\cite{LSWZ24} using Dirac operator techniques.

        \begin{figure}
    \centering
    \includegraphics[width = 12cm]{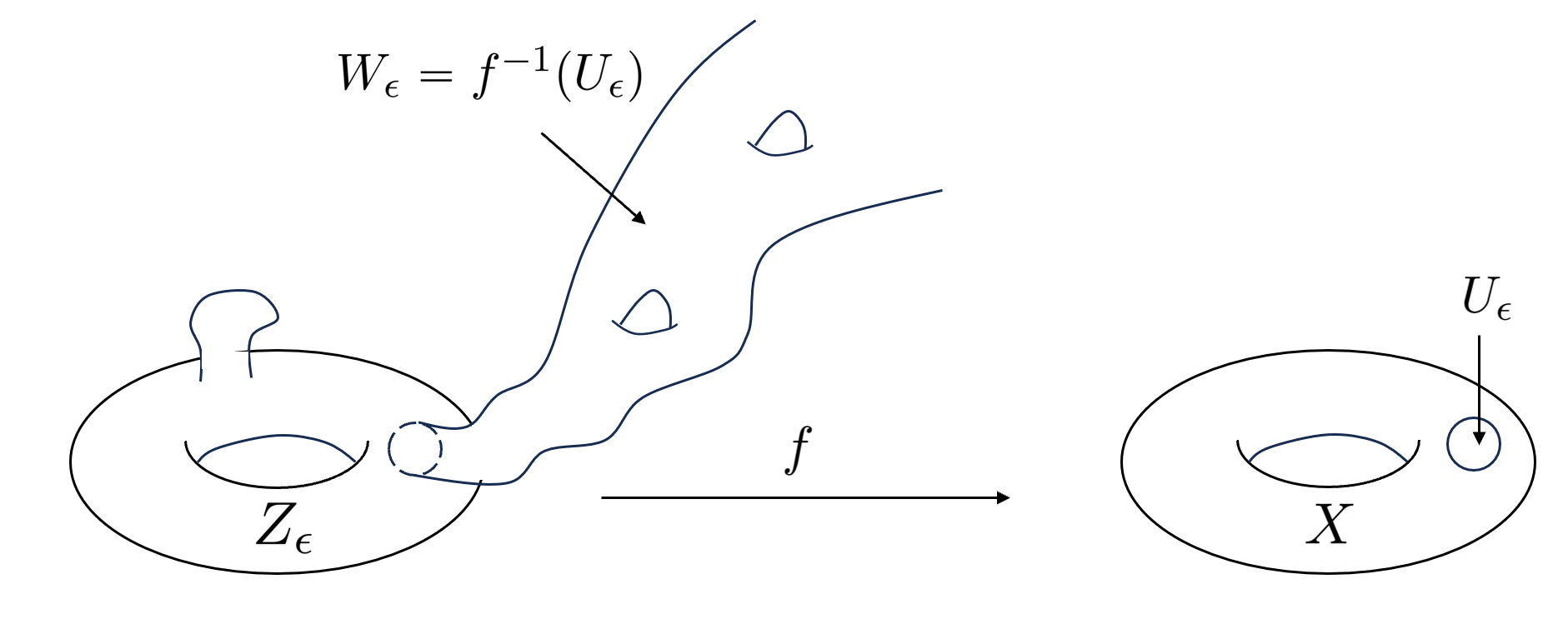}
    \caption{A noncompact manifold that admits a non-zero degree map to an enlargeable manifold}
    \label{f2}
\end{figure}
        
        \begin{proof}[Proof of Proposition \ref{prop: noncompact dominate enlargeable}]

$\quad$

Let us sketch the proof first. We argue by contradiction, by assuming $R_g>0$ in $Y$. The enlargeability of $X$ yields an arbitrarily long over-torical band $\mathcal{V}$ lying in certain covering space $\tilde{X}$. By pulling back the non-zero degree map $f:Y\longrightarrow X$, we obtain another non-zero degree map $\tilde{f}:\tilde{Y}\longrightarrow \tilde{X}$. We construct a subspace $\hat{\mathcal{V}} = \tilde{f}^{-1}(\mathcal{V})$, which can be viewed as a long over-torical band with some arbitrary ends connected to it. See Figure \ref{figure V} for illustration. We show that $\hat{\mathcal{V}}$ is partitioned into the union of $\Omega_-$ and $\Omega_+$ by a \textit{deep} hypersurface $\hat{S}$, which is far from all boundary components of $\hat{\mathcal{V}}$. The hypersurface $\hat{S}$ has the following key property: any compact hypersurface in $\hat{\mathcal{V}}$ homologous to $\hat{S}$ admits no PSC metric. This enables us to construct a $\mu$-bubble functional as in Step 4 and derive a contradiction.

            \textbf{Step 1: The selection of an overtorical band $\mathcal{V}$ in an covering space $\tilde{X}$.} 
            
            Let $S_{\infty} = \{p_1,p_2,\dots,p_m\}$. Fix a metric $h$ on $X^n$, and denote $U_{\epsilon}$ to be the $\epsilon$-neighborhood of $S_{\infty}$, so $U_\epsilon$ is the union of $m$ disjoint small balls. We select $\epsilon$ such that
            \begin{align*}
                \epsilon<\frac{1}{5}\min\{\min_{i\ne j}d_{(X,h)}(p_i,p_j),\inj_{(X,h)}\}.
            \end{align*}
            Then for any Riemannian covering $q_X:(\tilde{X},\tilde{h})\longrightarrow (X,h)$, each component of $q_X^{-1}(U_{2\epsilon})$ maps homeomorphically to a component of $U_{2\epsilon}$ under $q_X$. Therefore, for any $x_1,x_2\in\tilde{X}$ lying in different connected components of $q_X^{-1}(U_\epsilon)$, it holds
            \begin{align}\label{eq: 46}
                d(x_1,x_2)\ge 2\epsilon.
            \end{align}
            Similarly, $q_X^{-1}(U_\epsilon)$ is also disjoint union of small balls.

             Since $X$ is enlargeable, from \cite{Gro18}, for any $D>0$, there is a Riemannian covering $q_X:(\tilde{X},\tilde{h})\longrightarrow (X,h)$ with an embedded over-torical band $ {\mathcal{V}}\hookrightarrow\tilde{X}$ that satisfies $\width(\mathcal{V})>2D+1$. We use $\psi: \mathcal{V}\longrightarrow \mathbf{T}^{n-1}\times [-1,1]$ to be the non-zero degree map, and use $\pi: \mathbf{T}^{n-1}\times [-1,1]\longrightarrow \mathbf{T}^{n-1}$ to be the projection map. From \cite[Lemma 4.1]{Zhu21} there is a proper smooth function $\rho:\mathcal{V}\longrightarrow [-D,D]$ that satisfies $\Lip\rho<1$, $\rho|_{\partial_{\pm}\mathcal{V}} = \pm D$. Note that $q_X^{-1}(S_\infty)\subset \tilde{X}$ is a discrete set, we can therefore assume $0$ is a regular value of $\rho$ and $0$ is not contained in the image of $ q^{-1}_X(S_{\infty})$ under $\rho$ without loss of generality. Thus, $S = \rho^{-1}(0)$ is a closed hypersurface separating two distinguished boundary collection $\partial_-\mathcal{V}$ and $\partial_+\mathcal{V}$ of $\mathcal{V}$, such that
             \begin{align}\label{eq: C.1}
                 S\cap q^{-1}_X(S_{\infty}) = \emptyset.
             \end{align}
             
              From a standard differential topology argument, the map
            \begin{align*}
                \pi\circ(\psi|_S): S\longrightarrow \mathbf{T}^{n-1}
            \end{align*}
            has non-zero degree. In fact, the degree of the map
            \begin{align*}
                \pi\circ(\psi|_{\partial_-\mathcal{V}}): \partial_-\mathcal{V}\longrightarrow \mathbf{T}^{n-1}
            \end{align*}
            equals that of $\degg\psi$. Because $S = \rho^{-1}(0)$ and $\partial_-\mathcal{V} = \rho^{-1}(-D)$, we know that $S$ and $\partial_-\mathcal{V}$ bound a region, so it follows that
            \begin{align}\label{eq: 59}
                \degg \pi\circ(\psi|_S) = \degg \pi\circ(\psi|_{\partial_-\mathcal{V}}) \ne 0.
            \end{align}

\textbf{Step 2: The pullback construction.}

            Next, we obtain the following diagram  which is commutative, by considering pullback objects repeatedly from the right of the diagram to the left:

            \begin{equation*}
\xymatrix{\hat{S}\ar[r]\ar[d]^{\tilde{f}}&\hat{\mathcal{V}}\ar[r]\ar[d]^{\tilde{f}}&\tilde{Y}\ar[r]^{q_Y}\ar[d]^{\tilde{f}}&Y\ar[d]^f\\
S\ar[r]&\mathcal{V}\ar[r]\ar[d]^{\psi}&\tilde{X}\ar[r]^{q_X}&X\\
&\mathbf{T}^{n-1}\times[-1,1]\ar[r]^{\pi}&\mathbf{T}^{n-1}}
\end{equation*}
In the diagram:
\begin{itemize}
    \item Each unlabeled horizontal arrow is given by an inclusion map;
    \item $q_Y:(\tilde{Y},\tilde{g})\longrightarrow (Y,g)$ is a Riemannian covering, and $\tilde{f}$ is also a quasi-proper map;
    \item  $\hat{\mathcal{V}}$ is  an embedded band in $\tilde{Y}$ (possibly to be noncompact).
\end{itemize}
By sligtly perturbing $\tilde{f}$, we may assume $\tilde{f}$ is transversal to $\partial_{\pm}\mathcal{V}$ and $S$, so $\partial_{\pm}\hat{\mathcal{V}}$ and $\hat{S}$ are regular hypersurfaces in $\tilde{Y}$. Furthermore, \eqref{eq: C.1} and the quasi-properness of $\tilde{f}$ ensures $\hat{S}$ is compact. Denote $W_{\epsilon} = f^{-1}(U_{\epsilon})$. Since $f$ is quasi-proper, we know that $Z_{\epsilon} = f^{-1}(X\backslash U_{\epsilon})$ is compact. See Figure \ref{f2} for illustration. We would like to call $Z_\epsilon$ the \textit{uniform part} and $W_\epsilon$ the \textit{end part}. The compactness of $Z_\epsilon$ yields a constant $C$, such that
        \begin{align}\label{eq: 44}
            \Lip(f|_{Z_{\epsilon}})<C.
        \end{align}
        Denote $\tilde{W}_{\epsilon} = q_Y^{-1}(W_{\epsilon}) = \tilde{f}^{-1}(q_X^{-1}(U_\epsilon))$ and $\tilde{Z}_{\epsilon} = q_Y^{-1}(Z_{\epsilon}) = \tilde{f}^{-1}(q_X^{-1}(U_\epsilon))$, which we also interpret as the \textit{uniform part} and the \textit{end part}. It follows from \eqref{eq: 44} that
        \begin{align}\label{eq: 45}
            \Lip(\tilde{f}|_{\tilde{Z}_{\epsilon}})<C.
        \end{align}

\textbf{Step 3: The depth estimate of $\hat{S}$}

        We show $\hat{S}$ is a {\it deep} hypersurface in $\hat{\mathcal{V}}$ in the following sense:
        \begin{align}\label{eq: 47}
            d(\partial_{\pm}(\hat{\mathcal{V}}),\hat{S})\ge\frac{D}{2C}.
        \end{align}
        To prove \eqref{eq: 47}, let $\gamma$ be a curve in $\hat{\mathcal{V}}$ joining $\partial_-\hat{\mathcal{V}}$ and $\hat{S}$. We represent $\gamma$ in a piecewise manner: $\gamma = \xi_1\zeta_1\xi_2\zeta_2\dots\xi_{m-1}\zeta_{m-1}\xi_m$, where $\xi_i\subset \tilde{Z}_{\epsilon}$ and $\zeta_i\subset \tilde{W}_{\epsilon}$. Denote $\bar{\gamma} = \tilde{f}\circ \gamma$, $\bar{\xi}_i = \tilde{f}\circ \xi_i$ and $\bar{\zeta}_i = \tilde{f}\circ \zeta_i$. Since $\tilde{f}$ is continuous, each segment $\zeta_i$ is mapped to a single component of $q_X^{-1}(U_\epsilon)$, which we denote by $\mathcal{O}_i$

        Since each $\mathcal{O}_i$ has diameter less than $2\epsilon$, we can construct a curve $\bar{\gamma}'$, which coincides with $\gamma$ on each $\bar{\xi}_i$ segment. The $\bar{\zeta}_i$ segment is replaced by a minimizing curve $\bar{\zeta}_i'$ in $\mathcal{O}_i$ which share common endpoints with $\bar{\zeta}_i$. Consequently, $\bar{\gamma}' = \bar{\xi}_1\bar{\zeta}_1'\bar{\xi}_2\bar{\zeta}_2'\dots\bar{\xi}_m$ is a curve connecting $\partial_-\mathcal{V}$ and $S$. Since $d(\partial_{\pm}({\mathcal{V}}),S)\ge D$, we obtain
        \begin{align}\label{eq: 57}
            D<L(\bar{\gamma}') = \sum_{i=1}^m L(\bar{\xi}_i)+ \sum_{i=1}^m L(\bar{\zeta}_i')\le \sum_{i=1}^m L(\bar{\xi}_i)+2(m-1)\epsilon\le 2\sum_{i=1}^m L(\bar{\xi}_i),
        \end{align}
where in the second inequality we have used $\diam(\mathcal{O}_i)\le 2\epsilon$, and in the third inequality we have used the distance between $\mathcal{O}_i$ and $\mathcal{O}_{i+1}$ is at least $2\epsilon$ (as a consequence of \eqref{eq: 46}). On the other hand, \eqref{eq: 45} implies
        \begin{align}\label{eq: 58}
            \sum_{i=1}^m L(\bar{\xi}_i)\le\frac{1}{C}\sum_{i=1}^m L({\xi}_i)\le \frac{1}{C}L(\gamma).
        \end{align}
       Combining  \eqref{eq: 57} with \eqref{eq: 58} yields \eqref{eq: 47}.

        \begin{figure}
    \centering
    \includegraphics[width = 12cm]{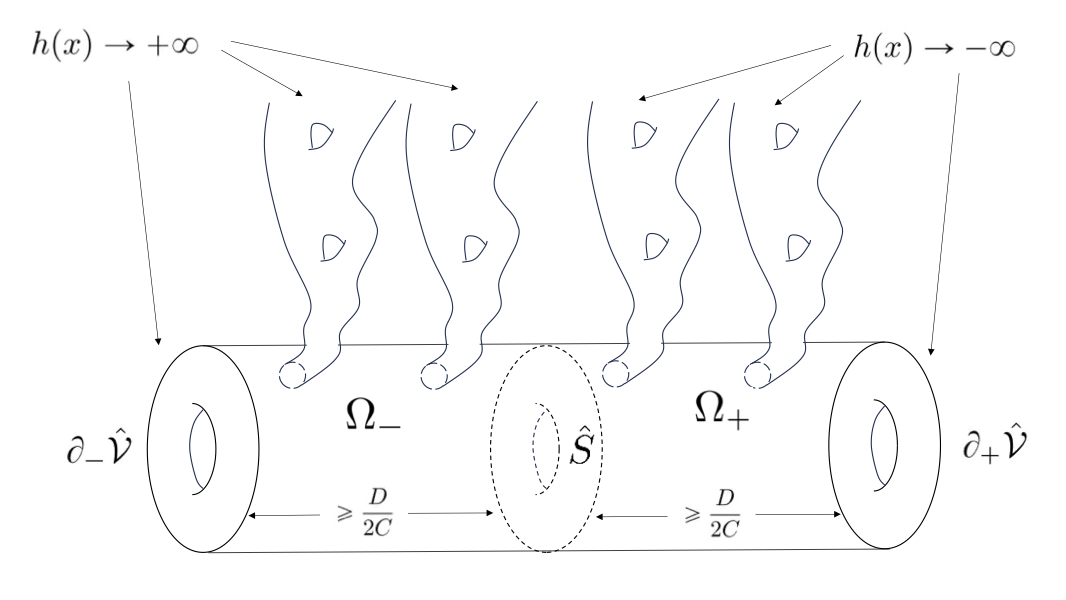}
    \caption{The illustration of the band $\hat{\mathcal{V}}$ obtained from the pullback construction}
    \label{figure V}
\end{figure}

\textbf{Step 4: The $\mu$-bubble construction}

Define $\Omega_- = (\rho\circ\tilde{f})^{-1}((-D,0))$ and $\Omega_+ = (\rho\circ\tilde{f})^{-1}((0,D))$. Then $\hat{\mathcal{V}} $ is sepatated into $\Omega_-\cup\Omega_+$ by $\hat{S}$. Let
\begin{align*}
    \mathcal{G} = \{ &\Sigma\subset \hat{\mathcal{V}}, \Sigma = \partial\Omega\cap\mathring{\hat{\mathcal{V}}},\\
    &\Omega\subset \mathcal{V}\mbox{ is a region with smooth boundary and }\Omega\Delta\Omega_0 \mbox{ is compact }.\}
\end{align*}
For each $\Sigma\subset\mathcal{G}$, \eqref{eq: 59} implies
\begin{align*}
    \degg (\pi\circ\psi\circ\tilde{f}|_{\Sigma}) = \degg (\pi\circ\psi\circ\tilde{f}|_{\hat{S}}) = \degg f\degg (\pi\circ\psi|_S) \ne 0.
\end{align*}
In conjunction with \cite{SY79b}, $\Sigma$ admits no PSC metric.

Now, we are in the position to use the $\mu$-bubble technique to derive a contradiction. We have another decomposition $\hat{\mathcal{V}} = \mathcal{W}\cup\mathcal{Z}$, where $\mathcal{W} = \tilde{W}_\epsilon\cap \mathcal{\hat{V}}$ is the \textit{end part} and $\mathcal{Z} = \tilde{Z}_\epsilon\cap \mathcal{\hat{V}}$ is the \textit{uniform part}. Since $(Y,g)$ is PSC, we have $R_{\tilde{g}}\ge R_0>0$ for a uniform number $R_0$ on $\mathcal{Z}$. Now our space is very similar to that appeared in \cite[Section 7]{CL2024}: The uniform part $\mathcal{Z}$, which is induced from the over-torical band $\mathcal{V}$ corresponds to $\mathbf{T}^{n-1}\times \mathbf{R}$ in \cite{CL2024} while the end part $\mathcal{W}$ corresponds to arbitrary ends $X_k$'s in \cite{CL2024}. Moreover, by mollifying the signed distance function to $\hat{S}$, we get a distance like function $\rho_1$ which corresponds to the function $\rho_1$ defined by \cite[Section 7]{CL2024}. We have $\rho_1<0$ in $\Omega_-$ and $\rho_1>0$ in $\Omega_+$. When $D$ is sufficiently large, we can adopt the construction of \cite{CL2024} to find a function $h(x)$ (which behaves like $-\tan C_1\rho_1(x)$ in $\mathcal{Z}$ and $\frac{C_2}{\pm \rho_1(x)-C_3}$ in each components of $\mathcal{W}$, here $C_1,C_2$ and $C_3$ are constants), such that
\begin{itemize}
    \item $h(x)\to +\infty$ when $x\to \partial_-\hat{\mathcal{V}}$ and the infinity of $\mathcal{W}\cap\Omega_-$;
    \item $h(x)\to -\infty$ when $x\to \partial_+\hat{\mathcal{V}}$ and the infinity of $\mathcal{W}\cap\Omega_+$;
    \item $R_{\tilde{g}}+h^2-2|\nabla h|>0$
\end{itemize}
See Figure \ref{figure V}. Setting a $\mu$-bubble functional using $h$ as in \cite{CL2024}, we get a hypersurface $\Sigma\in\mathcal{G}$ with a PSC metric, a contradiction.
\end{proof}

\section{Decay estimates for solutions to elliptic equations on AF ends}
In this appendix, we recall some results about the decay estimates for solutions to elliptic equations on AF ends.

\begin{lemma}\label{lem: asymptotic behavior1}
		Let $(E,g)$ be a $n$-dimensional AF end with asymptotic order $\tau\le n-2$. Let $f\in C^{2,\alpha}(E)$ for some $\alpha>0$. Suppose $f$ satisfies
		$$
		f(x)=O(r^{-\tau-2}) \ \quad \text{ for $x\rightarrow \infty$ on $E$},
		$$
		where $r:=|x|$, $u\in C^{4,\alpha}(E)$ satisfies
		$$
		-\Delta u=f \quad \text{in }\ \ E 
		$$
		with $u(x)\rightarrow 0$ for $x\rightarrow \infty$ on $E$. Then for each $\epsilon>0$, there holds
		
		$$
		|u|+r|\nabla u|+r^2|\nabla^2 u| = O(r^{-\tau+\epsilon}).
		$$
	\end{lemma}
	
	\begin{proof}
		Without loss of generality, we can assume $E$ is diffeomorphic to $\mathbf{R}^n$, then $u$ satisfies the equation
		\begin{align*}
			g^{ij}(D_{ij}u-\Gamma_{ij}^k D_ku) = 0 \ \text{ on}\quad \mathbf{R}^{n}.
		\end{align*}
		Since $\Gamma_{ij}^k = O(r^{-\tau-1})$, by Schauder estimate we know $u = O(r^{-1})$. Therefore, we may assume 
		$$
		-\bar{\Delta} u=\bar{f} \ \text{ on} \quad \mathbf{R}^{n}
		$$
		for some   $\bar{f}\in C^{2,\alpha}(\mathbf{R}^{n})$ with
        $f(x)=O(r^{-\tau-2})$ for $x\rightarrow \infty$ on $\mathbf{R}^{n}$,
        where $\bar{\Delta}$ denotes the standard Laplace operator on $\mathbf{R}^{n}$ .
        The conclusion then follows from the Claim in \cite[p.327]{Lee19}.
	\end{proof}

\begin{lemma}\label{lem: asymptotic behavior2}
    Let $(E,g)$ be a $n$-dimensional AF end with asymptotic order $\frac{n-2}{2}<\tau\le n-2$. $u\in C^{4,\alpha}(E)$ satisfies
		\begin{equation}\label{eq: harmonic function}
		-\Delta u=0 \quad \text{in}\ \ \  E
		\end{equation}
    with $u(x)\to 0$ as $|x|\to\infty$. Then there is a constant $a$, such that for any $\varepsilon'>0$ there holds
    \begin{align*}
        u(x) = ar^{2-n}+O(r^{1-n})+O(r^{-\tau+2-n+\varepsilon'}),
    \end{align*}
    where $r=|x|$.
\end{lemma}

\begin{proof}
    By Lemma \ref{lem: asymptotic behavior1}, 
    for any $\epsilon>0$,
    we have $|u(x)|+r|\nabla u|+r^2
    |\nabla^2 u|=O(r^{-\tau+\varepsilon})$. Rewrite \eqref{eq: harmonic function} as
    \begin{equation}\label{eq: possion eq}
        \bar{\Delta}u=f\ \ \text{on}
        \ \ \mathbf{R}^n\setminus B^n_R(O),
    \end{equation}
    where $\bar{\Delta}$ is the Laplacian in 
    $\mathbf{R}^n$. Then we have $f=O(r^{-2-2\tau+\varepsilon})$.
By standard PDE theory,  given any $\varepsilon'>0$, we can find a solution $v$ of the problem
\begin{equation}\label{eq: laplacian of v}
        \bar{\Delta}v=f\ \ \text{on}
        \ \ \mathbf{R}^n\setminus B^n_R(O),
    \end{equation}
satisfying
\begin{equation}\label{eq: asypmtotic expansion1}
 |v|+r|\partial v|+r^2|\partial ^2 v|=O(r^{-2\tau+\varepsilon+\varepsilon'}).   
\end{equation}
Now $w =u-v$ is a harmonic function defined in $\mathbf{R}^n\setminus B^n_R(O)$ with $w(x)\rightarrow 0$ as $x\rightarrow\infty$.
Thus
\begin{equation}\label{eq: asypmtotic expansion2}
    w(x)=ar^{2-n}+O(r^{1-n}).
\end{equation} 
Note that $\tau>\frac{n-2}{2}$. By \eqref{eq: asypmtotic expansion1}
and \eqref{eq: asypmtotic expansion2} we can improve the term $f$ in \eqref{eq: possion eq}
by $f=O(r^{-\tau-n+\varepsilon+\varepsilon'})$. Repeating the argument before, we find a solution $v$ of \eqref{eq: laplacian of v} with
\begin{equation}\label{eq: asypmtotic expansion3}
  |v|+r|\partial v|+r^2|\partial ^2 v|=O(r^{-\tau+2-n+\varepsilon'}).  
\end{equation}
for any $\varepsilon'>0$.
Combining \eqref{eq: asypmtotic expansion2}
and \eqref{eq: asypmtotic expansion3}
gives the desired estimate.

\end{proof}

\bibliographystyle{alpha}

\bibliography{Positive}

\end{document}